%% file: main.tex
\documentclass{article}
\PassOptionsToPackage{numbers}{natbib}

\usepackage[final]{neurips_2024}
\usepackage{amsmath}
\usepackage{amsthm}
\usepackage[utf8]{inputenc} 
\usepackage[T1]{fontenc}    
\usepackage{hyperref}       
\usepackage{url}            
\usepackage{booktabs}       
\usepackage{amsfonts}       
\usepackage{nicefrac}       
\usepackage{microtype}      
\usepackage{xcolor}         
\usepackage{float}

\usepackage{derivative}
\usepackage{cleveref}
\usepackage{graphicx}
\usepackage{subcaption}
\graphicspath{{figures/}}

\theoremstyle{plain} 
\newtheorem{theorem}{Theorem}[section]
\newtheorem{lemma}[theorem]{Lemma}

\theoremstyle{definition} 
\newtheorem{definition}[theorem]{Definition} 
\newtheorem{assumption}[theorem]{Assumption}
\newtheorem{example}{Example}[section]
\Crefname{assumption}{Assumption}{Assumptions}

\theoremstyle{remark} 
\newtheorem{remark}[theorem]{Remark}

\numberwithin{theorem}{section}

\usepackage{amssymb}
\input{macros}

\title{Estimating Generalization Performance Along the Trajectory of Proximal SGD in Robust Regression
}
\author{%
  Kai Tan\\
  Department of Statistics\\
  Rutgers University\\
  Piscataway, NJ 08854 \\
  \texttt{kai.tan@rutgers.edu} \\
  \And
  Pierre C. Bellec \\
  Department of Statistics\\
  Rutgers University\\
  Piscataway, NJ 08854 \\
  \texttt{pierre.bellec@rutgers.edu} \\
}

\begin{document}

\maketitle

\begin{abstract}
	This paper studies the generalization performance of iterates obtained by Gradient Descent (GD), Stochastic Gradient Descent (SGD) and their proximal variants in high-dimensional robust regression problems. The number of features is comparable to the sample size and errors may be heavy-tailed. We introduce estimators that precisely track the generalization error of the iterates along the trajectory of the iterative algorithm. These estimators are provably consistent under suitable conditions. The results are illustrated through several examples, including Huber regression, pseudo-Huber regression, and their penalized variants with non-smooth regularizer. We provide explicit generalization error estimates for iterates generated from GD and SGD, or from proximal SGD in the presence of a non-smooth regularizer. The proposed risk estimates serve as effective proxies for the actual generalization error, allowing us to determine the optimal stopping iteration that minimizes the generalization error. Extensive simulations confirm the effectiveness of the proposed generalization error estimates.
  \end{abstract}

\section{Introduction} \label{sec:intro}
Consider the linear model: 
\begin{equation}\label{eq:linear-model}
	\by = \bX \bb^* + \bep,
\end{equation}
where $\by\in\R^n$ is the response vector, $\bX\in\R^{n\times p}$ is the design matrix, $\bb^*\in\R^p$ is the unknown regression vector, and $\bep\in \R^n$ is the noise vector that we assume independent of $\bX$. 
The entries of $\bep$ may be heavy-tailed, for instance our working assumptions allow for infinite second moment.

For the estimation of $\bb^*$, we 
consider the following regularized optimization problem
\begin{equation}\label{eq:hat-b}
	\hbb 
	\in \argmin_{\bb\in \R^p}
	\frac1n \sum_{i=1}^n \rho(y_i - \bx_i^\top \bb) + g(\bb), 
\end{equation}
where $\rho:\R \to \R$ is a data-fitting loss and $g: \R^p \to \R$ is a regularization function.
In the present robust regression setting, typical examples of $\rho$ include 
the Huber \cite{huber2004robust} loss $\rho(r;\delta) = \delta^2 \int_{0}^{|r/\delta|} \min(1, x) \, \text{d}x$, the Pseudo-Huber loss $\rho(r;\delta) = \delta^2 (\sqrt{1 + (r/\delta)^2} - 1)$ or other Lipschitz loss functions to combat the possible
heavy-tails of the additive noise.
Typical examples of penalty functions include the L1/Lasso
\cite{tibshirani1996regression}
penalty
$g(\bb) = \lambda \norm{\bb}_1$, group-Lasso penalty \cite{yuan2006model} for grouped variables,
or their non-convex variants including for instance SCAD \cite{fan2001variable} or MCP
\cite{zhang10-mc+}.

In order to solve the optimization problem \eqref{eq:hat-b}, 
practitioners resort to
iterative algorithms, for instance gradient descent, accelerated gradient descent, stochastic gradient descent, and the corresponding proximal methods
\cite{parikh2014proximal}
in the presence of a non-smooth regularizer.
Let the algorithm starts with some initializer $\hbb^1\in \R^p$ (typically $\hbb^1 = \boldzero$) followed by consecutive iterates $\hbb^2, \hbb^3,\dots$, 
where
$\hbb^t$ is typically obtained, for gradient descent and its variants
as will be detailed below, 
from $\hbb^{t-1}$ and by applying an additive correction
involving the gradient of the objective function.
Our goal of this paper is to quantify the predictive performance of each iterate $\hbb^t$. 

We assume throughout that the covariance $\E[\bx_i \bx_i^\top]=\bSigma$
of the feature vectors is finite. We measure the predictive
performance of $\hbb^t$ using the out-of-sample error
$$
\E\Bigl[
    \Bigl(\bx_{new}^\top\hbb^t - \bx_{new}^\top \bb^*\Bigr)^2
\mid (\bx_i,y_i)_{i\in[n]}
\Bigr]
=
\|\bSigma^{1/2}(\hbb^t-\bb^*)\|^2
$$
where $\bx_{new}$ is a new feature vector, independent of the data
$(\bx_i,y_i)_{i\in[n]}$
and has the same distribution as $\bx_i$.
The above squared metric is used because the noise $\ep_i$ (and thus $y_i$)
is allowed to have infinite variance, and in this case
the squared prediction error
$\E[
    (\bx_{new}^\top\hbb^t - y_{new})^2
\mid (\bx_i,y_i)_{i\in[n]}
]=+\infty$ irrespective of the value of $\hbb^t$. 

The paper proposes to estimate the out-of-sample error
$\|\bSigma^{1/2}(\hbb^t-\bb^*)\|^2$ of the $t$-th iterate
using the right-hand side of the approximation
\begin{align}
	\norm{\bSigma^{1/2}(\hbb^t - \bb^*)}^2
	+ \norm{\bep}^2/n
	\approx 
	\Big\|{(\by-\bX\hbb^t) + \sum_{s=1}^{t-1} w_{t,s} \bS_s \bpsi(\by-\bX\hbb^s)} \Big\|^2 /n,
        \label{approx}
\end{align}
where $\bpsi:\R^n \to \R$ is the derivative of $\rho$ acting component-wise on each coordinate in $\R^n$
and $\bS_s$ is a diagonal matrix of the form
$\bS_s = \sum_{i\in I_s} \be_i\be_i^\top$ where $I_s\subset [n]$
is the batch for the $s$-th stochastic gradient update
and $\be_i\in\R^n$ is the $i$-th canonical basis vector.
Here the $w_{t,s}$ are quantities, introduced in \Cref{sec:main} below, that can be computed from data and do not require the knowledge of $\bSigma$. 
The approximation \eqref{approx} is made rigorous in \Cref{thm:using-Sigma},
where the right-hand side is proved to be consistent (i.e., the difference
between the two sides of the inequality converges to 0 in probability)
for a first set of weights $(w_{s,t})_{s<t}$,
and in \Cref{thm:unknwon-Sigma} where a second set of weights are proposed.

Because the right-hand side of \eqref{approx} is observable from the data
and the iterates $(\hbb^s)_{s\le t}$ are computed from the iterative algorithm,
the approximation \eqref{approx} lets us compare the out-of-sample error
of iterates $\hbb^t$ at different time $t$ up to the additive term
$\|\bep\|^2/n$ (which does not depend on $t$ nor on the choice of
the iterative scheme or the choice of loss and penalty).
It also lets us compare different tuning parameters, for instance
learning rate, multiplicative parameter of the penalty function,
batch size in Stochastic Gradient Descent (SGD).
The right-hand side of \eqref{approx} can serve as the criteria to choose the
iteration number or tuning parameters that achieves the smallest out-of-sample error.

\subsection{Related literature}
Estimation of prediction risk of regression estimates has received
significant attention in the last few decades.
One natural avenue to estimate
the generalization performance is to use $V$-fold cross-validation or leave-one-out schemes.
In the proportional regime of interest here, where dimension $p$ and
sample size $n$ are of the same order,
$V$-fold cross-validation with finite $V$, e.g., $V=5,10$ is known to fail
at consistently estimate the risk of the estimator trained on the full dataset
\cite[Figure 1]{rad2018scalable}; this is simply explained because 
training with the biased sample size $n(V-1)/V$ may behave differently
than training with the full dataset.
Leave-one-out schemes, or drastically increasing $V$,
requires numerous refitting and is thus computationally expensive. 

This motivates computationally efficient estimates of the risk
of an estimator trained on the full dataset without sample-splitting,
including Approximate Leave-One-out (ALO) schemes
\cite{rad2018scalable} that do not rely on sample-splitting and refitting;
see \cite{auddy2024approximate} and the references therein for recent
developments.
For ridge regression and other estimators constructed from the square loss,
the Generalized Cross-Validation (GCV)
\cite{wahba1985comparison} has been shown be to be effective, and it avoids data-splitting and refitting; it only needs to fit the full data once and then adjust the training error by a multiplicative factor larger than 1. Beyond ridge regression, the extension of GCV using degrees-of-freedom has been studied for Lasso regression \cite{bayati2013estimating,bellec2020out,miolane2018distribution,celentano2020lasso}, and alternatives were developed for robust M-estimators
\cite{bellec2020out,bellec2021derivatives}.
While ALO or GCV and its extensions are good estimators of the predictive risk of a solution $\hbb$ to the optimization problem \eqref{eq:hat-b}, they are not readily applicable to quantify the prediction risk of iterates $\hbb^t$ obtained by widely-used iterative algorithms such as gradient descent (GD),
stochastic gradient descent (SGD) or their proximal variants:  ALO and GCV focus on estimating 
the final $(t\to+\infty)$
iterate of the algorithm, when a solution $\hbb$ in \eqref{eq:hat-b}
is found. Our goal in the present paper is to develop risk estimation
methodologies along the trajectory of the algorithm.

\citet{luo2023iterative} developed methods to estimate the cross-validation error of iterates that solves an empirical risk minimization problem. Their approach requires the Hessian of the objective function to be well-conditioned (i.e., the smallest and largest eigenvalues are bounded) along all iterates. This condition is not satisfied for the regression problems we consider in this paper, such as high-dimensional robust regression with a Lasso penalty.
In the context of least squares problems with both \( p \) and \( n \) being large,
\cite{celentano2020estimation} studied the fundamental limits on the
performance of first order methods, showing that these
are dominated by a specific Approximate Message Passing
algorithm.
\citet{paquette2021sgd} demonstrated that the dynamics of Stochastic Gradient Descent (SGD) become deterministic in the large sample and dimensional limit, providing explicit expressions for these dynamics when the design matrix is isotropic. Our work differs from \cite{paquette2021sgd} in two key ways: First, we address a more general regression problem incorporating a non-smooth regularizer, thereby considering both SGD and proximal SGD; second, we offer explicit risk estimates for each iteration, rather than focusing solely on the theoretical dynamics of the iterates.
\citet{celentano2021high} and
\cite{gerbelot2022rigorous} characterize
the dynamical mean-field dynamics of iterative schemes,
and identify that the limiting process involves a ``memory'' kernel,
describing how the dynamics of early iterates affect later ones.

Most recently, and most closely related to the present paper,
\cite{bellec2024uncertainty} proposed risk estimate for iterates $\hbb^t$ obtained by running gradient descent and proximal gradient descent methods for solving penalized least squares optimizations. 
However, \cite{bellec2024uncertainty} focuses exclusively on the square
loss for $\rho$ in \eqref{approx}, which is not readily applicable to robust regression with heavy tailed noise for which the Huber or other robust losses
must be used. \citet{bellec2024uncertainty} is further restricted to
gradient updates using the full dataset, which does not cover 
stochastic gradient descent. 
While several proof techniques used 
in the present paper are inspired by \cite{bellec2024uncertainty},
we will explain in \Cref{remark:comparison} that directly generalizing
\cite{bellec2024uncertainty} to SGD in robust regression leads to
a poor risk estimate for small batch sizes. The proposal of the present
paper leverages out of batch samples to overcome this issue.

For gradient descent for the square loss and without penalty,
\citet{patil2024failures} demonstrates both the failure 
of GCV along the trajectory and the success of computationally expensive
leave-one out schemes, and develops a proposal to reduce the computational
cost.  Finally, let us mention the works
\cite{chandrasekher2021sharp,lou2024hyperparameter} that characterize
the dynamics of the iterates in
phase retrieval and matrix sensing problems, assuming that a fresh
batch of observations (independent of all previous updates) is used
at each iteration. This is different from the  usual SGD setting studied in
the present
paper where the observations used during a stochastic gradient update
may be reused in future stochastic gradient updates, creating
intricate probabilistic dependence between gradient updates at different
iterations.

Robust regression is highly valuable in real data analysis due to its ability to handle heavy-tailed noise effectively, and we will see below that
the use of stochastic gradient updates and data-fitting loss functions
different from the square loss require estimates that have a drastically
different structure that in the square loss case.
The present paper develops generalization error estimates
in situations where no consistent estimate have been proposed:
(1) we develop generalization error estimates along the trajectory
of iterative algorithms aimed at solving \eqref{eq:hat-b}
for robust loss functions including the pseudo-Huber loss;
(2) the estimates are applicable not only to gradient updates involving the full dataset (gradient descent and its variants), but also to SGD and proximal SGD where a random batch is used for each update.

\section{Problem setup}
The paper studies iterative algorithms aimed at solving the optimization
problem \eqref{eq:hat-b}.
We consider the algorithm that generates iterates $\hbb^t$ for $t=1,2,...,T$ according to the following iteration:
\begin{equation}\label{eq:unified}
	\hbb^{t+1} 
        = \bphi_{t} \Bigl( \hbb^t + \frac{\eta_t}{|I_t|} \bX^\top \bS_t\bpsi(\by-\bX\hbb^t)
	\Bigr),
\end{equation}
where $\bS_t\in \R^{n\times n}$ is the diagonal matrix
$\bS_t = \sum_{i\in I_t} \be_i \be_i^\top$ for $I_t\subset[n]$
the $t$-th batch (independent of $(\bX,\by)$),
where
$\bphi_t: \R^{p} \mapsto \R^p$ and $\bpsi: \R^{n} \mapsto \R^n$ are two functions and $\eta_t$ is the step size. 
Typically, $\bpsi:\R^n\to\R^n$ is the componentwise application of
$\rho'$ (where $\rho$ is the data-fitting loss in \eqref{eq:hat-b}),
and the matrix
$\bS_t\in\R^{n\times n}$ is diagonal with elements in $\{0,1\}$
encoding the observations $i\in[n]$ used in the $t$-th stochastic gradient
update.
The presence of $\bS_t$ and possibly nonlinear function $\bpsi$ is such that
the above iteration scheme is not covered by previous related works
including \cite{bellec2024uncertainty,patil2024failures}, which only
tackle $\bS_t=\bI_n$ (full batch gradient updates) and $\bpsi:\R^n\to\R^n$
the identity map
($\rho$ in \eqref{eq:hat-b} restricted to be the square loss).
The iterative scheme \eqref{eq:unified}, on the other hand, covers SGD
with robust loss functions.

In the next section, we first provide a few examples 
of algorithms encompassed in the general iteration \eqref{eq:unified}.
This includes Gradient Descent (GD), Stochastic Gradient Descent (SGD), and their corresponding proximal methods \cite{parikh2014proximal}, Proximal GD and Proximal SGD. 
GD and SGD are widely used in practice, while the proximal methods are particularly useful for solving the optimization problem \eqref{eq:hat-b} with non-smooth regularizers.
\subsection{Robust regression without penalty}
If there are no penalties in \eqref{eq:hat-b}, \ie, $g(\bb)=0$, then 
the minimization problem becomes 
\begin{align*}
	\hbb \in \argmin_{\bb\in \R^p}
	\frac1n \sum_{i=1}^n \rho(y_i - \bx_i^\top \bb). 
\end{align*}
To solve this problem, provided $\rho$ is differentiable, one may use
gradient descent (SGD) and stochastic gradient descent (SGD).

\begin{example}[GD]
	The GD method consists of the following iteration:
	\begin{equation}\label{eq:gd}
		\hbb^{t+1} = \hbb^t + \tfrac{\eta_t}{n} \bX^\top \bpsi(\by - \bX \hbb^t),
	\end{equation}
	where $\bpsi$ is the derivative of $\rho$ acting component-wise on its argument, and $\eta_t$ is the step size (also known as learning rate). For the least squares loss $\rho(x) = x^2/2$, we have $\bpsi(\bu) = \bu$.
\end{example}

\begin{example}[SGD]
	Suppose at $t$-th iteration, we use the batch $I_t\subset [n]$ to compute the gradient, 
\begin{align}\label{eq:sgd}
	\hbb^{t+1} 
	= \hbb^t + 
	\frac{\eta_t}{|I_t|} 
	\sum_{i\in I_t} \bx_i \psi(y_i - \bx_i^\top \hbb^t)
	=\hbb^t + \frac{\eta_t}{|I_t|} \bX^\top \bS_t \bpsi(\by - \bX \hbb^t),
\end{align}
where $\bS_t= \sum_{i\in I_t} \be_i\be_i^\top$ and $\be_i$ is the $i$-th canonical vector in $\R^n$. 
If $I_t = [n]$ for each $t$, then 
$|I_t| = n$ and 
$\bS_t =\bI_n$, hence this SGD method reduces to the GD method in \eqref{eq:gd}.
\end{example}

\subsection{Robust regression with Lasso penalty}
Regularized regression is useful for high-dimensional regression problems where $p$ is larger than $n$.
We consider the Lasso penalty $g(\bb) = \lambda \|\bb\|_1$ to fight for the curse of dimensionality and obtain sparse estimates (our working assumptions, on the other hand, do not assume that the ground truth $\bb^*$ is sparse).
While GD and SGD are not directly applicable to solve the optimization problem \eqref{eq:hat-b} with Lasso penalty due to $\|\cdot\|_1$ lacking differentiability at 0, 
Proximal Gradient Descent (Proximal GD) \cite{parikh2014proximal} and 
Stochastic Proximal Gradient Descent (Proximal SGD) can be used to solve this optimization with Lasso penalty. 
\begin{example}[Proximal GD]
	For $g(\bb) = \lambda\norm{\bb}_1$ in \eqref{eq:hat-b}, the Proximal GD gives the following iterations: 
	\begin{align*}
		\hbb^{t+1} 
		&= \soft_{\lambda\eta_t}\bigl(\hbb^t + \tfrac{\eta_t}{n} \bX^\top \bpsi(\by - \bX \hbb^t)\bigr),
	\end{align*}
	where $\soft_{\theta}(\cdot)$  applies the soft-thresholding $u\mapsto \sign(u)(|u|-\theta)_+$ component-wise. 
\end{example}

\begin{example}[Proximal SGD]
Similar to the Proximal GD, the Proximal SGD consists of the following iterations:
\begin{align*}
	\hbb^{t+1} 
	&= \soft_{\lambda \eta_t}\bigl(\hbb^t + \tfrac{\eta_t}{|I_t|} \bX^\top \bS_t \bpsi(\by - \bX \hbb^t)\bigr).
\end{align*}
\end{example}

Let $\brho':\R^n\to \R^n$ be the function applies the derivative of $\rho:\R\to\R$ to each of its component, \ie, $\brho'(\bu) = (\rho'(u_1), ..., \rho'(u_n))^\top$.
Then the above examples can be summarized in the following table with different definition of $\bpsi$, $\bphi_t$, and $\bS_t$.
\begin{table}[H]
	\caption{Specification of $\bpsi, \bphi_t, \bS_t$ for each algorithm}
	\label{tab:specification}
	\centering
	\begin{tabular}{|l|l|l|l|l|}
	  \toprule
	       & GD & SGD & Proximal GD & Proximal SGD \\
	  \midrule
	  $\bpsi(\bu)$ & $\brho'(\bu)$ & $ \brho'(\bu)$ & $\brho'(\bu)$ & $\brho'(\bu)$ \\
	  $\bphi_t(\bv)$ & $\bv$ & $\bv$ & $\soft_{\lambda\eta_t}(\bv)$ & $\soft_{\lambda\eta_t}(\bv)$ \\
	  $\bS_t$ & $\bI_n$ & $\bS_t$ & $\bI_n$ & $\bS_t$\\
	  \bottomrule
	\end{tabular}
\end{table}

To define the proposed estimators of the generalization error,
we further define the following Jacobian matrices:
\begin{align*}
	\bD_t = \pdv{\bpsi(\br)}{\br}\big|_{\br = \by - \bX\hbb^t} \in \R^{n\times n}, \quad
	\tbD_t = \pdv{\bphi_t(\bv)}{\bv}\big|_{\bv = \hbb^t + \tfrac{\eta_t}{|I_t|} \bX^\top \bS_t \bpsi(\by - \bX\hbb^t)} \in \R^{p\times p}. 
\end{align*}
Then, we have $\tbD_t = \bI_p$ for GD and SGD, and
$\tbD_t = \sum_{j\in \hat S_t} \be_j\be_j^\top$ for Proximal GD and Proximal SGD based on soft-thresholding, where $\hat S_t = \{j\in[p]: \be_j^\top \hbb^{t+1}\ne 0\}$.

\section{Main results}

\begin{assumption}\label{assu:X}
    The design matrix $\bX$ has \iid rows from $\mathsf{N}_p(\bf0, \bSigma)$ for some positive definite matrix $\bSigma$ satisfying 
    $0< \lambda_{\min}(\bSigma)\le 1 \le \lambda_{\max}(\bSigma)$ and 
    $\opnorm{\bSigma} \|\bSigma^{-1}\|_{\rm op} \le \kappa$. 
    We assume
    $\operatorname{Var}[\bx_i^\top\bb^*]\le \delta^2$, that is, the signal of the model \eqref{eq:linear-model} is bounded from above.
\end{assumption}

\begin{assumption}\label{assu:noise}
	The noise $\bep$ is independent of $\bX$ and has \iid entries 
        from a fixed distribution
        independent of $n,p$,
        with $\E[|\varepsilon_i|]\le \delta$, that is, bounded first moment.
	\end{assumption}

\begin{assumption}\label{assu:rho-1}
	The data fitting loss $\rho:\R\to \R$ is convex, continuously differentiable and its derivative $\psi$ is 1-Lipschitz and $|\psi(x)|\le \delta$ for all $x\in \R$.
	The function $\bphi_t$ is 1-Lipschitz and satisfies $\bphi_t(\boldzero) = \boldzero$.
	The matrices $\bS_t = \sum_{i\in I_t}\be_i \be_i^\top$, and $|I_t|\ge c_0 n$ for some positive constant $c_0\in (0, 1]$. 
	Let $\eta_{\max} = \max_{t\in[T]} \eta_t$. 
\end{assumption}
Huber loss and Psuedo-Huber loss all satisfy \Cref{assu:rho-1}.

\begin{assumption}\label{assu:rho-2}
	The data fitting loss $\rho$ is twice continuously differentiable with positive second derivative. 
\end{assumption}

\begin{assumption}\label{assu:regime}
The sample size $n$ and feature dimension $p$ satisfy $p/n\le \gamma$ for a constant $\gamma \in (0, \infty)$.
\end{assumption}

\subsection{Intuition regarding the estimates of the generalization error}
\label{sec:intuition}

This subsection provides the intuition behind the definition of the
estimates define below.
For the sake of clarify, and in this subsection only,
assume that 
\begin{equation}
    \label{simplifications}
\bSigma=\bI_p, \qquad  \bep=\boldzero, \qquad \eta_t/|I_t| = 1/n.
\end{equation}
With the above working assumptions,
the validity of the estimates defined below relies on the 
probabilistic approximation
$$
\|\hbb^t - \bb^*\|^2
\approx
\frac1n\sum_{i=1}^n
\Bigl(-\bx_i^\top (\hbb^t - \bb^*) + \sum_{j=1}^p \be_j^\top \frac{\partial \hbb^t}{\partial x_{ij}}\Bigr)^2,
$$
which was developed in \cite{bellec2020out} for risk estimation
purposes, but outside the context of iterative algorithms.
Above, $\be_j\in\R^p$ is the $j$-th canonical basis vector.
In the present noiseless case with $\bep=\boldzero$, the first term inside the
squared norm in the right-hand side is equal to the residual $y_i-\bx_i^\top \hbb^t$,
so that the above display resembles \eqref{approx}.
Taking this probabilistic approximation for granted, to study the second
term in the right-hand side, we must understand the derivatives of $\hbb^t$
with respect to the entries $(x_{ij})_{i\in [n],j\in [p]}$ of $\bX$.
In \eqref{eq:unified}, each iterate is a relatively simple function of the previous ones, with the simplifications \eqref{simplifications} this is
$\hbb^{t+1}=\bphi_t(\hbb^t + \bX^\top\bS_t\bpsi(\by-\bX\hbb^t)/n)$.
For $t=1$, given that $\hbb^1$ is a constant initialization,
$\frac{\partial}{\partial x_{ij}} \hbb^2 = 
\tbD_1\be_j \be_i^\top \bS_1\bpsi(\by-\bX\hbb^1)/n
-\tbD_1 \bX^\top \bS_1 \bD_1 \be_i (\hbb^1-\bb^*)_j/n$.
We find in the proof, that when summing these quantities
over $j\in[p]$, the second term involving
$(\hbb^1-\bb^*)_j$ is negligible, and the same negligibility holds at later
iterations with terms involving 
$(\hbb^t-\bb^*)_j$ (or any $(\hbb^s-\bb^*)_j$, $s\le t$).
By performing a similar simple calculation at the next iteration,
and ignoring these terms, we find with
$f_i^1 = \be_i^\top \bS_1 \bpsi(\by-\bX\hbb^1)$
and
$f_i^2 = \be_i^\top \bS_2 \bpsi(\by-\bX\hbb^2)$
by the chain rule
$$
\sum_{j=1}^p \be_j^\top \frac{\partial \hbb^2}{\partial x_{ij}}
\approx
\underbrace{\trace\Bigl[\frac{\tbD_1}{n}\Bigr]}_{w_{2,1}} f_i^1,
\quad
\sum_{j=1}^p \be_j^\top\frac{\partial \hbb^3}{\partial x_{ij}}
\approx
\underbrace{\trace\Bigl[\frac{\tbD_2}{n}\Bigr]}_{w_{3,2}} f_i^2
+
\underbrace{
\trace\Bigl[\tbD_2(\bX^\top\bS_2 \bD_2\bX/n - \bI_p)\tbD_1/n\Bigr]
}_{w_{3,1}}  f_i^1.
$$
This reveals the weights $(w_{s,t})_{s<t}$ in \eqref{approx} 
at iteration $t=2$ and $t=3$.
We could continue this further by successive applications of the chain
rule, although for later iterations
this unrolling of the derivatives, capturing the interplay
between the Jacobians $\bD_t,\tbD_t$ and the stochastic gradient
matrix $\bS_t$, becomes increasingly complex.
This recursive unrolling of the derivatives can be performed numerically
at the same time as the computation of the iterates.
On the other hand, for the mathematical proof, for the formal definition of the
weights in \eqref{approx} and for the proposed estimates of the generalization error,
the matrix notation defined in the next subsection exactly
captures this unrolling of the derivatives.

\subsection{Formal matrix notation to capture recursive derivatives}

We now set up the matrix notation that captures this recursive
unrolling of the derivatives by the chain rule.
Throughout, $T$ is the final number of iterations.
Define three block diagonal matrices 
$\calD\in\R^{nT \times nT}$,
$\tcalD\in\R^{pT \times pT}$, and $\calS\in \R^{nT \times nT}$ by
$\calD = \sum_{t=1}^T \bigl((\be_t \be_t^\top) \otimes \bD_t\bigr)$, 
$\tcalD = \sum_{t=1}^{T} \bigl((\be_t \be_t^\top) \otimes \tbD_t\bigr)$, and 
$\calS = \sum_{t=1}^T \bigl((\be_t \be_t^\top) \otimes \bS_t\bigr)$. 
Now we are ready to introduce the following matrices of size $T\times T$:
\begin{align}
	\label{eq:W}
	\bW 
	&= \sum_{j=1}^p (\bI_T \otimes \be_j^\top) 
	(\bI_T \otimes \bSigma^{1/2})
	\bGamma
	(\bI_T \otimes \bSigma^{1/2})
	(\bI_T \otimes \be_j),\\
	\label{eq:A}
	\hbA 
	&= 
	\sum_{i=1}^n (\bI_T \otimes \be_i^\top) \calD (\bI_T \otimes \bX) 
	\bGamma
	(\bI_T \otimes \bX^\top) (\bI_T \otimes \be_i), \\
	\label{eq:K}
	\hbK
	&= 
	\sum_{t=1}^T \trace(\bD_t) \be_t\be_t^\top 
	- 
	\sum_{i=1}^n (\bI_T \otimes \be_i^\top) \calD (\bI_T \otimes \bX) 
	\bGamma
	(\bI_T \otimes \bX^\top) \calS \calD (\bI_T \otimes \be_i),
\end{align}
where 
$\bGamma = \calM^{-1}\bL (\bLambda\otimes \bI_p) \tcalD \in \R^{pT \times pT}$, 
	$\bL = \sum_{t=2}^T \bigl((\be_t\be_{t-1}^\top) \otimes \bI_p\bigr)$, 
	$\bLambda = 
	\sum_{t=1}^T \tfrac{\eta_t}{|I_t|}
	\be_t\be_t^\top$,
	\begin{align*}
		\calM = 
		\begin{bmatrix}
			\bI_p&&&&\\
			-\bP_1&\bI_p&&&\\
			&\ddots&\ddots&&\\
			&&-\bP_{T-1}&\bI_p&\\
		\end{bmatrix}
		\qquad\text{ and }\qquad
		\bP_t = \tbD_t\Bigl(\bI_p - \tfrac{\eta_t}{|I_t|} \bX^\top \bS_t \bD_t \bX\Bigr). 
	\end{align*}
Although notationally involved, the purpose of these matrices is
just to formalize the recursive computation of the derivatives by the chain
rule mentioned in \Cref{sec:intuition}. 

\subsection{Main results: estimating the generalization error consistently}\label{sec:main}

For each iterate $\hbb^t$, define the target $r_t$ (generalization error)
and its estimate $\hat r_t$ by
\begin{align}
        \label{r_t_hat_r_t}
        r_t \defas \norm{\bSigma^{1/2}(\hbb^t - \bb^*)}^2 + \frac{\norm{\bep}^2}{n},
        \qquad
	\hat r_t 
	=&\frac1n \Big\|(\by-\bX\hbb^t) + \sum_{s=1}^{t-1} w_{t,s} \bS_s \bpsi(\by-\bX\hbb^s)\Big\|^2,
\end{align}
where $w_{t,s}:= \be_t^\top \bW \be_s$ and  $\bW\in\R^{T\times T}$ is
the matrix defined in \eqref{eq:W}. The following shows that
$|\hat r_t - r_t|\to^P 0$ (convergence to 0 in probability) under suitable assumptions.

\begin{theorem}[Proved in \Cref{proof-thm:using-Sigma}]\label{thm:using-Sigma}
	Let \Cref{assu:X,assu:regime,assu:rho-1} be fulfilled.
        Then $\forall \epsilon > 0$,
	\begin{equation}
\P\Bigl(|\hat r_t - r_t|>\epsilon \Bigr)
\le
\max
\Bigl\{1, 
\frac{C(T, \gamma, \eta_{\max}, c_0,\delta,\kappa)}{\epsilon}
\Bigr\}
\Bigl(\frac{1}{\sqrt n}
    +
\E\Bigl[\min\Bigl\{1,
\frac{\|\bep\|}{n}
\Bigr\}\Bigr] 
\Bigr).
	\end{equation} 
If additionally \Cref{assu:noise} holds then
$\E[\min\{1,
\frac{\|\bep\|}{n}
\}]\to 0$, 
so that, as $n,p\to+\infty$ while
$(T,\gamma,\eta_{\max},c_0,\delta,\kappa,\epsilon)$ are held fixed, the
right-hand side converges to 0 and $\hat r_t - r_t$ converges to 0 in
probability.
\end{theorem}
This establishes that $\hat r_t$ is consistent at estimating $r_t$.
The statement $\E[\min\{1,
\|\bep\|/n
\}]\to 0$ is equivalent to $\|\bep\|^2/n^2\to^P 0$ (convergence in probability), and is proved in \cite{large-deviation} under the assumption
that $\E|\veps_i|<+\infty$ with $\veps_i$ \iid from a fixed distribution; 
this allows $\operatorname{Var}[\veps_i]=+\infty$ as long as the first moment
is finite.
The expression of $\bW$ involves $\bSigma$, which is typically unknown in practice. 
Our next result provides a consistent estimate of $\bW$ using quantities that do not require the knowledge of $\bSigma$. 

We propose to estimate $\bW$ by $\tbW \defas \hbK^{-1} \hbA$ where $\hbK$ and $\hbA$ are the $T\times T$ matrices defined in \eqref{eq:A}-\eqref{eq:K}.
We define another estimate $\tilde r_t$ by replacing $\bW$ in \eqref{r_t_hat_r_t} with $\tbW = \hbK^{-1} \hbA$:  
\begin{align*}
	\tilde r_t 
	=& \frac 1 n \Big\|{(\by-\bX\hbb^t) + \sum_{s=1}^{t-1} \tilde w_{t,s} \bS_s \bpsi(\by-\bX\hbb^s)}\Big\|^2 ,
\end{align*}
where $\tilde w_{t,s} = \be_t^\top \tbW \be_s$. 

\begin{theorem}[Proved in \Cref{proof-thm:unknwon-Sigma}]\label{thm:unknwon-Sigma}
	Under \Cref{assu:X,assu:rho-1,assu:rho-2,assu:regime,assu:rho-2},
        for any $\epsilon > 0$,
\begin{align*}
	\P\Bigl(
	|\tilde r_t - r_t| > \epsilon
	\Bigr)
	\le 2e^{-n/18} + 
        \max\Bigl\{1,
        \tfrac{C(T, \gamma, \eta_{\max}, c_0, \delta,\kappa)}{\epsilon}
        \Bigr\}
        \Bigl[\tfrac{1}{\sqrt{n}} + \E [\min(1, \tfrac{\norm{\bep}}{n} )] \Bigr].
\end{align*}
If additionally \Cref{assu:noise} holds then
$\E[\min\{1,
\frac{\|\bep\|}{n}
\}]\to 0$,
so that, as $n,p\to+\infty$ while
$(T,\gamma,\eta_{\max},c_0,\delta,\kappa,\epsilon)$ are held fixed, the
right-hand side converges to 0 and $\tilde r_t - r_t$ converges to 0 in
probability.
\end{theorem}
This establishes the consistency of $\tilde r_t$. The simulations presented next confirm that the two estimates $\tilde r_t$ and $\hat r_t$ both are accurate estimates of $r_t$. 
The estimate $\tilde r_t$ has the advantage of not relying on the knowledge of $\bSigma$ and are recommended in practice. 

\begin{remark}
        \label{remark:comparison}
	We highlight that directly generalizing the approach in \cite{bellec2024uncertainty} would lead to the approximation 
	$\tbA \approx \tbK \bW$, where 
	$\tbA$ and $\tbK$ are given in \eqref{eq:tbA} and \eqref{eq:tbK}, respectively.
	From $\tbA \approx \tbK \bW$, 
            obtaining an estimate of $\bW$ requires inverting $\tbK$.
            However, this inversion fails for SGD for small (but still very realistic) batch sizes of order $0.1 n$  in simulations (see \Cref{fig:add-suboptimal}).
	The matrix $\tbK$ is lower triangular, and the reason for the lack of invertibility of $\tbK$ can be seen in the diagonal terms equal to $\text{Tr}[\bS_t\bD_t]$ in \eqref{eq:tbK}, where $\bS_t \in \{0,1\}^{n \times n}$ is the diagonal matrix with 1 in position $(i,i)$ if and only if the $i$-th observation is used in the $t$-th batch. This diagonal element of $\tbK$ can easily be small (or even 0) for small batches, if the batch only contains observations such that $(\bD_t)_{ii}$ is 0 or small. 
        Let $\tilde r_t^{\text{sub}}$ denote the estimate of the same form as $\tilde r_t$ but using the weight matrix $\tbK^{-1} \tbA$ instead. Simulation results in \Cref{fig:add-suboptimal} confirm that $\tilde r_t^{\text{sub}}$ is suboptimal compared to our proposed $\tilde r_t$. 
        For SGD and proximal SGD, we solved this issue regarding the invertibility of $\tbK$ by using out-of-batch samples in the construction of $\hbK$ and $\hbA$, in order to avoid $\bS_t$ in the diagonal elements of $\hbK$ in equation \eqref{eq:hbK}. This is the key to making these estimators work for SGD and proximal SGD, and this use of out-of-batch samples is new compared to \cite{bellec2024uncertainty} (which only tackles the square loss with full-batch gradients).

\end{remark}

\begin{remark}
The constant $C(T,\gamma,\eta_{\max},c_0,\delta,\kappa)$ in the above
results is not explicit. Inspection of the proof reveals that
the dependence of this constant
in $T$ is currently $T^T$, allowing $T$ of order
$\log(n)/\log\log n$ before the bound becomes vacuous.
Improving this dependence in $T$ appears challenging and possibly
out of reach of current tools, even for the well-studied Approximate Message Passing (AMP) algorithms. The papers \cite{rush2018finite,rush2020asymptotic} feature for instance the same $\log(n)/\log\log n$ dependence for approximating the risk of AMP. The preprint \cite{li2024non} offers the latest advances on the dependence on $T$ in the bounds satisfied by AMP. It allows $T\asymp \text{poly}(n)$ while still controlling certain AMP related quantities, although for the risk \cite[equations (16)-(17)]{li2024non} the condition required on $T$ is still logarithmic in $n$. This suggests that advances on this front are possible, at least for isotropic design and specific loss and regularizer such as those studied in \cite{li2024non}: Lasso or Robust M-estimation with no regularizer.
Since these latest advances in \cite{li2024non} are obtained for specific estimates (Lasso or Robust M-estimation with no regularizer), it may be possible to follow a similar strategy and improve our bounds for specific examples of iterative algorithms closer to AMP, or algorithms featuring only separable losses and penalty. We leave such improvements for specific examples for future work,
as the goal of the current paper is to cover a general framework allowing iterations of the form \eqref{eq:unified}
with little restriction on the nonlinear functions
except being Lipschitz.
\end{remark}

\section{Simulation} \label{sec:simulation}
In this section, we present numerical experiments to assess the performance of the proposed risk estimates. 
All necessary code for reproducing these experiments is provided in the supplementary material and is publicly available in the GitHub repository
\url{https://github.com/kaitan365/SGD-generlization-errors}.
Our goal is to compare the performance of the proposed risk estimates with the true risk $r_t$ for different regression methods and iterative algorithms. 

We generate the dataset $(\bX, \by)$ from the linear model \eqref{eq:linear-model}, that is, $\by = \bX \bb^* + \bep$. Here, the rows of $\bX \in \R^{n \times p}$ are sampled from a centered multivariate normal distribution with covariance matrix $\bSigma = \bI_p$. 
The noise vector $\bep$ consists of i.i.d. entries drawn from a $t$ distribution with two degrees of freedom so that the noise variance is infinite. The true regression vector $\bb^* \in \R^p$ is chosen with $p/20$ nonzero entries, set to a constant value such that the signal strength $\norm{\bb^*}^2$ equals 10.

We explore two scenarios of the $(n, p)$ pairs and corresponding iterative algorithms:
\begin{itemize}
    \setlength{\itemindent}{0cm}   
    \item[(i)] $(n, p) = (10000, 5000)$: In this configuration, with $n$ much larger than $p$, we examine Huber regression and Pseudo-Huber regression (without penalty or soft-thresholding). Both the gradient descent (GD) and stochastic gradient descent (SGD) algorithms are implemented for each type of regression.
    \item[(ii)] $(n, p) = (10000, 12000)$: Here, we investigate Huber regression and Pseudo-Huber regression with an L1 penalty, $\lambda \|\bb\|_1$ ($\lambda = 0.002$) and corresponding soft-thresholding step. For each penalized regression, we employ the Proximal Gradient Descent (Proximal GD) and Stochastic Proximal Gradient Descent (Proximal SGD) algorithms.
\end{itemize}
In all algorithms, we start with the initial vector $\hbb^1 = \mathbf{0}_p$ and proceed with a fixed step size 
$\eta = (1 + \sqrt{p/n_*})^{-2}$
where $n_* = n$ for GD and proximal GD, and $n_* = n/5$ for SGD and proximal SGD. 
We run each algorithm for $T = 100$ steps. For SGD and Proximal SGD, batches $I_t \subset \{1, 2, \ldots, n\}$ are randomly sampled without replacement and independently of $(\bX,\by, (I_s)_{s\ne t})$, each with cardinality $|I_t| = \frac{n}{5}$.

A crucial component of the proposed risk estimates $\hat r_t$ and $\tilde r_t$
involve the weight matrices $\bW$ and $\tbW$. The matrix $\bW$ is defined in \Cref{thm:using-Sigma}, and $\tbW = \hbK^{-1} \hbA$ is defined in \Cref{thm:unknwon-Sigma}. 
We employ Hutchinson's trace approximation to compute 
$\bW$, $\hbA$, and $\hbK$. This implementation is computationally efficient. 
We refer readers to \cite[Section 4]{bellec2024uncertainty} for more details.

Recall that we have proposed two estimates for 
$r_t = \norm{\bSigma^{1/2}(\hbb^t - \bb^*)}^2 + \norm{\bep}^2/n$, one is $\hat r_t$ in \Cref{thm:using-Sigma} which requires knowing $\bSigma=\E[\bx_i\bx_i^\top]$, and the other is $\tilde r_t$ in \Cref{thm:unknwon-Sigma} which does not need $\bSigma$. 
Since the quantity $\norm{\bep}^2/n$ remains constant along the algorithm trajectory, we only focus on the estimation of $\norm{\bSigma^{1/2}(\hbb^t - \bb^*)}^2$. 
We repeat each numerical experiment 100 times and present the aggregated results in \Cref{fig:low-dim,fig:high-dim,fig:SGD_eta_dim_1_loss_1}.

In \Cref{fig:low-dim}, we focus on the scenario with $(n,p) = (10000, 5000)$, and plot the actual risk $\norm{\bSigma^{1/2}(\hbb^t - \bb^*)}^2$, 
and its two estimates 
$\hat r_t - \norm{\bep}^2/n$ and $\tilde r_t - \norm{\bep}^2/n$ along with the 2 standard error bar for GD and SGD algorithms applied to both Huber and Pseudo Huber regression.
In \Cref{fig:high-dim}, we focus on the scenario with $(n,p) = (10000, 12000)$, and present the risk curves for the 
Proximal GD and Proximal SGD algorithms applied to both 
L1-penalized Huber regression and Pseudo-Huber regression.

\Cref{fig:low-dim} and 
\Cref{fig:high-dim} confirm the three curves are in close agreement, indicating that the proposed estimates $\hat r_t - \norm{\bep}^2/n$ and $\tilde r_t - \norm{\bep}^2/n$ are consistent estimates of the actual risk $\norm{\bSigma^{1/2}(\hbb^t - \bb^*)}^2$. 
The two estimates closely capture the risk $\norm{\bSigma^{1/2}(\hbb^t - \bb^*)}^2$ over the entire trajectory of the algorithms. For GD and Proximal GD, the risk curves exhibit a U-shape, first decreasing and then increasing, and the estimates $\hat r_t$ and $\tilde r_t$ closely capture this pattern. This suggests that the proposed estimates are reliable and can be used to monitor the risk of the iterates and find the optimal iteration (the iteration minimizing the generalization error) along the trajectory of the algorithm.

\begin{figure}[!ht]
	\centering
	\begin{subfigure}{0.4\linewidth}
		\centering
		\includegraphics[width=\linewidth]{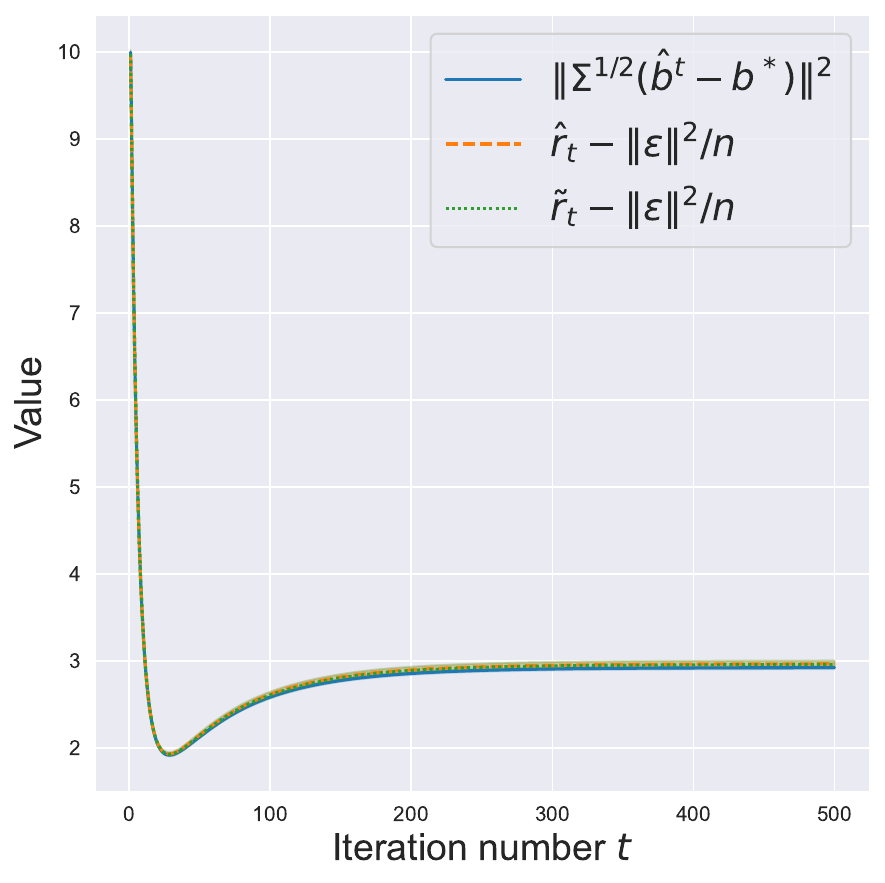}
		\caption{Risk curves of GD for Huber regression}
		\label{fig:Huber-GD}
	  \end{subfigure}~ 
	  \begin{subfigure}{0.4\linewidth}
		\centering
		\includegraphics[width=\linewidth]{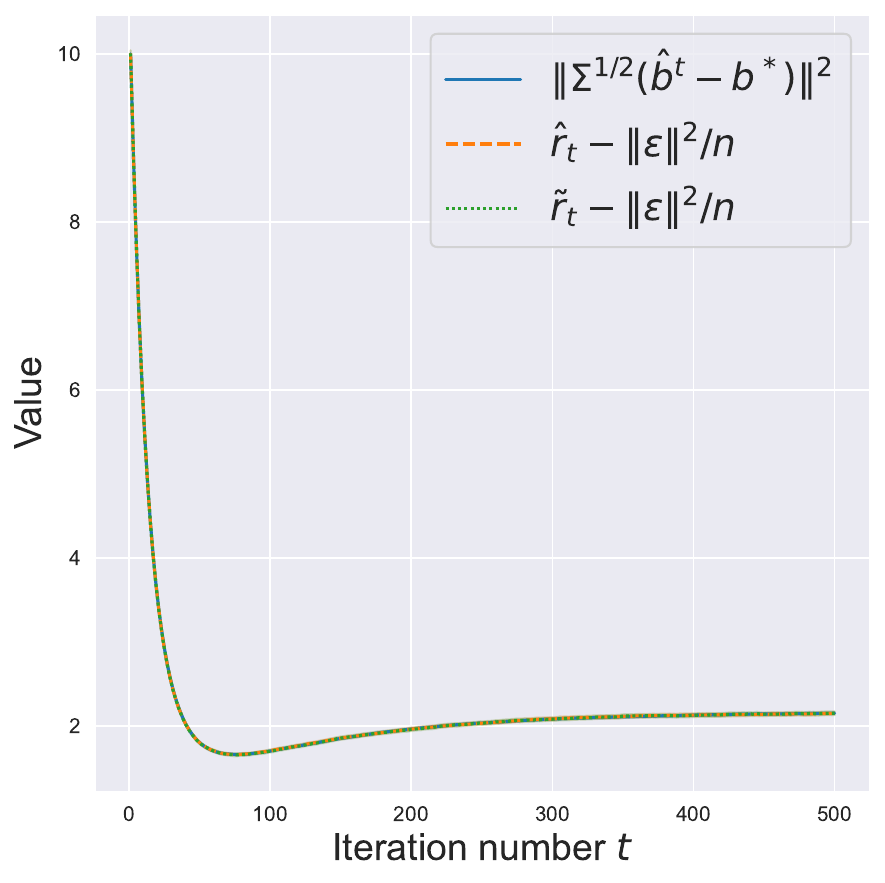}
		\caption{Risk curves of SGD for Huber regression}
		\label{fig:Huber-SGD}
	  \end{subfigure}
	\\
	\begin{subfigure}{0.4\linewidth}
		\centering
		\includegraphics[width=\linewidth]{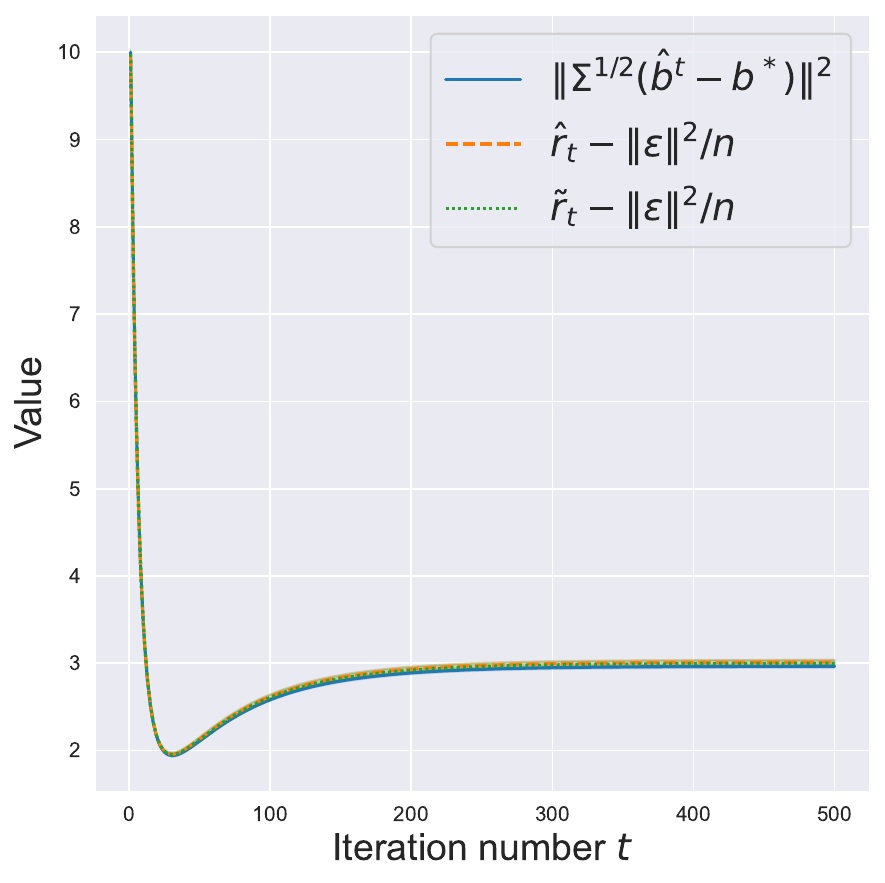}
		\caption{Risk curves of GD for Pseudo-Huber regression}
		\label{fig:pseudo-Huber-GD}
	  \end{subfigure}~ 
	  \begin{subfigure}{0.4\linewidth}
		\centering
		\includegraphics[width=\linewidth]{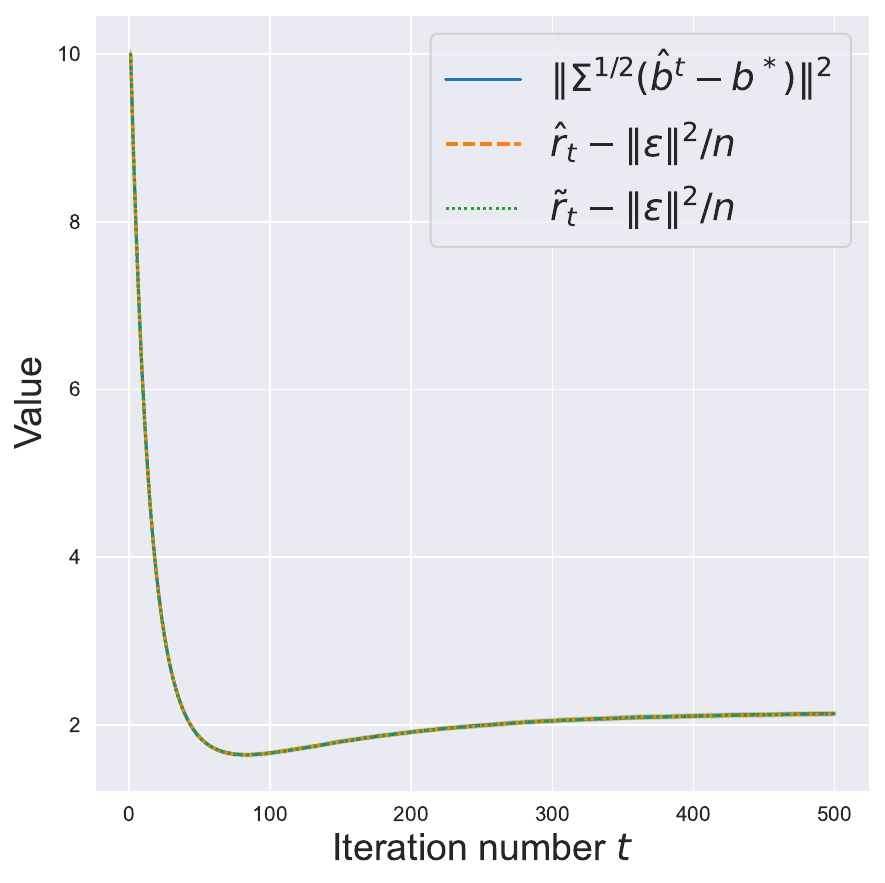}
		\caption{Risk curves of SGD for Pseudo-Huber regression}
		\label{fig:pseudo-Huber-SGD}
	  \end{subfigure}
	\caption{Risk curves for Huber and Pseudo-Huber regression with GD and SGD algorithms for the scenario $(n,p) = (10000,5000)$. \textbf{Upper row:} Huber regression, \textbf{Lower row:} Pseudo-Huber regression. \textbf{Left column:} GD, \textbf{Right column:} SGD.
	}\label{fig:low-dim}
  \end{figure}

\begin{figure}[!ht]
\centering
\begin{subfigure}{0.4\linewidth}
	\centering
	\includegraphics[width=\linewidth]{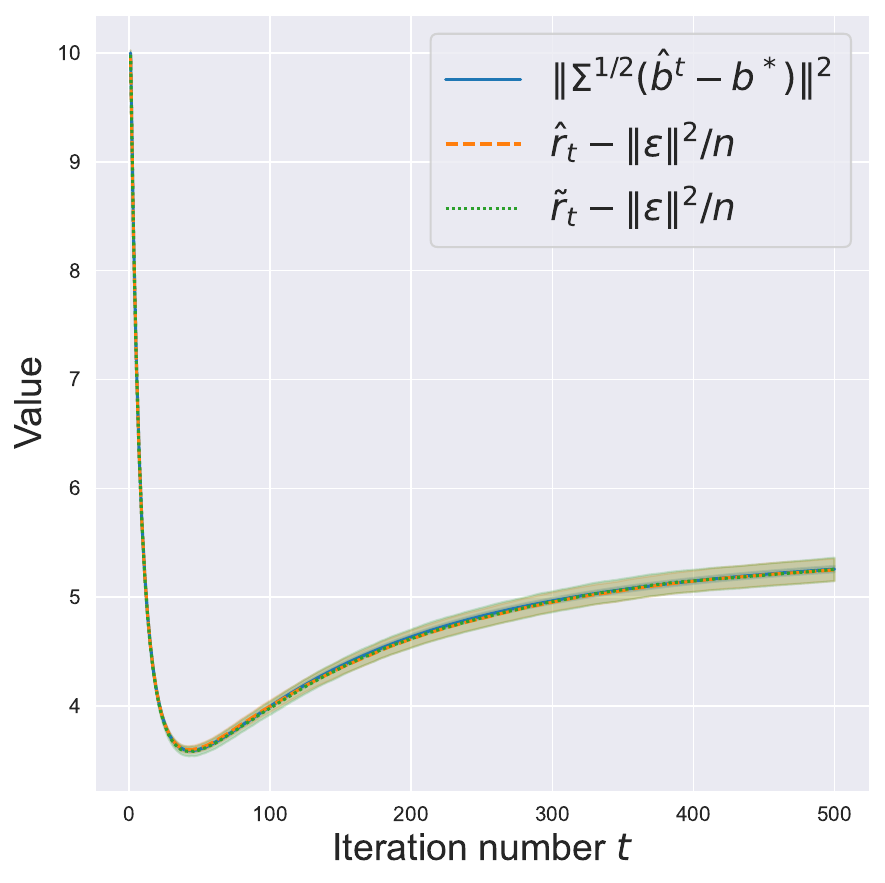}
	\caption{Risk curves of proximal GD for Huber regression}
	\label{fig:Huber-Proximal GD}
	\end{subfigure}~ 
	\begin{subfigure}{0.4\linewidth}
	\centering
	\includegraphics[width=\linewidth]{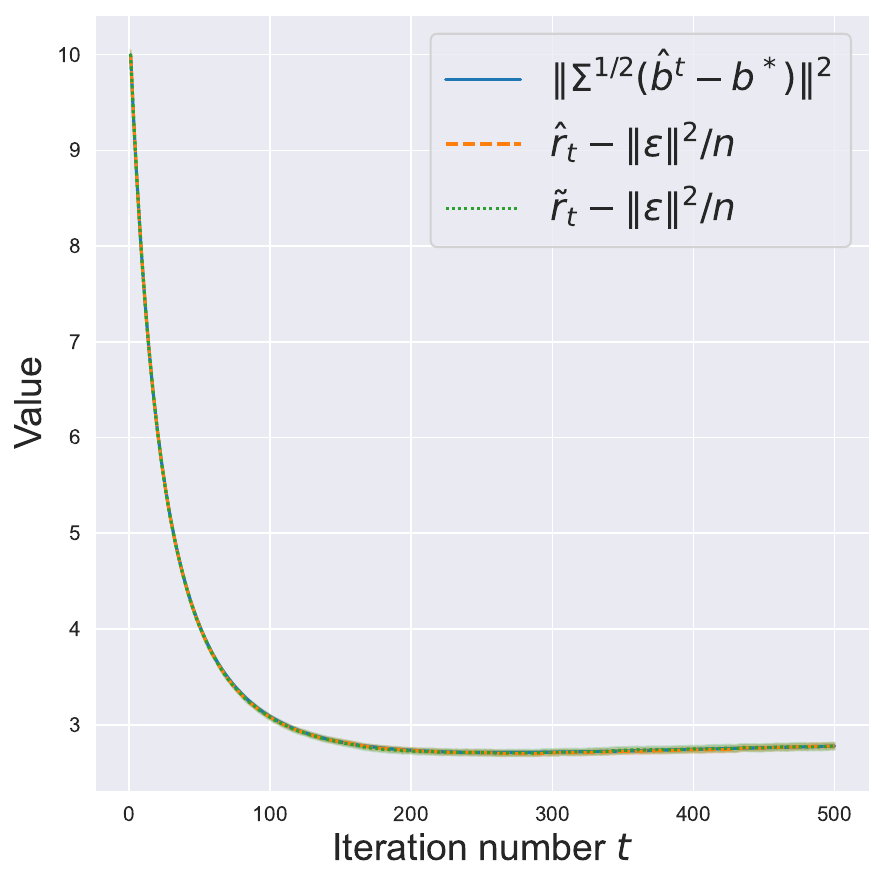}
	\caption{Risk curves of proximal SGD for Huber regression}
	\label{fig:Huber-Proximal SGD}
	\end{subfigure}
\\
\begin{subfigure}{0.4\linewidth}
	\centering
	\includegraphics[width=\linewidth]{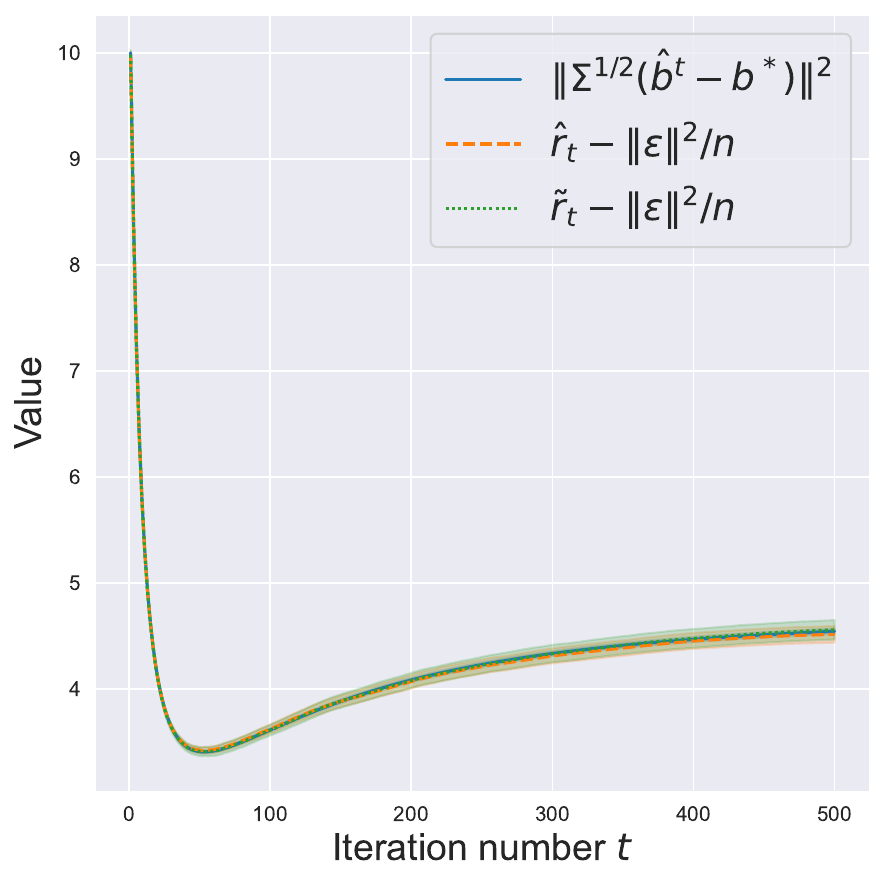}
	\caption{Proximal GD for Pseudo-Huber regression}
	\label{fig:pseudo-Huber-Proximal GD}
	\end{subfigure}~ 
	\begin{subfigure}{0.4\linewidth}
	\centering
	\includegraphics[width=\linewidth]{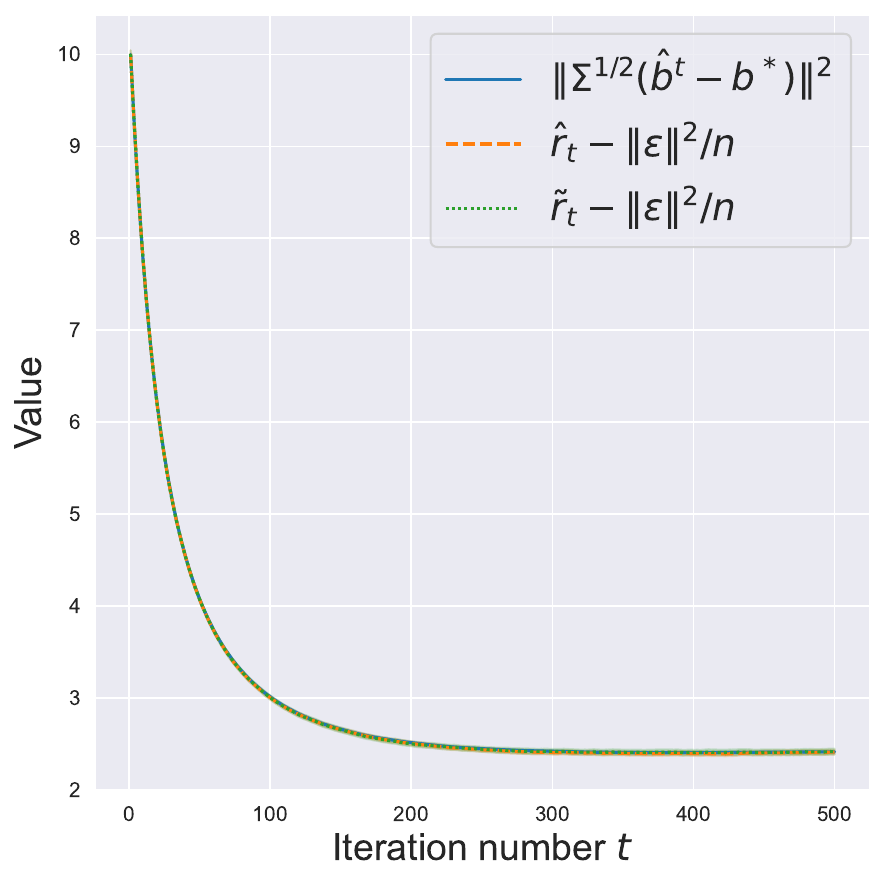}
	\caption{Proximal SGD for Pseudo-Huber regression}
	\label{fig:pseudo-Huber-Proximal SGD}
	\end{subfigure}
\caption{Risk curves for L1-penalized Huber and Pseudo-Huber regression with Proximal GD and Proximal SGD algorithms for the scenario $(n,p) = (10000,12000)$. \textbf{Upper row:} L1-penalized Huber regression, \textbf{Lower row:} L1-penalized  Pseudo-Huber regression. \textbf{Left column:} Proximal GD, \textbf{Right column:} Proximal SGD.
}\label{fig:high-dim}
\end{figure}

\paragraph{Additional experiments: varying step sizes for different iterations.}
We also conduct simulations to investigate the accuracy of the proposed risk estimates in a setting with varying step size. 
We consider two types of step sizes: 
1). $\eta_t = 1$ if $t$ is odd, and $\eta_t = 0$ if $t$ is even; 2). $\eta_t = 1$ if $t$ is odd, and $\eta_t = 0.5$ if $t$ is even.
While the above choices of step size are not preferred in practice, here the goal is show that the proposed risk estimates is able to accurately capture the dynamics of the risk even when the step size changes along the trajectory of the algorithm. For instance, the first choice of step size should produce a risk curve that is flat when $t$ is even.
The results are presented in \Cref{fig:SGD_eta_dim_1_loss_1}, illustrating
that the risk estimates accurately capture the flat segments
of the true risk curve.

\paragraph{Additional experiments: the estimate $\tilde r_t^{\rm sub}$ is suboptimal.}
We compare the performance of $\tilde r_t^{\rm sub}$ with our proposed estimates in Huber regression with $(n,p,T) = (4000, 1000, 20)$ and batch size $ |I_t| = n/10$ and $\eta_t=0.2$ for all $t\in[T]$. It is clear from \Cref{fig:add-suboptimal} that $\tilde r_t$ is more accurate than the suboptimal estimator $\tilde r_t^{\rm sub}$, especially when $t$ increases. 

\newpage
\section{Discussion}\label{sec:disscussion}
This paper proposes a novel risk estimate for the generalization error 
of 
iterates generated by the proximal GD and proximal SGD algorithms in robust regression. 
The proposed risk estimates accurately capture the predictive risk of the iterates along the trajectory of the algorithms, and are
provably consistent (\Cref{thm:using-Sigma,thm:unknwon-Sigma}).
Three matrices in $\R^{T\times T}$ in \eqref{eq:W}-\eqref{eq:K} reveal the interplay between the squared risk, the residuals and the gradients, so that
the approximation \eqref{approx} holds. This structure is different from
the square loss case studied in \cite{bellec2024uncertainty} where only
two matrices (inverse of each other) are sufficient.

Let us mention some open questions along with potential future research directions. The first question regards the probabilistic model: we currently
assume Gaussian features $\bx_i$, and it would be of interest to study
the extension in which our consistency results are universal, allowing non-Gaussian feature distributions.
Second, it is of interest to extend the current estimates to more general optimization problems of the form \eqref{eq:hat-b} with non-smooth data-fitting loss, for instance the Least Absolute Deviation loss $\|\by - \bX \bb\|_1$.
In this case the gradient does not exist at the origin, which calls for
different algorithms than the GD and SGD variants presented here,
for instance the Alternating Direction Method of Multipliers (ADMM) \citep{boyd2011distributed}. 
It is of independent interest to derive the risk estimates for iterates obtained by such primal-dual methods.

\section*{Acknowledgments}
P. C. Bellec acknowledges partial support from the NSF Grants DMS-1945428
and DMS-2413679.
The authors acknowledge the Office of Advanced Research Computing (OARC) at Rutgers, The State University of New Jersey for providing access to the Amarel cluster and associated research computing resources that have contributed to the results reported here.
The authors thank the anonymous reviewers for their valuable comments and suggestions that helped improve the presentation of the paper.
\bibliographystyle{plainnat} 
\bibliography{main}


\newpage
\begin{center}
	\Large
	Supplementary Material 
	of 
	``Estimating Generalization Performance Along the Trajectory of Proximal SGD in Robust Regression''
\end{center}

\appendix

\section{Additional simulation results}
\label{sec:add_simulation_results}
The following figures illustrate the proposed risk estimates accurately estimate
the trajectory of the risk even when the step size changes at every step.

\begin{figure}[H]
	\centering
	\begin{subfigure}{0.5\linewidth}
		\centering
		\includegraphics[width=\linewidth]{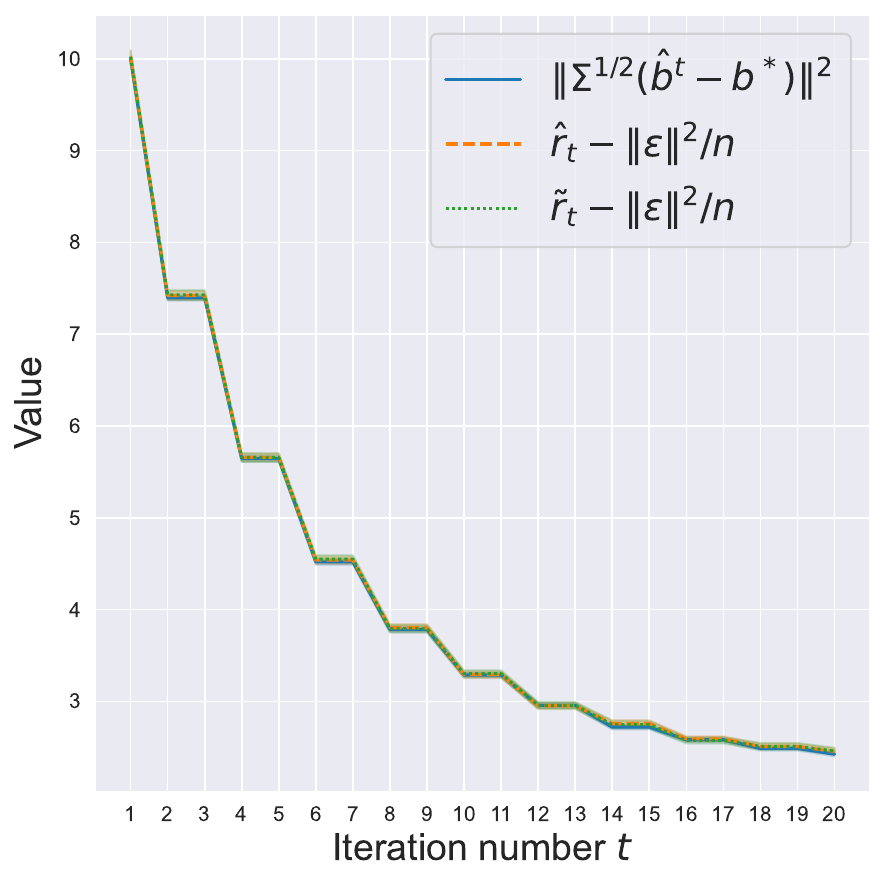}
		\caption{$\eta_t = 1$ if $t$ is odd, and $\eta_t = 0$ if $t$ is even.}
		\label{fig:eta_1}
		\end{subfigure}~ 
		\begin{subfigure}{0.5\linewidth}
		\centering
		\includegraphics[width=\linewidth]{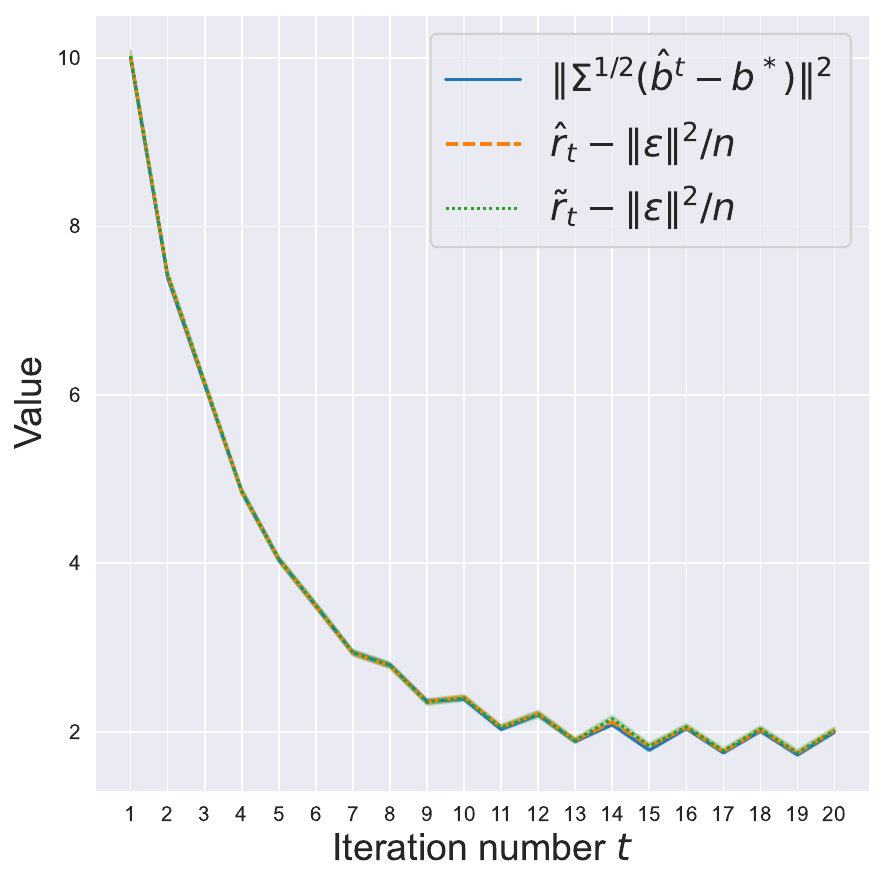}
		\caption{$\eta_t = 1$ if $t$ is odd, and $\eta_t = 0.5$ if $t$ is even.}
		\label{fig:eta_2}
		\end{subfigure}
	\caption{Risk curves for SGD applied to Huber regression with $(n,p) = (3000,1000)$ using different choices of step sizes. \textbf{Left panel:} $\eta_t = 1$ if $t$ is odd, and $\eta_t = 0$ if $t$ is even. \textbf{Right panel:} $\eta_t = 1$ if $t$ is odd, and $\eta_t = 0.5$ if $t$ is even.
	\label{fig:SGD_eta_dim_1_loss_1}
	}
\end{figure}

\Cref{fig:add-suboptimal} compares the performance of the proposed estimators $\hat r_t$, $\tilde r_t$ and the estimator $\tilde r_t^{\rm sub}$ generalized directly from \cite{bellec2024uncertainty}. 
It confirms that our proposed estimators outperforms $\tilde r_t^{\rm sub}$. 
\begin{figure}[H]
	\centering
	\begin{subfigure}{0.5\linewidth}
		\centering
		\includegraphics[width=\linewidth]{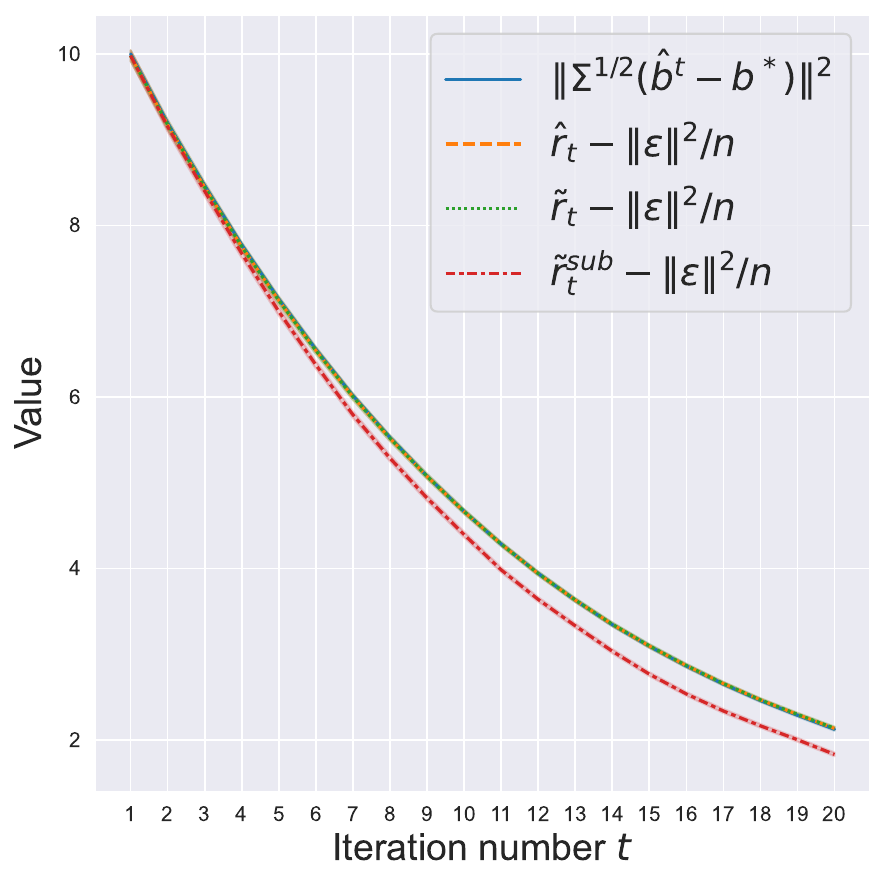}
		\caption{Huber regression}
		\label{fig:Huber-suboptimal}
		\end{subfigure}~ 
		\begin{subfigure}{0.5\linewidth}
		\centering
		\includegraphics[width=\linewidth]{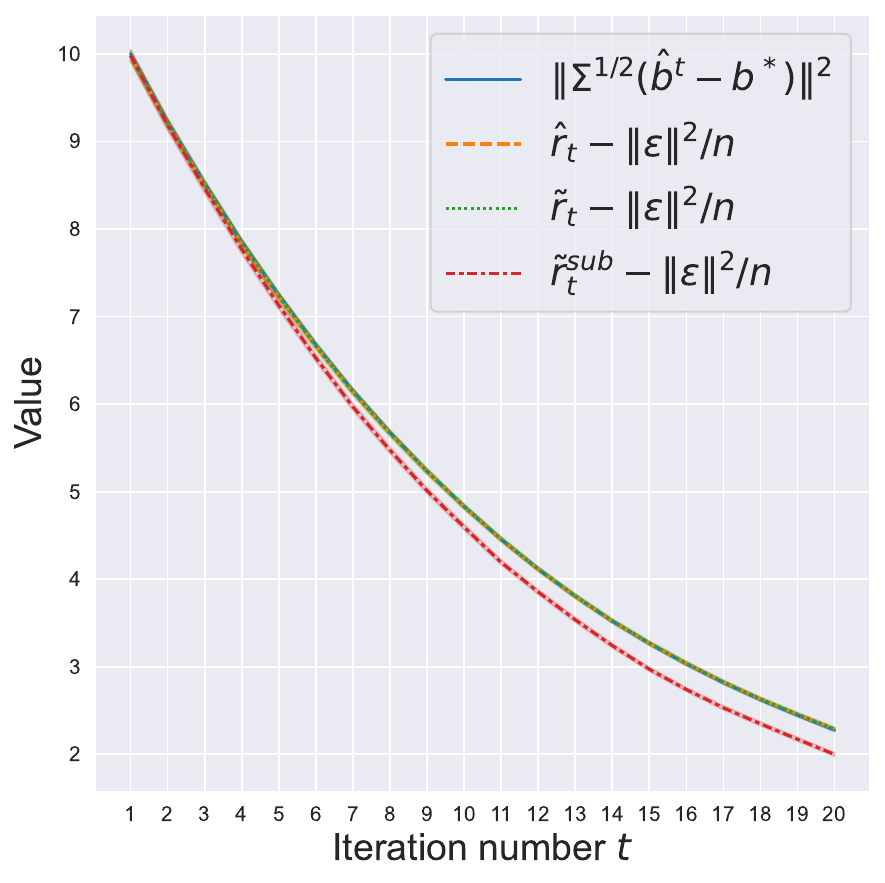}
		\caption{Pseudo-Huber regression}
		\label{fig:pseudo-Huber-suboptimal}
		\end{subfigure}
		\caption{Risk curves for SGD applied to Huber and pseudo-Huber regression with $(n,p,T)=(4000, 1000, 20)$, $|I_t|=n/10$ and $\eta_t=0.2$ for all $t$.
		\label{fig:add-suboptimal}}
\end{figure}

\section{Auxiliary Results}
Throughout, we define
\begin{align}\label{eq:FHR}
	&\bE = [\bep, ..., \bep]\in \R^{n\times T}, 
	\quad\quad\quad\quad
	\bF = [\bS_1\bpsi(\by-\bX\hbb^1), ..., \bS_T\bpsi(\by-\bX\hbb^T)]\in \R^{n\times T},\\
	\quad 
	&\bH = \bSigma^{1/2}[\hbb^1 - \bb^*, ..., \hbb^T - \bb^*] \in \R^{p\times T},
	\quad \quad
	\bR = [\by - \bX\hbb^1,..., \by - \bX\hbb^T]\in \R^{n\times T}. 
\end{align}
Note that in the above $\bE, \bH$ are not observable since $\bep$ and $\bb^*$ are unknown. However, $\bF$ and $\bR$ are observable and can be easily computed once the iterates $(\hbb^t)_{t\in[T]}$ are calculated.

\subsection{Change of variables}
\label{sec:change-var}
In this section, we conduct the change of variable to simplify the proof.
Specifically, we view the linear model $\by = \bX \bb^* + \bep$ as a model with design matrix $\bG$ and the regression vector $\btheta^*$, \ie
\begin{align*}
	\by = \bX \bb^* + \bep 
	= \underbrace{\bX \bSigma^{-1/2}}_{\bG} \underbrace{\bSigma^{1/2} \bb^*}_{\btheta^*} + \bep.
\end{align*}
This way, the design matrix $\bG$ has \iid entries from standard normal distribution. 
Using the same argument in \cite[Appendix D]{bellec2024uncertainty}, we can show that the matrices $\bH, \bF, \hbA, \hbK$ remains the same under the change of variable.
Therefore, we can prove the main results using the model with design matrix $\bG$ and the regression vector $\btheta^*$. 
In other words, we assume without of loss of generality that the design matrix $\bX$ has \iid $N(0,1)$, or equivalently that the independent
rows of $\bX$ are normally distributed with covariance $\bSigma = \bI_p$. 
We prove the main results using $\bSigma = \bI_p$, and the results for general $\bSigma$ follow by this change of variable with the constant $C(T,\gamma, \eta_{\max}, c_0, \delta)$ appearing in the bounds depending additionally on $\kappa$ (the upper bound of the condition number of $\bSigma$ from \Cref{assu:X}).

\subsection{Derivative formulae}
\label{sec:derivative}
In this section, we present derivative formulae that will be useful in later proofs. 
The following formulas differ from \cite{bellec2024uncertainty}
due to the use of robust loss functions and the application of SGD with random batches at each iteration. The formulae
are also significantly more complex than in the case of regularized
M-estimators
\cite{bellec2020out,bellec2021derivatives}.

\begin{lemma}[Proved in \Cref{proof-lem:dot-b}]
	\label{lem:dot-b}
	Let $(\hbb^t)_{t\in[T]}$ be the iterates generated from the recursion \eqref{eq:unified} and the initial value $\hbb^1$ is independent of $\bX$. 
	Then the derivative of $\hbb^t$ with respect to $\bX$ is given by
	\begin{align}
		\pdv{\hbb^t}{x_{ij}}
		= (\be_t^\top \otimes \bI_p) 
		\bGamma 
		\Bigl[((\bF^\top\be_i) \otimes\be_j)
		- (\bI_T\otimes \bX^\top) \calS \calD 
		((\bH^\top\be_j)\otimes \be_i)\Bigr]
		,\label{eq:db-dx}
	\end{align}
	where $\bGamma = \calM^{-1}\bL (\bLambda\otimes \bI_p) \tcalD$, 
	$\bL = \sum_{t=2}^T \bigl((\be_t\be_{t-1}^\top) \otimes \bI_p\bigr)$, 
	$\bLambda = 
	\sum_{t=1}^T \tfrac{\eta_t}{|I_t|}
	\be_t\be_t^\top$,
	and 
	\begin{align*}
		\calM = 
		\begin{bmatrix}
			\bI_p&&&&\\
			-\bP_1&\bI_p&&&\\
			&\ddots&\ddots&&\\
			&&-\bP_{T-1}&\bI_p&\\
		\end{bmatrix}
		\qquad\text{ where}\quad
		\bP_t = \tbD_t(\bI_p - \tfrac{\eta_t}{|I_t|} \bX^\top \bS_t \bD_t \bX). 
	\end{align*}
\end{lemma}

Recall $\bF = [\bS_1\bpsi(\by - \bX\hbb^1), ..., \bS_T\bpsi(\by - \bX\hbb^T)]$, we have $F_{lt} = \be_l^\top \bS_t\bpsi(\by - \bX\hbb^t)$. The following two corollaries are a direct consequence of \Cref{lem:dot-b}. 
\begin{lemma}[Proved in \Cref{proof-lem:dF-dx}
	]\label{lem:dF-dx}
	Under the same conditions of \Cref{lem:dot-b}. 
	Let $F_{lt}=\be_l^\top \bF \be_t = \be_l^\top \bS_t \bpsi(\by - \bX\hbb^t)$,
	we have 
	\begin{align}\label{eq:dF-dx}
		\pdv{F_{lt}}{x_{ij}}
		= D_{ij}^{lt} + \Delta_{ij}^{lt},
	\end{align}
	where 
	\begin{align*}
		D_{ij}^{lt} 
		&= -\be_l^\top \bS_t \bD_t \be_i \be_j^\top \bH \be_t 
		+ 
		((\be_j^\top \bH)\otimes \be_i^\top)\calD \calS (\bI_T\otimes \bX)
		\bGamma^\top
		(\bI_T\otimes \bX^\top) \calS \calD 
		(\be_t \otimes \be_l),\\
		\Delta_{ij}^{lt} 
		&= -
		((\be_i^\top\bF)\otimes \be_j^\top)
		\bGamma^\top 
		(\bI_T\otimes \bX^\top) \calS \calD 
		(\be_t \otimes \be_l).
	\end{align*}
\end{lemma}

\begin{lemma}[Proved in \Cref{proof-lem:dtF-dx}]
	\label{lem:dtF-dx}
	Let $\tbF = [\bpsi(\by - \bX \hbb^1), ..., \bpsi(\by - \bX \hbb^T)]$ and $\tbF_{l,t} = \be_l^\top \tbF \be_t$. 
	Under the same conditions of \Cref{lem:dot-b}. We have 
	\begin{align}
		\pdv{\tbF_{l,t}}{x_{ij}}
		= \tD_{ij}^{lt} + \tDelta_{ij}^{lt},
	\end{align}
	where 
	\begin{align*}
		\tD_{ij}^{lt} 
		&= -\be_l^\top \bD_t \be_i \be_j^\top \bH \be_t 
		+ 
		((\be_j^\top \bH)\otimes \be_i^\top)\calD \calS (\bI_T\otimes \bX)
		\bGamma^\top
		(\bI_T\otimes \bX^\top) \calD 
		(\be_t \otimes \be_l),\\
		\tDelta_{ij}^{lt} 
		&= -
		((\be_i^\top\bF)\otimes \be_j^\top)
		\bGamma^\top 
		(\bI_T\otimes \bX^\top) \calD 
		(\be_t \otimes \be_l).
	\end{align*}
\end{lemma}

\begin{definition}
	\label{def:identity-sum-pdv}
        Define the matrices
        $\bUpsilon_1\in \R^{p\times T}$,
        $\bUpsilon_2\in \R^{n\times T}$,
        $\bUpsilon_3\in \R^{T\times T}$,
        $\bUpsilon_4\in \R^{T\times T}$,
        $\bUpsilon_5\in \R^{T\times T}$
        by the identities
	\begin{align}
		\label{eq:1}
                    \forall j\in[p],\quad 
		&\sum_{i=1}^n \pdv{\be_i^\top \bF}{x_{ij}} 
		= -\be_j^\top \bH \tbK^\top -\be_j^\top \bUpsilon_1,\\
		\label{eq:2}
                    \forall i\in[n],\quad 
		&\sum_{j=1}^p \pdv{\be_j^\top \bH}{x_{ij}}
		= \be_i^\top \bF \bW^\top - \be_i^\top \bUpsilon_2,\\
		\label{eq:3}
		&\sum_{i=1}^n\sum_{j=1}^p \pdv{\bF^\top \be_i\be_j^\top \bH}{x_{ij}}
		= - \tbK \bH^\top \bH + \bF^\top \bF \bW^\top 
		- \bUpsilon_3,\\
		\label{eq:4}
		&\sum_{i=1}^n\sum_{j=1}^p \pdv{\bH^\top \bX^\top \be_i\be_j^\top \bH}{x_{ij}}
		= (n\bI_T - \tbA)\bH^\top \bH + \bH^\top\bX^\top\bF\bW^\top - \bUpsilon_4,\\
		\label{eq:5}
		&\sum_{i=1}^n\sum_{j=1}^p \pdv{\bF^\top \be_i\be_j^\top\bX^\top \tbF}{x_{ij}}
		= -\tbK\bH^\top \bX^\top \tbF + 
		p \bF^\top\tbF - \bF^\top \bF \hbA^\top - \bUpsilon_5,
	\end{align}
	where the matrices $\tbK, \hbA, \tbA, \bW$ are defined as follows 
	\begin{align} 
		\label{eq:tbA}
		\tbA
		&= 
		\sum_{i=1}^n (\bI_T \otimes \be_i^\top) (\bI_T \otimes \bX) 
		\bGamma
		(\bI_T \otimes \bX^\top) \calS \calD (\bI_T \otimes \be_i), \\
		\label{eq:hbA}
		\hbA 
		&= 
		\sum_{i=1}^n (\bI_T \otimes \be_i^\top) \calD (\bI_T \otimes \bX) 
		\bGamma
		(\bI_T \otimes \bX^\top) (\bI_T \otimes \be_i), \\
		\label{eq:bW}
		\bW 
                &= \sum_{j=1}^p (\bI_T \otimes \be_j^\top) 
		\bGamma
		(\bI_T \otimes \be_j),\\
		\label{eq:tbK}
		\tbK
		&= 
		\sum_{t=1}^T \trace(\bS_t \bD_t) \be_t\be_t^\top 
		- 
		\sum_{i=1}^n (\bI_T \otimes \be_i^\top) \calS \calD (\bI_T \otimes \bX) 
		\bGamma
		(\bI_T \otimes \bX^\top) \calS \calD (\bI_T \otimes \be_i).
	\end{align}
\end{definition}

The matrices $\bUpsilon_1, \bUpsilon_2, \dots$ are negligible in the sense that their Frobenius norms are of smaller orders compared to their preceding terms in \eqref{eq:1}--\eqref{eq:5}.
We provide the bounds in next lemma,
which is obtained by deriving alternative expressions for
    $\bUpsilon_1,...,\bUpsilon_5$ in \Cref{proof-def:identity-sum-pdv}.
\begin{lemma}[Proved in \Cref{proof-lem:Upsilon-bound}]\label{lem:Upsilon-bound}
	Under the same conditions of \Cref{thm:using-Sigma} with $\bSigma = \bI_p$, we have 
	\begin{align*}
		\max_{k\in\{1,3,4\}} \E [\opnorm{\bUpsilon_k}^2 \mid \bep] 
		&\le C(T, \gamma, \eta_{\max}, c_0) (\delta^2 + \norm{\bb^*}^2),\\
		\E [\opnorm{\bUpsilon_2}^2 \mid \bep] 
		&\le n^{-1} C(T, \gamma, \eta_{\max}, c_0) (\delta^2 + \norm{\bb^*}^2),\\
		\E [\opnorm{\bUpsilon_5}^2 \mid \bep] 
		&\le n^2 C(T, \gamma, \eta_{\max}, c_0) (\delta^2 + \norm{\bb^*}^2).
	\end{align*}
\end{lemma}

We further define a few matrices of size $T\times T$:
\begin{equation}\label{eq:Theta}
\begin{aligned}
	&\bTheta_1
	= \bF^\top\bX\bH + \tbK \bH^\top \bH - \bF^\top \bF \bW^\top,\\
	&\bTheta_2
	= n^{-1} [\bF^\top\bX\bX^\top\tbF + \tbK\bH^\top \bX^\top \tbF - p\bF^\top\tbF 
	+ \bF^\top\bF\hbA^\top
	],\\
	&\bTheta_3
	= \bH^\top\bX^\top\bX\bH - (n\bI_T - \tbA)\bH^\top \bH - \bH^\top\bX^\top\bF\bW^\top,\\
	&\bTheta_4
	= \frac pn \bF^\top \tbF - \frac1n 
	(\tbK\bH^\top + \bF^\top \bX)
	(\hbK\bH^\top + \tbF^\top \bX)^\top,\\
	&\bTheta_5
	= n\bH^\top \bH - 
	(\bW\bF^\top -\bH^\top\bX^\top)(\bW\bF^\top -\bH^\top\bX^\top)^\top,\\
	&\bTheta_6
        = {\fnorm{\bE}^{-1}} (\bE^\top \bX\bH - \bE^\top \bF\bW^\top).
\end{aligned}
\end{equation}
The next lemma provides the moment bounds for the Frobenius norm of these matrices.
\begin{lemma}[Proved in \Cref{proof-lem:Theta-bound}]\label{lem:Theta-bound}
	Under the same conditions of \Cref{thm:using-Sigma}, we have 
	\begin{align}
		\max_{k\in\{1,2,3\}} 
		\E [\fnorm{\bTheta_k}^2 \mid \bep]  &\le n C(T, \gamma, \eta_{\max}, c_0) (\delta^2 + \norm{\bb^*}^2),\\
		\max_{k\in\{4,5\}} 
		\E [\fnorm{\bTheta_k} \mid \bep]  &\le n^{1/2} C(T, \gamma, \eta_{\max}, c_0) (\delta^2 + \norm{\bb^*}^2),\\
		\E [\fnorm{\bTheta_6}^2\mid \bep] &\le C(T, \gamma, \eta_{\max}, c_0) (\delta^2 + \norm{\bb^*}^2),
	\end{align}
	almost surely, where $\E[\cdot\mid\bep]$ is the conditional
        expectation given $\bep$.
\end{lemma}
We are able to prove the main theorems using
\Cref{lem:Theta-bound}.

\section{Proof of main results}

\subsection{Proof of \Cref{thm:using-Sigma}}
\label{proof-thm:using-Sigma}
It suffices to prove this theorem for the case $\bSigma = \bI_p$. When $\bSigma \ne \bI_p$, the result can be derived using a change of variables argument, as outlined in \Cref{sec:change-var}. 
By basic algebra, we have 
\begin{align*}
	&\bTheta_5 + \fnorm{\bE} (\bTheta_6 + \bTheta_6^\top)\\
	=~& n\bH^\top \bH - 
	(\bX\bH- \bF\bW^\top)^\top
	(\bX\bH- \bF\bW^\top) + \bE^\top (\bX\bH - \bF \bW^\top) + (\bX\bH - \bF \bW^\top)^\top \bE\\
	=~& n\bH^\top \bH + \bE^\top\bE 
	- (\bE - \bX\bH + \bF\bW^\top)^\top (\bE - \bX\bH + \bF\bW^\top)\\
	=~& n\bH^\top \bH + \bE^\top\bE 
	- (\bR +\bF\bW^\top)^\top (\bR+ \bF\bW^\top).
\end{align*}
Notice that $r_t = \norm{\hbb^t - \bb^*}^2 + \norm{\bep}^2/n$ is the $t$-th diagonal entry of 
$(\bH^\top \bH + \bE^\top \bE/n)$, and 
$\hat r_t$ is the $t$-th diagonal entry of $(\bR +\bF\bW^\top)^\top (\bR+ \bF\bW^\top)/n$.
Since $\fnorm{\bE} = \sqrt{T} \norm{\bep}$, using the previous display that
conditionally on $\bep$, we have
\begin{align*}
	\E\Bigl[\big|\hat r_t - r_t\big|
        \mid \bep \Bigr]
\le n^{-1}
\E \Bigl[ 
	\fnorm{\bTheta_5} + 2\fnorm{\bE} \fnorm{\bTheta_6} \mid \bep
	\Bigr]
= n^{-1}
\E 
\Bigl[ 
    \fnorm{\bTheta_5} +  2\sqrt T \|\bep\| \fnorm{\bTheta_6} \mid \bep
\Bigr].
\end{align*}
Using the moment bounds of $\bTheta_5$ and $\bTheta_6$ in \Cref{lem:Theta-bound}, we have
\begin{align*}
	\E\Bigl[{\big|\hat r_t - r_t\big|} \mid \bep\Bigr]
\le \frac{C(T, \gamma, \eta_{\max}, c_0,\delta)}{\sqrt n}
\Bigl(1+\frac{\|\bep\|}{\sqrt n}\Bigr).
\end{align*}
Furthermore, if $\E[|\veps_i|]$ is finite, we have by \cite{large-deviation} that 
$\norm{\bep}/n \to^P 0$ (convergence in probability) if $\bep$ has \iid entries with a fixed
distribution independent of $n$.
Under this assumption, the right-hand side of the previous display
converges to 0 in probability.
By enlarging the constant if necessary, assume
$C(T, \gamma, \eta_{\max}, c_0,\delta)\ge 1$.
To obtain a quantitative bound,
by the conditional version of Markov's inequality, for any $\epsilon > 0$,
and almost surely with respect to $\bep$ that
\begin{align*}
\P\Bigl(|\hat r_t - r_t|>\epsilon \mid \bep\Bigr)
&\le 
\min\Bigl\{1,
\frac{C(T, \gamma, \eta_{\max}, c_0,\delta)}{\epsilon}
\Bigl(\frac{1}{\sqrt n} + \frac{\|\bep\|}{n}\Bigr)
\Bigr\}
\\&\le
\max\Bigl\{1,
\frac{C(T, \gamma, \eta_{\max}, c_0,\delta)}{\epsilon}
\Bigr\}
\min\Bigl\{1,
\frac{1}{\sqrt n} + \frac{\|\bep\|}{n}
\Bigr\}.
\end{align*}
Taking expectation with respect to $\bep$, we obtain
\begin{align*}
\P\Bigl(|\hat r_t - r_t|>\epsilon \Bigr)
&\le
\max\Bigl\{1,
\frac{C(T, \gamma, \eta_{\max}, c_0,\delta)}{\epsilon}
\Bigr\}
\E\Bigl[\min\Bigl\{1,
\frac{1}{\sqrt n} + \frac{\|\bep\|}{n}
\Bigr\}\Bigr] 
\\&\le
\max\Bigl\{1,
\frac{C(T, \gamma, \eta_{\max}, c_0,\delta)}{\epsilon}
\Bigr\}
\Bigl(\frac{1}{\sqrt n}
    +
\E\Bigl[\min\Bigl\{1,
\frac{\|\bep\|}{n}
\Bigr\}\Bigr] 
\Bigr)
\end{align*}
with
$
\E[\min\{1,
\frac{\|\bep\|}{n}
\}]\to 0$ (equivalently, $\|\bep\|/n\to^P0$) if the entries of $\bep$ are \iid with a fixed distribution
independent of $n$ with $\E[|\varepsilon_i|]<+\infty$
by \cite{large-deviation}.
This finishes the proof of \Cref{thm:using-Sigma}.

\subsection{Operator norm bound on $\widehat K$}
We first recall the definition of $\hbK$ from \eqref{eq:K} in the main text:
\begin{equation}\label{eq:hbK}
\hbK
= 
\sum_{t=1}^T \trace(\bD_t) \be_t\be_t^\top 
- 
\sum_{i=1}^n (\bI_T \otimes \be_i^\top) \calD (\bI_T \otimes \bX) 
\bGamma
(\bI_T \otimes \bX^\top) \calS \calD (\bI_T \otimes \be_i).
\end{equation}

Define two events: 
\(
\Omega_1 = \{\bX\in \R^{n\times p}: \opnorm{\bX}/\sqrt{n} \le 2 + \sqrt{p/n} \}
\)
and 
\(
\Omega_2 = 	\{|\{i\in[n]: |\bep_i| \le M\}| \ge \frac{2n}{3}\},
\)
where $M$ is a large enough constant such that $\P(|\bep_i| > M)\le 1/6$.

\begin{lemma}\label{lem:hbK-opnorm}
	Under the same conditions of \Cref{thm:using-Sigma} with $\bSigma = \bI_p$, we have in the event $\Omega_* = \Omega_1 \cap \Omega_2$ that 
	\begin{align*}
		n \opnorm{\hbK^{-1}} \le C(T, \gamma, \eta_{\max}, c_0, \delta, \norm{\bb^*}) .
	\end{align*}
    Furthermore, $\Omega_*$ has probability at least $1-e^{-n/18} - e^{-n/2}$.
\end{lemma}
\begin{proof}[Proof of \Cref{lem:hbK-opnorm}]
Under \Cref{assu:regime,assu:X}, we know that $\P(\Omega_1) \ge 1 - e^{-n/2}$ from \cite{DavidsonS01}. 
In the event $\Omega_1$, we have by \Cref{lem:XH} that 
\begin{align*}
	\fnorm{\bX\bH}^2/n 
	\le C(T, \gamma, \eta_{\max}, c_0) (\delta^2 + \norm{\bb^*}^2):= C_*. 
\end{align*}
Markov's inequality further implies 
\begin{align*}
	|\{i \in [n]: \norm{\bx_i^\top\bH}^2 > 3C_*\}| \le n/3. 
\end{align*}
In other words, $|\{i \in [n]: \norm{\bx_i^\top\bH}^2 \le 3C_*\}| \ge \frac{2n}{3}.$ 
Recall that $M$ is such that $\P(|\bep_i| > M)\le 1/6$. 
By Hoeffding's inequality, we have
\begin{align*}
	\P\Bigl(\frac1n \sum_{i=1}^n 
	\boldone\{|\bep_i| >M\} 
	\ge \P(|\bep_i| > M) + a
	\Bigr)  
	\le e^{-2n a^2}.
\end{align*}
Taking $a = 1/6$, we have $\P\bigl(\frac1n \sum_{i=1}^n 
	\boldone\{|\bep_i| >M\} 
	\ge \P(|\bep_i| > M) + 1/6
	\bigr)  
	\le e^{-n/18}.$ Furthermore, 
using 
$|\{i\in[n]: |\bep_i| > M\}|
= \sum_{i=1}^n 
\boldone\{|\bep_i| > M\}$, we have 
\begin{align*}
	\Bigl\{ 
		|\{i\in[n]: |\bep_i| > M\}| \ge n/3 |
	\Bigr\}
	=& \Bigl\{ 
		\frac 1n \sum_{i=1}^n 
		\boldone\{|\bep_i| > M\}
		\ge 1/6 + 1/6
	\Bigr\}\\
	\subseteq & \Bigl\{ 
		\frac 1n \sum_{i=1}^n 
		\boldone\{|\bep_i| > M\}
		\ge \P(|\bep_i| > M) + 1/6
	\Bigr\}.
\end{align*}
Therefore, 
\begin{align*}
	\P\Bigl(
		|\{i\in[n]: |\bep_i| > M\}| \ge n/3
	\Bigr)  
	\le e^{-n/18}.
\end{align*}
Equivalently, we have $\P(
		|\{i\in[n]: |\bep_i| \le M\}| \ge \frac{2n}{3}
	)  
	\ge 1 - e^{-n/18}.$ That is, at least $\frac{2n}{3}$ of the entries of $\bep$ are bounded by $M$ with probability at least $1-e^{-n/18}$.

Recall that 
\(
\Omega_2 = 	\{|\{i\in[n]: |\bep_i| \le M\}| \ge \frac{2n}{3}\},
\)
then $\P(\Omega_2) \ge 1 - e^{-n/18}$. 
Hence, 
$\P(\Omega_1 \cap \Omega_2) \ge 1 - e^{-n/18} - e^{-n/2}$. 
In the event $\Omega_1 \cap \Omega_2$, the set 
\begin{align*}
	\hat I = \{i \in [n]: 
	|\bep_i| \le M, 
	\norm{\bx_i^\top\bH}^2 \le 3C_*\}
\end{align*}
has size at least $\frac{n}{3}$.
For any $i \in \hat I$ and $t\in[T]$, we have 
\begin{equation}
    \label{previous_2}
|y_i - \bx_i^\top \hbb^t|
= |\bep_i - \bx_i^\top \bH \be_t|
\le |\bep_i| + |\bx_i^\top \bH \be_t|
\le M + \sqrt{3C_*}. 
\end{equation}

By the definition of $\bD_t$, under \Cref{assu:rho-2}, we have 
\begin{align*}
	\trace(\bD_t) 
	&=
	\sum_{i=1}^n 
	\rho''(y_i - \bx_i^\top \hbb^t)\\
	&> 
	\sum_{i\in \hat I} 
	\rho''(y_i - \bx_i^\top \hbb^t)
        &&\text{since $\rho''\ge 0$ by convexity}
        \\ &\ge |\hat I| \min_{u: |u| \le M + \sqrt{3C_*}} \rho''(u)
           &&\text{due to \eqref{previous_2}}
        \\
	&\ge n/3 \min_{u: |u| \le M + \sqrt{3C_*}} \rho''(u)
	:= c_* n
            &&\text{since $\hat I$ has size at least $n/3$}.
\end{align*}
Here $c_*$ is a constant depending on $\rho,M,C_*$ only.

By the definition of $\hbK$ in \eqref{eq:hbK}, $\hbK/n$ is a lower triangular matrix with diagonal entries equal to $\trace(\bD_t)/n$. It is invertible if and only if all its diagonal entries are non-zero. 
Therefore, in the event $\Omega_1 \cap \Omega_2$, we have $\hbK/n$ is invertible. 

Let $\hbLambda = \sum_{t=1}^T \trace(\bD_t) \be_t\be_t^\top
$. Then it is diagonal, $\opnorm{\hbLambda^{-1}}=\max_{t\in[T]}\trace[\bD_t]^{-1}\le (c_* n)^{-1}$ and 
\begin{align}\label{eq:hbK-hbLambda}
	\hbK = \hbLambda - \hbL
\end{align}
where $\hbL = \sum_{i=1}^n (\bI_T \otimes \be_i^\top) \calD (\bI_T \otimes \bX) 
\bGamma
(\bI_T \otimes \bX^\top) \calS \calD (\bI_T \otimes \be_i).$
Here $\hbL$ is a strictly lower triangular matrix. 
Using the upper bound of $\opnorm{\bGamma}$ in \Cref{lem:Gamma-opnorm}, we have
$\opnorm{\hbL} \le n C(T, \gamma, \eta_{\max}, c_0)$ in the event $\Omega_1$.  
Now we rewrite $\hbK^{-1}$ as 
\begin{align*}
	\hbK^{-1} 
	= (\hbK \hbLambda^{-1} \hbLambda)^{-1}
	= \hbLambda^{-1} (\hbK \hbLambda^{-1})^{-1}
	= \hbLambda^{-1} (\bI_T - \hbL \hbLambda^{-1})^{-1}.
\end{align*}
Notice that $\hbL \hbLambda^{-1}\in \R^{T\times T}$ is a strictly lower triangular matrix, thus
\begin{align*}
	(\bI_T - \hbL \hbLambda^{-1})^{-1}
	= \sum_{k=0}^\infty (\hbL \hbLambda^{-1})^k
	= \sum_{k=0}^{T-1} (\hbL \hbLambda^{-1})^k.
\end{align*}
By the triangle inequality,
\begin{align*}
	\opnorm{(\bI_T - \hbL \hbLambda^{-1})^{-1}}
	\le \sum_{k=0}^{T-1} \opnorm{\hbL \hbLambda^{-1}}^k
	\le C(T, \gamma, \eta_{\max}, c_0, \delta, \norm{\bb^*}). 
\end{align*}
Therefore, in the event $\Omega_1 \cap \Omega_2$ which has probability $\P(\Omega_1 \cap \Omega_2) \ge 1-e^{-n/18} - e^{-n/2}$, 
\begin{align*}
	\opnorm{\hbK^{-1}} 
	\le \opnorm{\hbLambda^{-1}} \opnorm{(\bI_T - \hbL \hbLambda^{-1})^{-1}}
	\le n^{-1} C(T, \gamma, \eta_{\max}, c_0, \delta, \norm{\bb^*}).
\end{align*}
\end{proof}

\begin{lemma}\label{lem:A-K-W-opnorm}
	Under the same conditions of \Cref{thm:unknwon-Sigma} with $\bSigma=\bI_p$, we have
	\begin{align*}
		\opnorm{\hbK} \le n (1 + \opnorm{\bX}^2 \opnorm{\bGamma}),
                \qquad
                \opnorm{\hbA} \le n \opnorm{\bX}^2 \opnorm{\bGamma},
                \qquad
		\opnorm{\bW} \le n \opnorm{\bGamma}. 
	\end{align*}
\end{lemma}
\begin{proof}[Proof of \Cref{lem:A-K-W-opnorm}]
	By the definition of $\hbK$ in \eqref{eq:hbK}, 
	using $\opnorm{\calD}\le 1$ and $\opnorm{\calS}\le1$, 
	we have 
	\begin{align*}
		\opnorm{\hbK} 
		\le \opnorm{\hbLambda} + \opnorm{\hbL} 
		\le n (1 + \opnorm{\bX}^2 \opnorm{\bGamma}). 
	\end{align*}
	Similarly, by the definition of $\hbA$ in \eqref{eq:hbA}, we have
	\begin{align*}
		\opnorm{\hbA} \le n \opnorm{\bX}^2 \opnorm{\bGamma}.
	\end{align*}
	Last, by the definition of $\bW$ in \eqref{eq:bW}, we have
        $\opnorm{\bW} \le n \opnorm{\bGamma}.$
\end{proof}

\begin{lemma}\label{lem:W-W}
	Under the same conditions of \Cref{thm:using-Sigma} with $\bSigma = \bI_p$, we have 
	\begin{align*}
		&\E [\fnorm{\bF^\top \bF (\hbA - \hbK\bW)^\top}
                \mid \bep
                ]
		\le n^{3/2} C(T, \gamma, \eta_{\max}, c_0) (\delta^2 + \norm{\bb^*}^2),\\
		&\E [\fnorm{\bR^\top \bF (\hbA - \hbK\bW)^\top}
                \mid \bep]
		\le n^{2} (\tfrac{\norm{\bep}}{n} + \tfrac{1}{\sqrt{n}}) C(T, \gamma, \eta_{\max}, c_0) (\delta^2 + \norm{\bb^*}^2).
	\end{align*}
\end{lemma}

\begin{proof}[Proof of \Cref{lem:W-W}]
First, using the definitions of $\bTheta_1,\bTheta_2,\bTheta_4$ in \eqref{eq:Theta}, we have 
\begin{align}
n^{-1} \bF^\top \bF (\hbA - \hbK\bW)^\top
=n^{-1} \bTheta_1 \hbK^\top + \bTheta_2 + \bTheta_4.
\label{previous_1}
\end{align}
Hence,
\begin{align*}
	&\E[\fnorm{\bF^\top \bF(\hbA - \hbK\bW)^\top}\mid \bep]\\
	=~& \E [\fnorm{(\bTheta_1 \hbK^\top + n\bTheta_2 + n\bTheta_4)}\mid \bep]
          &&\mbox{by \eqref{previous_1}}
        \\
	\le~& \E [\fnorm{\bTheta_1}^2 \mid \bep]^{1/2} 
	\E[\opnorm{\hbK}^2 \mid \bep]^{1/2} 
        + n \E [\fnorm{\bTheta_2} + \fnorm{\bTheta_4}\mid \bep]
            &&\mbox{by the Cauchy-Schwarz inequality}
        \\
	\le~& n^{3/2} C(T, \gamma, \eta_{\max}, c_0) (\delta^2 + \norm{\bb^*}^2).
\end{align*}
Here the last line uses the upper bounds of $\E [\fnorm{\bTheta_k}\mid \bep]$ from \Cref{lem:Theta-bound}, and the bound of $\E[\opnorm{\hbK}\mid \bep]$ follows from  
\Cref{lem:A-K-W-opnorm} and the bound of $\opnorm{\bGamma}$ from \Cref{lem:Gamma-opnorm}.
This proves the first inequality.

For the second inequality,
we define
\begin{align*}
	&\cbTheta_1
	= \frac{\bR^\top\bX\bH + \cbK \bH^\top \bH - \bR^\top \bF \bW^\top}{\fnorm{\bE}/\sqrt{n} + 1},\\
	&\cbTheta_2
	= 
	\frac{\bR^\top\bX\bX^\top\tbF + \cbK\bH^\top \bX^\top \tbF - p\bR^\top\tbF 
	+ \bR^\top\bF\hbA^\top}{n (\fnorm{\bE}/\sqrt{n} + 1)},\\
	&\cbTheta_4
	= \frac{p \bR^\top \tbF -
	(\cbK\bH^\top + \bR^\top \bX)
	(\hbK\bH^\top + \tbF^\top \bX)^\top}{n (\fnorm{\bE}/\sqrt{n} + 1)},
\end{align*}
where 
$\cbK = n\bI_T - \sum_{i=1}^n (\bI_T \otimes (\be_i^\top \bX)) \bGamma (\bI_T \otimes \bX^\top) \calS \calD (\bI_T \otimes (\bX^\top \be_i))$. 
Using similar argument that proves \Cref{lem:Theta-bound}, we can obtain the following bound of $\cbTheta_1, \cbTheta_2, \cbTheta_4$.
\begin{align}
	\max_{k\in\{1,2\}} 
	\E [\fnorm{\cbTheta_k}^2 \mid \bep]  &\le n C(T, \gamma, \eta_{\max}, c_0) (\delta^2 + \norm{\bb^*}^2),\\
	\E [\fnorm{\cbTheta_4} \mid \bep]  &\le n^{1/2} C(T, \gamma, \eta_{\max}, c_0) (\delta^2 + \norm{\bb^*}^2). 
\end{align}

By the definitions of $\cbTheta_1, \cbTheta_2, \cbTheta_4$, we have
\begin{align*}
	(\fnorm{\bE}/\sqrt{n} + 1)[n^{-1}\cbTheta_1 \hbK^\top + \cbTheta_2 + \cbTheta_4]
	= n^{-1}\bR^\top \bF (\hbA - \hbK\bW)^\top.
\end{align*}
Therefore, conditional on $\bep$, we have
\begin{align*}
	\E [\fnorm{\bR^\top \bF (\hbA - \hbK\bW)^\top} \mid \bep]
        &= (\fnorm{\bE}/\sqrt{n} + 1) \E [\fnorm{\cbTheta_1 \hbK^\top + n\cbTheta_2 + n\cbTheta_4} \mid \bep]\\
        &\le n^{3/2} (\fnorm{\bE}/\sqrt{n} + 1) 
	C(T, \gamma, \eta_{\max}, c_0) (\delta^2 + \norm{\bb^*}^2).
\end{align*}
This finishes the proof of \Cref{lem:W-W}.
\end{proof}

\subsection{Proof of \Cref{thm:unknwon-Sigma}}
\label{proof-thm:unknwon-Sigma}
In the event $\Omega_*=\Omega_1 \cap\Omega_2$, we know from \Cref{lem:hbK-opnorm} that $\hbK$ is invertible and $\opnorm{\hbK^{-1}} \le n^{-1} C$. 
Define $\tbW = \hbK^{-1} \hbA$. 
Using $\bR + \bF\tbW^\top = \bR + \bF \bW^\top + \bF(\tbW-\bW)^\top$, we have 
\begin{align*}
	&(\bR + \bF\tbW^\top)^\top (\bR + \bF\tbW^\top) - 
	(\bR + \bF\bW^\top)^\top (\bR + \bF\bW^\top)\\
	=~& (\bR + \bF \bW^\top)^\top \bF(\tbW-\bW)^\top 
	+ (\tbW-\bW)\bF^\top (\bR + \bF \bW^\top) + (\tbW-\bW)\bF^\top \bF (\tbW-\bW)^\top.
\end{align*}
We have by the triangle inequality
\begin{align*}
	&~\fnorm{(\bR + \bF\tbW^\top)^\top (\bR + \bF\tbW^\top) - 
	(\bR + \bF\bW^\top)^\top (\bR + \bF\bW^\top)}\\
	\le&~ 2 \fnorm{(\bR + \bF \bW^\top)^\top \bF(\tbW-\bW)^\top}
	+ \fnorm{(\tbW-\bW)\bF^\top \bF (\tbW-\bW)}\\ 
	\lesssim&~  \fnorm{\bR^\top \bF(\tbW-\bW)^\top} + \fnorm{\bW \bF^\top \bF(\tbW-\bW)^\top} 
	+ \fnorm{(\tbW-\bW)\bF^\top \bF (\tbW-\bW)}.
\end{align*} 
Recall that in the event $\Omega_*$, we have $\opnorm{\hbK^{-1}} \le n^{-1} C$. 
Using $ \hbA - \hbK \bW = \hbK (\tbW - \bW)$ and 
\Cref{lem:W-W}, we have
\begin{align*}
	&\E \Bigl[I(\Omega_*)
	\fnorm{(\bR + \bF\tbW^\top)^\top (\bR + \bF\tbW^\top) - 
	(\bR + \bF\bW^\top)^\top (\bR + \bF\bW^\top)} \bigm| \bep \Bigr]\\
	\lesssim~&  \E \Bigl[I(\Omega_*)\fnorm{\bR^\top \bF(\tbW-\bW)^\top} \bigm| \bep \Bigr]\\
	&+ \E\Bigl[I(\Omega_*)\fnorm{\bW \bF^\top \bF(\tbW-\bW)^\top} \bigm| \bep \Bigr]\\
	&+ \E\Bigl[I(\Omega_*)\fnorm{(\tbW-\bW)\bF^\top \bF (\tbW-\bW)^\top} \bigm| \bep \Bigr]\\
	=~&  \E \Bigl[I(\Omega_*)\fnorm{\bR^\top \bF(\hbA - \hbK \bW)^\top (\hbK^\top)^{-1}} \bigm| \bep \Bigr]\\
	&+ \E\Bigl[I(\Omega_*)\fnorm{\bW \bF^\top \bF(\hbA - \hbK \bW)^\top (\hbK^\top)^{-1}} \bigm| \bep \Bigr]\\
	&+ \E\Bigl[I(\Omega_*)\fnorm{\hbK^{-1} (\hbA - \hbK \bW) \bF^\top \bF (\hbA - \hbK \bW)^\top (\hbK^\top)^{-1}} \bigm| \bep \Bigr].
\end{align*}
According to \Cref{lem:W-W} and the bound of $\opnorm{\hbK^{-1}}$ in \Cref{lem:hbK-opnorm}, the first conditional expectation is bounded from above by 
$$n (\tfrac{\norm{\bep}}{n} + \tfrac{1}{\sqrt{n}})
C(T, \gamma, \eta_{\max}, c_0) (\delta^2 + \norm{\bb^*}^2).$$

Using the bound of $\opnorm{\bK^{-1}}$ in \Cref{lem:hbK-opnorm}, 
the bound of $\opnorm{\bW}$ in \Cref{lem:A-K-W-opnorm}, and the bound of 
$\E [\fnorm{\bF^\top \bF (\hbA - \hbK\bW)^\top}\mid \bep]$ in \Cref{lem:W-W}, the second conditional expectation is bounded from above by 
$$
n^{1/2} C(T, \gamma, \eta_{\max}, c_0) (\delta^2 + \norm{\bb^*}^2).
$$

Similarly, the third conditional expectation is bounded from above by
\begin{align*}
	n^{1/2} C(T, \gamma, \eta_{\max}, c_0) (\delta^2 + \norm{\bb^*}^2).
\end{align*}
In summary, we have 
\begin{align*}
	&n^{-1}\E [I(\Omega_*)
	\fnorm{(\bR + \bF\tbW^\top)^\top (\bR + \bF\tbW^\top) - 
	(\bR + \bF\bW^\top)^\top (\bR + \bF\bW^\top)} \mid \bep]\\
	\le~& \tfrac{1}{\sqrt{n}} \bigl(\tfrac{\norm{\bep}}{\sqrt{n}} + 1 \bigr) C(T, \gamma, \eta_{\max}, c_0) (\delta^2 + \norm{\bb^*}^2).
\end{align*}
Since $\tilde r_t$ and $\hat r_t$ are the $t$-th diagonal entries of $(\bR + \bF\tbW^\top)^\top (\bR + \bF\tbW^\top)/n$ and $(\bR + \bF\bW^\top)^\top (\bR + \bF\bW^\top)/n$, respectively, we have
\begin{align*}
	&\E[I(\Omega_*)|\tilde r_t - \hat r_t| \mid \bep]\\
	\le~& \E\Bigl[I(\Omega_*)\fnorm{(\bR + \bF\tbW^\top)^\top (\bR + \bF\tbW^\top) - 
	(\bR + \bF\bW^\top)^\top (\bR + \bF\bW^\top)}\bigm| \bep \Bigr]\\
	\le~& \tfrac{1}{\sqrt{n}} \bigl(\tfrac{\norm{\bep}}{\sqrt{n}} + 1 \bigr)  C(T, \gamma, \eta_{\max}, c_0) (\delta^2 + \norm{\bb^*}^2).
\end{align*}
Using the same argument in the proof of \Cref{thm:using-Sigma}, we have for any $\epsilon > 0$, 
\begin{align*}
	\P\Bigl(
	I(\Omega_*) |\tilde r_t - \hat r_t| > \epsilon \mid \bep
	\Bigr)
        &\le \min\Bigl(1, \tfrac{C(T, \gamma, \eta_{\max}, c_0, \delta)}{\epsilon} (\tfrac{1}{\sqrt{n}} + \tfrac{\norm{\bep}}{n})\Bigr)\\
        &\le
        \max\{1,
            \tfrac{C(T, \gamma, \eta_{\max}, c_0, \delta)}{\epsilon}
        \}
        \min(1, \tfrac{1}{\sqrt{n}} + \tfrac{\norm{\bep}}{n}).
\end{align*}
Taking expectation with respect to $\bep$, we have
\begin{align*}
	&\P\Bigl(
	I(\Omega_*) |\tilde r_t - \hat r_t| > \epsilon
	\Bigr)
	\le \max\{1,
        \tfrac{C(T, \gamma, \eta_{\max}, c_0, \delta)}{\epsilon} 
        \}
        \E[\min(1, \tfrac{1}{\sqrt{n}} + \tfrac{\norm{\bep}}{n})].
\end{align*}
Using the union bound and $\P(\Omega_*)\ge 1-e^{-n/18}-e^{-n/2}\ge 1- 2 e^{-n/18}$, we obtain 
\begin{align*}
	&\P\Bigl(
	|\tilde r_t - \hat r_t| > \epsilon
	\Bigr)
	\le 2e^{-n/18} +
        \max\{1,
        \tfrac{C(T, \gamma, \eta_{\max}, c_0, \delta)}{\epsilon}
        \}
        \bigl[\tfrac{1}{\sqrt{n}} + \E [\min(1, \tfrac{\norm{\bep}}{n} )] \bigr].
\end{align*}
Using the triangle inequality and the tail probability of $|\hat r_t - r_t|$ in \Cref{thm:using-Sigma}, we have
\begin{align*}
	\P\Bigl(
	|\tilde r_t - r_t| > \epsilon
	\Bigr)
	\le 2e^{-n/18} +
        \max\{1,
        \tfrac{C(T, \gamma, \eta_{\max}, c_0, \delta)}{\epsilon}
        \}
        \bigl[\tfrac{1}{\sqrt{n}} + \E [\min(1, \tfrac{\norm{\bep}}{n} )] \bigr].
\end{align*}

\section{Proof of results in \Cref{sec:derivative}}

\subsection{Proof of \Cref{lem:dot-b}}
\label{proof-lem:dot-b}
By assumption, we know $\hbb^1$ is independent of $\bX$ and 
$\hbb^{t+1} 
= \bphi_{t} \bigl( \hbb^t + \tfrac{\eta_t}{|I_t|} \bX^\top \bS_t \bpsi(\by-\bX\hbb^t)
\bigr)$ from \eqref{eq:unified}. 
Recall that $\bD_t = \pdv{\bpsi(\bu)}{\bu}\big|_{\bu=\by-\bX\hbb^t}$ and $\tbD_t = \pdv{\bphi_t(\bv)}{\bv}\big|_{\bv = \hbb^t + \tfrac{\eta_t}{|I_t|} \bX^\top \bS_t\bpsi(\by-\bX\hbb^t)}$.
Let a dot denote the derivative with respect to $x_{ij}$. By product rule and chain rule and using 
$\by - \bX\hbb^t = \bep - \bX (\hbb^t-\bb^*)$, 
we have
\begin{align*}
	\dot \bb^{t+1} 
	&= \tbD_t \Bigl[\dot\bb^t + \frac{\eta_t}{|I_t|} 
	\Bigl(\dot \bX^\top \bS_t\bpsi(\by - \bX\hbb^t)
	- \bX^\top \bS_t \bD_t (\dot\bX (\hbb^t-\bb^*)+ 
	\bX\dot\bb^t)
	\Bigr)\Bigr] \\
	&= \tbD_t \Bigl[\dot\bb^t + \frac{\eta_t}{|I_t|} 
	\Bigl(\dot \bX^\top F_t
	- \bX^\top \bS_t \bD_t (\dot\bX H_t + 
	\bX\dot\bb^t)
	\Bigr) \Bigr],
\end{align*}
where the last line uses $F_t = \bS_t\bpsi(\by - \bX\hbb^t)$ and $H_t = \hbb^t-\bb^*$. 
Arranging terms gives
\begin{align*}
	- \tbD_t (\bI_p - \tfrac{\eta_t}{|I_t|} \bX^\top \bS_t \bD_t \bX)\dot\bb^t
	+ 
	\dot \bb^{t+1} 
	= \tfrac{\eta_t}{|I_t|} \tbD_t
	(\dot \bX^\top F_t - \bX^\top \bS_t \bD_t \dot\bX H_t).
\end{align*}
Let  $\bP_t = \tbD_t (\bI_p - \tfrac{\eta_t}{|I_t|} \bX^\top \bS_t \bD_t \bX)$ and 
$\ba_t= \tfrac{\eta_t}{|I_t|} \tbD_t
(\dot \bX^\top F_t - \bX^\top \bS_t \bD_t \dot\bX H_t)$, we can rewrite the above recursion of $\dot \bb^t$ as a linear system:
\begin{align*}
	\underbrace
	{\begin{bmatrix}
		\bI_p&&&&\\
		-\bP_1&\bI_p&&&\\
		&\ddots&\ddots&&\\
		&&-\bP_{T-1}&\bI_p&\\
	\end{bmatrix}}_{\calM}
	\begin{bmatrix}
		\dot\bb^1\\
		\dot\bb^2\\
		\vdots\\
		\dot\bb^T
	\end{bmatrix}
	= \underbrace{
	\begin{bmatrix}
		\boldzero&&&&\\
		\bI_p&\boldzero&&&\\
		&\ddots&\ddots&&\\
		&&\bI_p&\boldzero&\\
	\end{bmatrix}}_{\bL}
	\underbrace
	{\begin{bmatrix}
		\ba_1\\
		\ba_2\\
		\vdots\\
		\ba_T
	\end{bmatrix}}_{\ba}.
\end{align*}

Solving the above system, we have 
$\dot\bb^t
= (\be_t^\top \otimes \bI_p) \calM^{-1}\bL \ba.$
Since $\dot\bX = \pdv{\bX}{x_{ij}} = \be_i \be_j^\top$, $\ba_t$ can be further simplified as 
$$\ba_t= \tfrac{\eta_t}{|I_t|} \tbD_t
\bigl(\be_j\be_i^\top F_t -  \bX^\top \bS_t \bD_t \be_i\be_j^\top H_t
\bigr).$$
Using $\calD = \sum_{t=1}^T \bigl((\be_t\be_t^\top) \otimes \bD_t\bigr)$, $\tcalD = \sum_{t=1}^T \bigl((\be_t\be_t^\top) \otimes \tbD_t\bigr)$, $\calS = \sum_{t=1}^T \bigl((\be_t\be_t^\top) \otimes \bS_t\bigr)$, 
and 
$\bLambda = \sum_{t=1}^T \tfrac{\eta_t}{|I_t|}\be_t\be_t^\top$, 
we have 
\begin{align*}
	\ba 
	&= (\bLambda \otimes \bI_p) \tcalD  \bigl[\vec(\be_j\be_i^\top \bF) 
	- (\bI_T\otimes \bX^\top) \calS \calD \vec(\be_i\be_j^\top\bH)\bigr]\\
	&= (\bLambda \otimes \bI_p) \tcalD  
	\bigl[((\bF^\top\be_i) \otimes\be_j)
	- (\bI_T\otimes \bX^\top) \calS \calD 
	((\bH^\top\be_j)\otimes \be_i)\bigr].
\end{align*}
Plugging this expression for $\ba$ into $\dot\bb^t
= (\be_t^\top \otimes \bI_p) \calM^{-1}\bL \ba$ gives 
\begin{equation*}
	\pdv{\hbb^t}{x_{ij}}
	= (\be_t^\top \otimes \bI_p) \calM^{-1}\bL (\bLambda\otimes \bI_p) \tcalD  [((\bF^\top\be_i) \otimes\be_j)
	- (\bI_T\otimes \bX^\top) \calS \calD 
	((\bH^\top\be_j)\otimes \be_i)].
\end{equation*}
This finishes the proof of \Cref{lem:dot-b}.

\subsection{Proof of \Cref{lem:dF-dx}}
\label{proof-lem:dF-dx}
By definition, $F_{lt} = \be_l^\top \bS_t\bpsi(\by - \bX \hbb^t)$. By the chain rule of differentiation, we have
\begin{align*}
	\pdv{F_{lt}}{x_{ij}}
	= \be_l^\top \pdv{\bS_t\bpsi(\by - \bX \hbb^t)}{x_{ij}}
	= - \be_l^\top \bS_t \bD_t (\be_i \be_j^\top \bH \be_t + \bX \pdv{\hbb^t}{x_{ij}}). 
\end{align*}
Notice that $(\be_t\otimes (\bD_t\bS_t)) = \calD \calS (\be_t \otimes \bI_n)$. 
The desired formula then follows by plugging in the expression of $\pdv{\hbb^t}{x_{ij}}$ in \Cref{lem:dot-b}.

\subsection{Proof of \Cref{lem:dtF-dx}}
\label{proof-lem:dtF-dx}
The desired identity follows by 
\begin{align*}
	\pdv{\tF_{lt}}{x_{ij}}
	= \be_l^\top \pdv{\bpsi(\by - \bX \hbb^t)}{x_{ij}}
	= - \be_l^\top \bD_t (\be_i \be_j^\top \bH \be_t + \bX \pdv{\hbb^t}{x_{ij}})
\end{align*}
and the expression of $\pdv{\hbb^t}{x_{ij}}$ in \Cref{lem:dot-b}.

\subsection{
Alternative expressions for
the matrices defined in \Cref{def:identity-sum-pdv}}
\label{proof-def:identity-sum-pdv}

This section derives
alternative expressions for
the matrices $\bUpsilon_1,...,\bUpsilon_5$ defined in \Cref{def:identity-sum-pdv}.

We first study $\bUpsilon_1$ in \eqref{eq:1}. 
Using $\bF = \sum_{t=1}^T\bF\be_t\be_t^\top$, we have by \Cref{lem:dF-dx}
\begin{align}\label{eq:1-two-terms}
	\sum_{j=1}^n \pdv{\be_i^\top \bF}{x_{ij}}
	= \sum_{i=1}^n \sum_{t=1}^T \pdv{F_{it}}{x_{ij}}\be_t^\top 
	= \sum_{i=1}^n \sum_{t=1}^T D_{ij}^{it} \be_t^\top 
	+ \sum_{i=1}^n \sum_{t=1}^T \Delta_{ij}^{it}\be_t^\top.
\end{align}
Now we compute the two terms in the right-hand side of \eqref{eq:1-two-terms}. 
For the first term, using the expression of $D_{ij}^{lt}$ in \Cref{lem:dF-dx}, 
\begin{align*}
	&\sum_{i,t} D_{ij}^{it}\be_t^\top\\
	=& - \be_j^\top \bH \sum_{t=1}^T \trace(\bS_t \bD_t) \be_t \be_t^\top 
	+ \sum_{i,t}
	((\be_j^\top \bH)\otimes \be_i^\top) \calD \calS (\bI_T\otimes \bX)
	\bGamma^\top
	(\be_t \otimes (\bX^\top \bD_t \bS_t \be_i))
	\be_t^\top\\
	=& - \be_j^\top \bH \sum_{t=1}^T \trace(\bS_t \bD_t) \be_t \be_t^\top
	+ \be_j^\top \bH 
	\sum_{i=1}^n (\bI_T \otimes \be_i^\top) \calD \calS (\bI_T\otimes \bX) \bGamma^\top  (\bI_T\otimes \bX^\top)  \calD \calS (\bI_T \otimes \be_i)\\
	=& -\be_j^\top \bH \tbK^\top &&\mbox{by \eqref{eq:tbK}}.
\end{align*}
Next, we compute the second term in the right-hand side of \eqref{eq:1-two-terms}. Using the expression of $\Delta_{ij}^{lt}$ in \Cref{lem:dF-dx},
\begin{align*}
	\sum_{i,t} \Delta_{ij}^{it}\be_t^\top
        &= - \sum_{i,t} 
	((\be_i^\top\bF)\otimes \be_j^\top)
	\bGamma^\top 
	(\be_t\otimes (\bX^\top\bD_t \bS_t \be_i)) \be_t^\top\\
        &= -\be_j^\top 
	\underbrace{\sum_{i}  
	((\be_i^\top\bF)\otimes \bI_p)
	\bGamma^\top
	(\bI_T \otimes \bX^\top)\calD \calS (\bI_T \otimes \be_i)}_{\bUpsilon_1}.
\end{align*}
The identity \eqref{eq:1} then follows by substituting the above two expressions into \eqref{eq:1-two-terms}.

To study $\bUpsilon_2$ in \eqref{eq:2}, we use a similar procedure. 
Using the mixed property of Kronecker product and the fact that the transpose of a scalar remains the same, we have 
\begin{align*}
	\sum_{j=1}^p 
	\pdv{\be_j^\top \bH}{x_{ij}}
        = \sum_{j=1}^p\sum_{t=1}^T
	\pdv{\be_j^\top \hbb^t}{x_{ij}} \be_t^\top
        &= \sum_{j,t} \be_j^\top (\be_t^\top \otimes \bI_p) \bGamma ((\bF^\top\be_i) \otimes\be_j) \be_t^\top \\
	& \quad - \sum_{j,t} \be_j^\top  (\be_t^\top \otimes \bI_p) \bGamma 
	(\bI_T\otimes \bX^\top) \calS \calD 
	((\bH^\top\be_j)\otimes \be_i)\be_t^\top 
	&&\mbox{by \eqref{eq:db-dx}}\\
        &= \be_i^\top \bF 
	\sum_{j} (\bI_T \otimes \be_j^\top) \bGamma^\top (\bI_T \otimes \be_j) \\
        & \quad  - \sum_{j,t} 
	((\be_j^\top \bH)\otimes \be_i^\top)
	\calD \calS 
	(\bI_T\otimes \bX)
	\bGamma^\top 
	(\be_t \otimes \be_j) 
	\be_t^\top \\
        &= \be_i^\top \bF 
	\sum_{j} (\bI_T \otimes \be_j^\top) \bGamma^\top (\bI_T \otimes \be_j) \\
        & \quad - \be_i^\top 
	\sum_{j} 
	((\be_j^\top \bH)\otimes \bI_n)
	\calD \calS
	(\bI_T\otimes \bX)
	\bGamma^\top 
	(\bI_T \otimes \be_j)\\
        &= \be_i^\top \bF \bW^\top - \be_i^\top \bUpsilon_2,
\end{align*}
where $\bW = \sum_{j} (\bI_T \otimes \be_j^\top) \bGamma (\bI_T \otimes \be_j)$
and $\bUpsilon_2 = \sum_{j} 
((\be_j^\top \bH)\otimes \bI_n)
\calD \calS
(\bI_T\otimes \bX)
\bGamma^\top 
(\bI_T \otimes \be_j)$. 

To study $\bUpsilon_3$ in \eqref{eq:3}, we use the product rule of differentiation and \eqref{eq:1}-\eqref{eq:2}:
\begin{align*}
	\sum_{i=1}^n\sum_{j=1}^p \pdv{\bF^\top \be_i\be_j^\top \bH}{x_{ij}}
	=& \sum_{i=1}^n\sum_{j=1}^p \pdv{\bF^\top }{x_{ij}} \be_i\be_j^\top \bH + \bF^\top \sum_{i=1}^n\sum_{j=1}^p  \be_i \pdv{\be_j^\top \bH}{x_{ij}}\\
	=& -\sum_{j} (\tbK\bH^\top \be_j + \bUpsilon_1^\top \be_j) \be_j^\top \bH + \bF^\top \sum_{i} \be_i  (\be_i^\top \bF \bW^\top - \be_i^\top \bUpsilon_2)\\
	=& -\tbK\bH^\top\bH + \bF^\top \bF \bW^\top - \underbrace{(\bUpsilon_1^\top \bH + \bF^\top \bUpsilon_2)}_{\bUpsilon_3}.
\end{align*}

For $\bUpsilon_4$ in \eqref{eq:4}, by the product rule, 
\begin{align*}
	&\sum_{i=1}^n\sum_{j=1}^p \pdv{\bH^\top \bX^\top \be_i\be_j^\top \bH}{x_{ij}}\\
	=& \sum_{i,j} \pdv{\bH^\top }{x_{ij}} \bX^\top \be_i\be_j^\top \bH+ \bH^\top \sum_{i,j}  
	\be_j\be_i^\top
	\be_i\be_j^\top \bH + \bH^\top \bX^\top\sum_{i,j} \be_i\be_j^\top  \pdv{\bH}{x_{ij}}\\
	=& \sum_{i,j} \pdv{\bH^\top }{x_{ij}} \bX^\top \be_i\be_j^\top \bH+ n \bH^\top \bH + \bH^\top \bX^\top (\bF \bW^\top - \bUpsilon_2) &&\mbox{by \eqref{eq:2}}.
\end{align*}
We then compute the first term of the last line as follows
\begin{align*}
	&\sum_{i,j} \pdv{\bH^\top }{x_{ij}} \bX^\top \be_i\be_j^\top \bH\\
	=& \sum_{i,j,t} \be_t \be_t^\top \pdv{\bH^\top }{x_{ij}} \bX^\top \be_i\be_j^\top \bH\\ 
	=& \sum_{i,j,t} \be_t 
	\Bigl(\pdv{\hbb^t}{x_{ij}}\Bigr)^\top \bX^\top \be_i\be_j^\top \bH\\
	=& \sum_{i,j,t} \be_t 
	 \be_i^\top \bX \pdv{\hbb^t}{x_{ij}}
	\be_j^\top \bH\\
	=& - \sum_{i,j,t} \be_t \be_i^\top \bX
	(\be_t^\top \otimes \bI_p) \bGamma (\bI_T\otimes \bX^\top) \calS \calD 
	((\bH^\top\be_j)\otimes \be_i)
	\be_j^\top \bH + \tbUpsilon_1 &&\mbox{by \eqref{eq:db-dx}}\\
	=& - \sum_{i} 
	(\bI_T \otimes (\be_i^\top \bX)) \bGamma (\bI_T\otimes \bX^\top) \calS \calD 
	(\bI_T \otimes \be_i)
	\bH^\top \bH + \tbUpsilon_1\\
	=& - \tbA \bH^\top \bH + \tbUpsilon_1 &&\mbox{by \eqref{eq:tbA},}
\end{align*}
where 
\begin{align*}
	\tbUpsilon_1 
	= \sum_{i,j,t} \be_t \be_i^\top \bX (\be_t^\top \otimes \bI_p) \bGamma ((\bF^\top\be_i) \otimes\be_j)
	\be_j^\top \bH
	= \sum_{i} 
	(\bI_T \otimes \be_i^\top \bX ) \bGamma((\bF^\top\be_i) \otimes\bI_p)\bH. 
\end{align*}
Combining the above pieces, we have 
\begin{align*}
	\sum_{i=1}^n\sum_{j=1}^p \pdv{\bH^\top \bX^\top \be_i\be_j^\top \bH}{x_{ij}}
        &= - \tbA \bH^\top \bH + \tbUpsilon_1 + n \bH^\top \bH + \bH^\top \bX^\top (\bF \bW^\top - \bUpsilon_2)\\
        &= (n\bI_T- \tbA) \bH^\top \bH + \bH^\top \bX^\top \bF \bW^\top - \underbrace{
	(\bH^\top \bX^\top \bUpsilon_2 - \tbUpsilon_1)}_{\bUpsilon_4}. 
\end{align*}
This provides an alternative expression for $\bUpsilon_4$ in \eqref{eq:4}. 

Last, we study $\bUpsilon_5$ in \eqref{eq:5}. We have
\begin{align*} 
	\sum_{i=1}^n\sum_{j=1}^p \pdv{\bF^\top \be_i\be_j^\top\bX^\top \tbF}{x_{ij}}
        &=\sum_{i,j} 
	\Bigl[\pdv{\bF^\top \be_i}{x_{ij}} \be_j^\top\bX^\top \tbF
	+  \bF^\top \be_i\be_j^\top \be_j\be_i^\top \tbF 
	+ \bF^\top \be_i\be_j^\top\bX^\top \pdv{\tbF}{x_{ij}}\Bigr]\\
        &= -(\tbK\bH^\top + \bUpsilon_1^\top) \bX^\top \tbF
	+ p \bF^\top  \tbF 
	+ \sum_{i,j}
	\bF^\top \be_i\be_j^\top\bX^\top \pdv{\tbF}{x_{ij}}.
\end{align*}
It remains to compute the third term in the last display. 
Using the fact that $\tbF = \sum_{l=1}^n \sum_{t=1}^T \be_l\be_l^\top \tbF \be_t \be_t^\top$, we have by \Cref{lem:dtF-dx} that
\begin{align*}
	\pdv{\tbF}{x_{ij}}
	= \sum_{l,t} \be_l  \pdv{\be_l^\top \tbF \be_t}{x_{ij}}  \be_t^\top
	= \sum_{l,t} \be_l  (\tD_{ij}^{lt} + \tDelta_{ij}^{lt}) \be_t^\top.
\end{align*}
Using 
\begin{align*}
	\tD_{ij}^{lt} 
	&= -\be_l^\top \bD_t \be_i \be_j^\top \bH \be_t 
	+ 
	((\be_j^\top \bH)\otimes \be_i^\top)\calD \calS (\bI_T\otimes \bX)
	\bGamma^\top
	(\bI_T\otimes \bX^\top) \calD 
	(\be_t \otimes \be_l),\\
	\tDelta_{ij}^{lt} 
	&= -
	((\be_i^\top\bF)\otimes \be_j^\top)
	\bGamma^\top 
	(\bI_T\otimes \bX^\top) \calD 
	(\be_t \otimes \be_l),
\end{align*}
we find
\begin{align*}
	\sum_{i,j}
	\bF^\top \be_i\be_j^\top\bX^\top \pdv{\tbF}{x_{ij}}
	&=\sum_{i,j,l,t}
	\bF^\top \be_i\be_j^\top\bX^\top \be_l \pdv{\be_l^\top\tbF\be_t}{x_{ij}} \be_t^\top\\
	&=\sum_{i,j,l,t}
	\bF^\top \be_i\be_j^\top\bX^\top \be_l \tilde\Delta_{ij}^{lt} \be_t^\top
	+ 
	\sum_{i,j,l,t}
	\bF^\top \be_i\be_j^\top\bX^\top \be_l \tilde D_{ij}^{lt} \be_t^\top.
\end{align*}
We now compute the two terms in the above display. 
For the first term, we have 
\begin{align*}
	&\sum_{i,j,l,t}
	\bF^\top \be_i\be_j^\top\bX^\top \be_l \tilde\Delta_{ij}^{lt} \be_t^\top\\
	=& \sum_{i,j,l,t}
	\bF^\top \be_i\be_j^\top\bX^\top \be_l \bigl[-
	((\be_i^\top\bF)\otimes \be_j^\top)
	\bGamma^\top 
	(\bI_T\otimes \bX^\top) \calD 
	(\be_t \otimes \be_l)\bigr] \be_t^\top\\
	=& - \bF^\top \bF \sum_{l=1}^n
	(\bI_T \otimes (\be_l^\top \bX))
	\bGamma^\top 
	(\bI_T \otimes \bX^\top) \calD
	(\bI_T \otimes \be_l)\\
	=& - \bF^\top \bF \hbA^\top &&\mbox{by \eqref{eq:hbA}.}
\end{align*}
For the second term, we have 
\begin{align*}
	&\sum_{i,j,l,t}
	\bF^\top \be_i\be_j^\top\bX^\top \be_l \tilde D_{ij}^{lt} \be_t^\top\\
	=& \sum_{i,j,l,t}
	\bF^\top \be_i\be_j^\top\bX^\top \be_l [-\be_l^\top \bD_t \be_i \be_j^\top \bH \be_t 
	+ 
	((\be_j^\top \bH)\otimes \be_i^\top)\calD \calS (\bI_T\otimes \bX)
	\bGamma^\top 
	(\be_t \otimes (\bX^\top \bD_t \be_l))]\be_t^\top\\
	=& \underbrace{- \sum_{t=1}^T \bF^\top \bD_t \bX\bH\be_t 
	+
	\sum_{l=1}^n (\be_l^\top \bX \bH \otimes \bF^\top) \calD \calS (\bI_T\otimes \bX)
	\bGamma^\top
	(\bI_T\otimes \bX^\top)
	\calD 
	(\bI_T\otimes \be_l)}_{\tbUpsilon_2}.
\end{align*}
Thus, we have  established
\begin{align*}
	\sum_{i=1}\sum_{j=1}^p \pdv{\bF^\top \be_i\be_j^\top\bX^\top \tbF}{x_{ij}}
	= -\tbK\bH^\top \bX^\top \tbF + 
	p \bF^\top \tbF - \bF^\top \bF \hbA^\top - 
	\underbrace{
	(\bUpsilon_1^\top \bX^\top \tbF - \tbUpsilon_2)}_{\bUpsilon_5}.
\end{align*}
This provides an alternative expression for $\bUpsilon_5$ in \eqref{eq:5}.

\subsection{Preparation results for proving \Cref{lem:Upsilon-bound,lem:Theta-bound}}

\begin{lemma}[Moment bounds for $\bH,\bF,\tbF$]
	\label{lem:HF-moment}
	Under \Cref{assu:X,assu:regime,assu:rho-1} with $\bSigma=\bI_p$. 
	Let $\bH, \bF$ be defined in \eqref{eq:FHR}, we have for any finite integer $k$,
	\begin{align*}
		&\E [\opnorm{\bX/\sqrt{n}}^{2k}] 
		\le C(\gamma, k),\\
		&\E [\fnorm{\bH}^{2k} \mid \bep] 
		\le C(T, \gamma, c_0, \eta_{\max},k) (\delta^2 + \norm{\bb^*})^{2k},\\
		&\E[\fnorm{\bF/\sqrt{n}}^{2k}\mid \bep] 
		\le \E[\fnorm{\tbF/\sqrt{n}}^{2k}\mid \bep] \le C(T, k) \delta^{2k}. 
	\end{align*}
\end{lemma}
\begin{proof}[Proof of \Cref{lem:HF-moment}]
For the first inequality, according to \Cref{assu:X} and $\bSigma=\bI_p$, $\bX$ has \iid standard normal entries. 
By \cite[Theorem II.13]{DavidsonS01}, there exists a random variable $z\sim N(0,1)$ such that 
$\opnorm{\bX} \le \sqrt{n} + \sqrt{p} + z$ almost surely. 
Under \Cref{assu:regime} that 
$p/n\le \gamma$, we have 
$\E [\opnorm{\bX/\sqrt{n}}^k] \le C(\gamma, k) $ for any finite integer $k$.

For the second inequality, since $\fnorm{\bH}^2 = \sum_{t=1}^T \norm{\hbb^t - \bb^*}^2$, it suffices to bound $\norm{\hbb^t - \bb^*}^2$ for each $t\in[T]$.
	Define the sequence of scalars 
	$a_t \defas 
	\max \{\norm{\hbb^t}, \delta\}$.
	Since 
	$\hbb^t = \bphi_{t-1}(\hbb^{t-1} + \tfrac{\eta_t}{n_t} \bX^\top \bS_t \bpsi(\by - \bX\hbb^{t-1}))$ where $n_t:= |I_t|$.
	Note that $\|\bpsi(\by_{I_t} - \bX_{I_t}\bb)\| \le \sqrt{|I_t|} \delta$ by \Cref{assu:rho-1}, we have 
	\begin{align*}
		\norm{\hbb^t - \boldzero} 
		&= \norm{\hbb^t - \bphi_{t-1}(\boldzero)} &&\mbox{since $\bphi_{t-1} = \boldzero$}\\
		&\le \norm{\hbb^{t-1} + \tfrac{\eta_t}{|I_t|} \bX^\top \bS_t \bpsi(\by - \bX\hbb^{t-1})} &&\mbox{since $\bphi_t$ is 1-Lipschitz}\\
		&= \norm{\hbb^{t-1} + \tfrac{\eta_t}{|I_t|} \bX_{I_t}^\top \bpsi(\by_{I_t} - \bX_{I_t}\hbb^{t-1})}\\
		&\le \norm{\hbb^{t-1}} + \tfrac{\eta_t}{\sqrt{|I_t|}} \opnorm{\bX_{I_t}} \delta
		&&\mbox{by the triangle inequality}\\
		&\le \norm{\hbb^{t-1}} + \tfrac{\eta_{\max}}{\sqrt{c_0 n}} \opnorm{\bX} \delta &&\mbox{by $|I_t|\ge c_0 n$}\\
		&\le a_{t-1} + \tfrac{\eta_{\max}}{\sqrt{c_0}} \opnorm{\bX/\sqrt{n}}a_{t-1}\\
		&= (1 + \tfrac{\eta_{\max}}{\sqrt{c_0}} \opnorm{\bX/\sqrt{n}}) a_{t-1}.
	\end{align*}
Since $a_t = \max\{\norm{\hbb^t}, \delta\}$ and $\delta \le a_{t-1}$, we have  
$$
a_t \le (1 + \tfrac{\eta_{\max}}{\sqrt{c_0}} \opnorm{\bX/\sqrt{n}}) a_{t-1}.
$$
Notice $a_1 = \delta$ since $\hbb^1=\boldzero_p$, we obtain 
$$
a_t \le (1 + \tfrac{\eta_{\max}}{\sqrt{c_0}} \opnorm{\bX/\sqrt{n}})^{t-1} \delta.
$$
Hence, using the inequality $\norm{\hbb^t - \bb^*}^2 \le 2\norm{\hbb^t}^2 + 2\norm{\bb^*}^2$, we have 
\begin{align}\label{eq:moment-H}
	\fnorm{\bH}^2 
	&\lesssim \sum_{t=1}^T (\norm{\hbb^t}^2 + \norm{\bb^*}^2)\notag\\
	&\le \sum_{t=1}^T \bigl[(1 + \tfrac{\eta_{\max}}{\sqrt{c_0}} \opnorm{\bX/\sqrt{n}})^{2t-2} \delta^2 + \norm{\bb^*}^2\bigr]\notag\\
	&\le T (\delta^2 + \norm{\bb^*}^2) 
	(1 + \tfrac{\eta_{\max}}{\sqrt{c_0}} \opnorm{\bX/\sqrt{n}})^{2T}. 
\end{align}
Taking conditional expectation on both sides given $\bep$, the desired moment bound for $\bH$ follows from the moment bound for $\opnorm{\bX/\sqrt{n}}$.

For the third inequality, since $|\psi(x)|\le \delta$ from \Cref{assu:rho-1}, we have 
$\|\bpsi(\bu)\| \le n \delta^2$ for any $\bu\in \R^n$. 
By the definitions of $\bF$ and $\tbF$, we have 
\begin{align*}
	\fnorm{\bF}^2 
	\le \fnorm{\tbF}^2
	=
	\sum_{t=1}^T \norm{\bpsi(\by-\bX\hbb^t)}^2
	\le T n \delta^2.
\end{align*}
Since the above display holds for any $\bep$, it implies the desired conditional moment bounds for $\bF,\tbF$. 
\end{proof}

\begin{lemma}[Frobenius norm bound for $\bX\bH$]
	\label{lem:XH}
	Under \Cref{assu:X,assu:regime,assu:rho-1}, we have
	\begin{align*}
		\fnorm{\bX\bH}^2 
		\le C(T, \gamma, \eta_{\max}, c_0) n (\delta^2 + \norm{\bb^*}^2)
	\end{align*}
	with probability at least $1 - \exp(-n/2)$.
\end{lemma}
\begin{proof}[Proof of \Cref{lem:XH}]
	By \cite[Theorem II.13]{DavidsonS01}, under \Cref{assu:X} with $\bSigma=\bI_p$, we have
	$$\P(\opnorm{\bX/\sqrt{n}} \le 2 + \sqrt{\gamma}) \ge 1 - \exp(-n/2).$$
	Using $\fnorm{\bX\bH}^2 \le \opnorm{\bX}^2 \fnorm{\bH}^2$ and the bound
	\eqref{eq:moment-H}, we have 
	\begin{align*}
		\fnorm{\bX\bH}^2 
		\le n C(T, \gamma, \eta_{\max}, c_0) (\delta^2 + \norm{\bb^*}^2) 
	\end{align*}
	holds with probability at least $1 - \exp(-n/2)$.
\end{proof}

\begin{lemma}[Operator norm bound for $\calM$]
\label{lem:M-opnorm}
Under \Cref{assu:X,assu:regime,assu:rho-1}, we have 
\begin{align*}
	\opnorm{\calM^{-1}} \le C(T) (1 + \xi)^T,
\end{align*}
where $\xi = \frac{\eta_{\max}}{c_0 n} \opnorm{\bX}^2$.
\end{lemma}

\begin{proof}[Proof of \Cref{lem:M-opnorm}]
	By the definition of $\calM$ in \Cref{lem:dot-b}, we have
	\begin{align*}
		\calM = 
		\begin{bmatrix}
			\bI_p&&&&\\
			-\bP_1&\bI_p&&&\\
			&\ddots&\ddots&&\\
			&&-\bP_{T-1}&\bI_p&\\
		\end{bmatrix}
		\text{ where }
		\bP_t = \tbD_t(\bI_p - \tfrac{\eta_t}{|I_t|} \bX^\top \bS_t \bD_t \bX). 
	\end{align*}
Hence, we can write $\calM = \bI_{pT} - \bA$, where $\bA$ is the lower triangular matrix with off-diagonal blocks $\bP_1, ..., \bP_{T-1}$.
Using the matrix identity $(\bI-\bA)^{-1} = \sum_{k=0}^\infty \bA^k$ and noticing $\bA^k=\boldzero$ for $k\ge T$, we have  
\begin{align*}
	\calM^{-1}
	= \sum_{k=0}^{T-1} 
	\begin{bmatrix}
		\boldzero&&&&\\
		\bP_1&\boldzero&&&\\
		&\ddots&\ddots&&\\
		&&\bP_{T-1}&\boldzero&\\
	\end{bmatrix}^k.
\end{align*}
Taking operator norm on both sides, we obtain 
\begin{equation}\label{eq:mat-inverse}
	\opnorm{\calM^{-1}}
	\le \sum_{k=0}^{T-1} 
	(\sum_{t=1}^{T-1}\opnorm{\bP_t})^k.
\end{equation}
Since $\bD_t = \pdv{\bpsi(\bu)}{\bu}\big|_{\bu=\by-\bX\hbb^t}$ and $\psi$ is 1-Lipschitz, we have $\opnorm{\bD_t} \le 1$. By the definition that $\bS_t = \sum_{i\in I_t} \be_i \be_i^\top $, we know $\opnorm{\bS_t}\le1$. 
Since $|I_t|\ge c_0 n$ and $\eta_t\le \eta_{\max}$ for any $t\in[T]$, we have
\begin{align*}
	\opnorm{\bP_t} 
	\le 1 + \frac{\eta_t}{|I_t|}\opnorm{\bX_{I_t}}^2
	\le 1 + \frac{\eta_{\max}}{c_0 n} \opnorm{\bX}^2 \defas 1 + \xi.
\end{align*}
Plugging this inequality into \eqref{eq:mat-inverse} gives 
\begin{align*}
	\opnorm{\calM^{-1}}
	\le \sum_{k=0}^{T-1} 
	(T(1+\xi))^k \le C(T) (1 + \xi)^T. 
\end{align*}
\end{proof}

\begin{lemma}
\label{lem:Gamma-opnorm}
Under the same conditions as \Cref{lem:M-opnorm}, we have
\begin{align*}
	\opnorm{\bGamma} 
	\le n^{-1} C(T, \eta_{\max}, c_0)(1 + \xi)^T,
\end{align*}
where $\xi = \frac{\eta_{\max}}{c_0 n} \opnorm{\bX}^2$.
\end{lemma}

\begin{proof}[Proof of \Cref{lem:Gamma-opnorm}]
	By the definition of $\bGamma$ in \Cref{lem:dot-b}, we have 
	$\bGamma = \calM^{-1}\bL (\bLambda\otimes \bI_p) \tcalD$. 
	Notice that $\bLambda = \sum_{t=1}^T \frac{\eta_t}{|I_t|}\be_t\be_t^\top$, we have $\opnorm{\bLambda} = \max_{t\in[T]} \frac{\eta_t}{|I_t|} \le n^{-1} \frac{\eta_{\max}}{c_0}$ using $|I_t|\ge c_0 n$ and $\eta_t\le \eta_{\max}$.
	Since $\bphi$ is 1-Lipschitz, we have $\opnorm{\tcalD}\le1$. 
	By definition of $\bL$ in \Cref{lem:dot-b}, we have $\opnorm{\bL}=1$. Using these upper bounds of $\opnorm{\bL}, \opnorm{\bLambda}, \opnorm{\tcalD}$ and the upper bound of $\opnorm{\calM^{-1}}$ in \Cref{lem:M-opnorm}, we obtain
	\begin{align*}
		\opnorm{\bGamma} 
		&\le \opnorm{\calM^{-1}} \opnorm{\bL} \opnorm{\bLambda\otimes \bI_p} \opnorm{\tcalD}\\
		&\le n^{-1} C(T, \eta_{\max}, c_0)(1 + \xi)^T. 
	\end{align*}
\end{proof}

\begin{lemma}[Moment bounds for derivative of $\bH,\bF,\tbF$]
\label{lem:HF-derivative-moment}
Under \Cref{assu:X,assu:regime,assu:rho-1} and $\bSigma = \bI_p$, we have for any finite integer $k$,
\begin{align*}
	\E\Bigl[\Bigl(\sum_{i=1}^n\sum_{j=1}^p\Big\|\pdv{\bH}{x_{ij}}\Big\|_{\rm F}^{2}\Bigr)^k ~\Big|~ \bep \Bigr] 
	&\le C(T, \gamma, \eta_{\max}, c_0) (\delta^2 + \norm{\bb^*}^2)^{2k},\\
	\E\Bigl[\Bigl(\sum_{i=1}^n\sum_{j=1}^p\Big\|\pdv{\bF/\sqrt{n}}{x_{ij}}\Big\|_{\rm F}^{2}\Bigr)^k~\Big|~ \bep \Bigr] 
	&\le \E\Bigl[\Bigl(\sum_{i=1}^n\sum_{j=1}^p\Big\|\pdv{\tbF/\sqrt{n}}{x_{ij}}\Big\|_{\rm F}^{2}\Bigr)^k~\Big|~ \bep \Bigr] 
	\le C(T, \gamma, \eta_{\max}, c_0) (\delta^2 + \norm{\bb^*}^2)^{2k}.
\end{align*}
\end{lemma}

\begin{proof}[Proof of \Cref{lem:HF-derivative-moment}]
	We first prove the first bound. 
	By \Cref{lem:dot-b}, we have 
	\begin{align*}
		\pdv{\hbb^t}{x_{ij}}
		= (\be_t^\top \otimes \bI_p) 
		\bGamma
		[((\bF^\top\be_i) \otimes\be_j)
		- (\bI_T\otimes \bX^\top) \calS \calD 
		((\bH^\top\be_j)\otimes \be_i)].
	\end{align*}
	Hence, using $\pdv{\be_k^\top\bH\be_t}{x_{ij}} = \pdv{\be_k^\top \hbb^t}{x_{ij}}$, we have 
	\begin{align*}
		\pdv{\be_k^\top\bH\be_t}{x_{ij}}
		&=  (\be_t^\top \otimes \be_k^\top) 
		\bGamma
		[((\bF^\top\be_i) \otimes\be_j)
		- (\bI_T\otimes \bX^\top) \calS \calD 
		((\bH^\top\be_j)\otimes \be_i)]\\
		&= (\be_t^\top \otimes \be_k^\top) 
		\bGamma
		[(\bF^\top \otimes \bI_p) (\be_i\otimes\be_j)
		- (\bI_T\otimes \bX^\top) \calS \calD 
		(\bH^\top \otimes \bI_n)
		(\be_j\otimes \be_i)].
	\end{align*}
	Using the above equality, $\sum_{i,j,t,k} [(\be_t^\top \otimes \be_k^\top) \bA (\be_j \otimes \be_i)]^2 = \fnorm{\bA}^2$ for $\bA\in \R^{pT\times np}$, and the triangle inequality, we have 
	\begin{align*}
		\sum_{i=1}^n\sum_{j=1}^p \Big\|\pdv{\bH}{x_{ij}}\Big\|_{\rm F}^2
		&= \sum_{i=1}^n\sum_{j=1}^p \sum_{k=1}^p \sum_{t=1}^T \left(\pdv{\be_k^\top\bH\be_t}{x_{ij}}\right)^2\\
		& \lesssim
		\fnorm{\bGamma(\bF^\top \otimes \bI_p)}^2 
		+ \fnorm{\bGamma(\bI_T\otimes \bX^\top) \calS \calD(\bH^\top \otimes \bI_n)}^2\\
		&\le p \opnorm{\bGamma}^2 \fnorm{\bF}^2 
		+ n \opnorm{\bGamma}^2 \opnorm{\bX}^2 \opnorm{\calS}^2\opnorm{\calD}^2\fnorm{\bH}^2\\
		&\le p \opnorm{\bGamma}^2 \fnorm{\bF}^2 
		+ n \opnorm{\bGamma}^2 \opnorm{\bX}^2 \fnorm{\bH}^2\\
		&\le C(T,\gamma,\eta_{\max}, c_0) (1 + \xi)^{2T} \fnorm{\bF}^2/n 
		+ C(T,\gamma,\eta_{\max}, c_0) (1 + \xi)^{2T} \opnorm{\bX}^2/n \fnorm{\bH}^2, 
	\end{align*}
where the last inequality uses \Cref{lem:Gamma-opnorm}.
Taking the conditional expectation on both sides given $\bep$, the desired moment bound follows from the moment bounds of $\bX,\bH,\bF$ in \Cref{lem:HF-moment}.

Now we prove the second bound.
By definition, the $t$-th column of $\bF$ is $F_t = \bS_t\bpsi(\by - \bX\hbb^t)$, it can be written using the $t$-th column of $\tbF$ as 
$F_t = \bS_t \tilde F_t$. 
Since $\opnorm{\bS_t}\le 1$, we have 
$$\fnorm{\pdv{\bF}{x_{ij}}}^2
= \sum_{t}\norm{\pdv{F_t}{x_{ij}}}^2
\le \sum_{t}\norm{\pdv{\tilde F_t}{x_{ij}}}^2 
= \fnorm{\pdv{\tbF}{x_{ij}}}^2.
$$
By \Cref{lem:dtF-dx}, we have
$\pdv{\be_l^\top \tbF \be_t}{x_{ij}}
	= \tD_{ij}^{lt} + \tDelta_{ij}^{lt}
$
where 
\begin{align*}
	\tD_{ij}^{lt} 
	&= -\be_l^\top \bD_t \be_i \be_j^\top \bH \be_t 
	+ 
	((\be_j^\top \bH)\otimes \be_i^\top)\calD \calS (\bI_T\otimes \bX)
	\bGamma^\top
	(\bI_T\otimes \bX^\top) \calD 
	(\be_t \otimes \be_l),\\
	&= ((\be_j^\top \bH)\otimes \be_i^\top) 
	[-\bI_{pT} + \calD\calS(\bI_T\otimes \bX)
	\bGamma^\top
	(\bI_T\otimes \bX^\top)] 
	\calD(\be_t \otimes \be_l),\\
	\tDelta_{ij}^{lt} 
	&= - 
	((\be_i^\top\bF)\otimes \be_j^\top)
	\bGamma^\top 
	(\bI_T\otimes \bX^\top) \calD 
	(\be_t \otimes \be_l).
\end{align*}
Using $\sum_{i,j,t,k} [(\be_j^\top \otimes \be_i^\top) \bA (\be_t \otimes \be_l)]^2 = \fnorm{\bA}^2$ for $\bA\in \R^{np\times nT}$ and the triangle inequality, we have 
\begin{align*}
	&\sum_{i=1}^n \sum_{j=1}^p \fnorm{\pdv{\tbF}{x_{ij}}}^2\\
	=~& \sum_{i=1}^n \sum_{j=1}^p \sum_{l=1}^n \sum_{t=1}^T (\tD_{ij}^{lt} + \tDelta_{ij}^{lt})^2\\
	\lesssim~&
	\fnorm{(\bH\otimes \bI_n) 
	[-\bI_{pT} + \calD\calS(\bI_T\otimes \bX)
	\bGamma^\top
	(\bI_T\otimes \bX^\top)]\calD }^2
	+ \fnorm{(\bF\otimes \bI_p) \bGamma^\top (\bI_T\otimes \bX^\top) \calD}^2\\
	\lesssim~& n \fnorm{\bH}^2 (1 + \opnorm{\bX}^4\opnorm{\bGamma}^2) + \fnorm{\bF}^2 \opnorm{\bGamma}^2 \opnorm{\bX}^2,
\end{align*}
where the last inequality uses $\opnorm{\calD}\le 1$ and $\opnorm{\calS}\le1$. 
Taking the conditional expectation on both sides given $\bep$, the desired moment bound follows from the bound of $\bGamma$ in \Cref{lem:Gamma-opnorm} and 
the moment bounds of $\bX,\bH,\bF$ in \Cref{lem:HF-moment}.
\end{proof}

\subsection{Proof of \Cref{lem:Upsilon-bound}}
\label{proof-lem:Upsilon-bound}

\paragraph{Bound of $\bUpsilon_1$.}
By the expression for the matrix $\bUpsilon_1\in \R^{p\times T}$ obtained in \Cref{proof-def:identity-sum-pdv},
\begin{align*}
	\bUpsilon_1
	&=\sum_{i=1}^n  
	((\be_i^\top\bF)\otimes \bI_p)
	\bGamma^\top
	(\bI_T \otimes \bX^\top)\calD \calS (\bI_T \otimes \be_i)\\
	&=\sum_{i=1}^n \sum_{t=1}^T
	((\be_i^\top\bF\be_t\be_t^\top)\otimes \bI_p)
	\bGamma^\top
	(\bI_T \otimes \bX^\top)\calD \calS (\bI_T \otimes \be_i)\\
	&= \sum_{i=1}^n \sum_{t=1}^T
	(\be_t^\top\otimes \bI_p)
	\bGamma^\top
	(\bI_T \otimes \bX^\top)\calD \calS (\bI_T \otimes \be_i)\be_i^\top\bF\be_t\\
	&= \sum_{t=1}^T
	(\be_t^\top\otimes \bI_p)
	\bGamma^\top
	(\bI_T \otimes \bX^\top)\calD \calS (\bI_T \otimes(\bF\be_t)).
\end{align*}
By the triangle inequality and $\opnorm{\calD}\vee \opnorm{\calS}\le 1$, 
\(\opnorm{\bUpsilon_1} 
\le T \opnorm{\bGamma} \opnorm{\bX} \fnorm{\bF}\). 
Using the bound of $\opnorm{\bGamma}$ in \Cref{lem:Gamma-opnorm}, the moment bound of $\opnorm{\bX},\fnorm{\bF}$ in \Cref{lem:HF-moment} gives 
\begin{align*}
	\E [\opnorm{\bUpsilon_1}^2\mid\bep] \le C(T, \gamma, \eta_{\max}, c_0) (\delta^2 + \norm{\bb^*}^2).
\end{align*}

\paragraph{Bound of $\bUpsilon_2$.}
By the expression for the matrix $\bUpsilon_2\in\R^{n\times T}$ obtained in
\Cref{proof-def:identity-sum-pdv}, we have
\begin{align*}
	\bUpsilon_2 
	&= \sum_{j=1}^p 
	((\be_j^\top \bH)\otimes \bI_n) 
	\calD \calS 
	(\bI_T\otimes \bX)
	\bGamma^\top (\bI_T\otimes \be_j)\\
	&= \sum_{j=1}^p\sum_{t=1}^T 
	((\be_j^\top \bH\be_t\be_t^\top)\otimes \bI_n)
	\calD \calS 
	(\bI_T\otimes \bX)
	\bGamma^\top (\bI_T\otimes \be_j)\\
	&= \sum_{j=1}^p\sum_{t=1}^T 
	(\be_t^\top \otimes \bI_n)
	\calD \calS 
	(\bI_T\otimes \bX)
	\bGamma^\top (\bI_T\otimes \be_j)
	\be_j^\top \bH\be_t\\
	&= \sum_{t=1}^T 
	(\be_t^\top \otimes \bI_n)
	\calD \calS 
	(\bI_T\otimes \bX)
	\bGamma^\top (\bI_T\otimes (\bH\be_t)).
\end{align*}
By the triangle inequality and $\opnorm{\calD}\vee \opnorm{\calS}\le 1$,
\(
\opnorm{\bUpsilon_2} 
\le T \opnorm{\bGamma} \opnorm{\bX} \fnorm{\bH}. 
\)
Similar to the moment bound of $\opnorm{\bUpsilon_1}$, we obtain 
\begin{align*}
	\E [\opnorm{\bUpsilon_2}^2\mid\bep] \le C(T, \gamma, \eta_{\max}, c_0) n^{-1} (\delta^2 + \norm{\bb^*}^2).
\end{align*}

\paragraph{Bound of $\bUpsilon_3$.}
By the expression for the matrix $\bUpsilon_3\in\R^{T\times T}$ obtained in
\Cref{proof-def:identity-sum-pdv}, we have
$\bUpsilon_3 
	= (\bUpsilon_1^\top \bH + \bF^\top \bUpsilon_2)$.
It directly follows that
\begin{align*}
	\opnorm{\bUpsilon_3} 
	\le \opnorm{\bUpsilon_1} \fnorm{\bH} + \fnorm{\bF} \opnorm{\bUpsilon_2}.
\end{align*}
Using the triangle inequality and the moment bounds of $\fnorm{\bH}, \fnorm{\bF}$ in \Cref{lem:HF-moment} and the moment bounds of $\opnorm{\bUpsilon_1},\opnorm{\bUpsilon_2}$ we just obtained, we have 
\begin{align*}
	\E[\opnorm{\bUpsilon_3}^2\mid\bep]
	\le C(T, \gamma, \eta_{\max}, c_0) (\delta^2 + \norm{\bb^*}^2).
\end{align*}

\paragraph{Bound of $\bUpsilon_4$.}
By the expression for the matrix $\bUpsilon_4\in\R^{T\times T}$ obtained in
\Cref{proof-def:identity-sum-pdv}, we have
\begin{align*}
	\bUpsilon_4 
	= \bH^\top \bX^\top \bUpsilon_2 - \tbUpsilon_1,
\end{align*}
where 
$\tbUpsilon_1=
\sum_{i} (\bI_T \otimes \be_i^\top \bX ) \bGamma((\bF^\top\be_i) \otimes\bI_p)\bH$. 
We can rewrite $\tbUpsilon_1$ as
\begin{align*}
	\tbUpsilon_1 
	&= \sum_{i=1}^n (\bI_T \otimes \be_i^\top \bX ) \bGamma((\bF^\top\be_i) \otimes\bI_p)\bH\\
	&= \sum_{i=1}^n \sum_{t=1}^T (\bI_T \otimes \be_i^\top \bX ) \bGamma((\be_t\be_t^\top\bF^\top\be_i) \otimes\bI_p)\bH\\
	&= \sum_{i=1}^n \sum_{t=1}^T (\bI_T \otimes (\be_t^\top\bF^\top\be_i\be_i^\top \bX) ) \bGamma(\be_t \otimes\bI_p)\bH\\
	&= \sum_{t=1}^T (\bI_T \otimes (\be_t^\top\bF^\top\bX) ) \bGamma(\be_t \otimes\bI_p)\bH.
\end{align*}
We have by the triangle inequality,
\begin{align*}
	\opnorm{\tbUpsilon_1} 
	\le T \opnorm{\bGamma} \opnorm{\bX} \fnorm{\bF} \fnorm{\bH}.
\end{align*}
By the triangle inequality and the upper bound of 
$\opnorm{\bUpsilon_2}$, we have 
\begin{align*}
	\opnorm{\bUpsilon_4} 
	\le \opnorm{\bUpsilon_2} \opnorm{\bX} \fnorm{\bH} + \opnorm{\tbUpsilon_1}
	\le T \opnorm{\bGamma} \opnorm{\bX}^2 \fnorm{\bH}^2 + T \opnorm{\bGamma} \opnorm{\bX} \fnorm{\bF} \fnorm{\bH}.
\end{align*}
Squaring both sides and taking conditional expectation, we have 
\begin{align*}
	\E[\opnorm{\bUpsilon_4}^2\mid\bep]
	\le C(T, \gamma, \eta_{\max}, c_0) (\delta^2 + \norm{\bb^*}^2),
\end{align*}
thanks to the upper bound of $\opnorm{\bGamma}$ in \Cref{lem:Gamma-opnorm} and the moment bounds of $\opnorm{\bX}, \fnorm{\bF}, \fnorm{\bH}$ in \Cref{lem:HF-moment}.

\paragraph{Bound of $\bUpsilon_5$.}
By the expression for the matrix $\bUpsilon_5\in\R^{T\times T}$ 
obtained in
\Cref{proof-def:identity-sum-pdv}, we have
$\bUpsilon_5 
	= \bUpsilon_1^\top \bX^\top \tbF - \tbUpsilon_2,$ where 
\begin{align*}
	\tbUpsilon_2 
	&= - \sum_{t=1}^T \bF^\top \bD_t \bX\bH\be_t 
	+
	\sum_{l=1}^n (\be_l^\top \bX \bH \otimes \bF^\top) \calD \calS (\bI_T\otimes \bX)
	\bGamma^\top
	(\bI_T\otimes \bX^\top)
	\calD 
	(\bI_T\otimes \be_l).
\end{align*}
By the triangle inequality, 
\begin{align*}
	\opnorm{\tbUpsilon_2} 
	\le T \opnorm{\bX} \fnorm{\bF} \fnorm{\bH} 
	+ \opnorm{\bX}^3 \opnorm{\bGamma} \fnorm{\bF}\fnorm{\bH}.
\end{align*}
Using the moment bounds of $\opnorm{\bUpsilon_1},\opnorm{\bX}, \fnorm{\bF}, \fnorm{\tbF},\fnorm{\bH}$, we have 
\begin{align*}
	\E[\opnorm{\bUpsilon_5}^2\mid\bep]
	\le n^2 C(T, \gamma, \eta_{\max}, c_0) (\delta^2 + \norm{\bb^*}^2).
\end{align*}

\subsection{Proof of \Cref{lem:Theta-bound}}
\label{proof-lem:Theta-bound}

We first state three useful lemmas. 

\begin{lemma}[Adopted from Lemma E.10 of \cite{tan2022noise}]
	\label{lem:steinX}
	Let $\bU, \bV:\R^{n\times p} \to \R^{n\times T}$ be two locally Lipschitz functions of $\bZ$ with \iid $\calN(0,1)$ entries, then 
	\begin{align*}
		&\E\Big[\norm[\Big]{\bU^\top \bZ \bV - 
			\sum_{j=1}^p\sum_{i=1}^n
			\frac{\partial}{\partial z_{ij} }\Bigl(\bU^\top \be_i \be_j^\top \bV \Bigr)
		}_{\rm F}^2\Big]\\
		\le~& \E \fnorm*{\bU}^2 \fnorm*{\bV}^2+ \E \sum_{ij}\Big[
		2\fnorm*{\bV}^2\fnorm*{ \frac{\partial \bU}{\partial z_{ij}} }^2
		+ 2\fnorm*{\bU}^2\fnorm*{ \frac{\partial \bV}{\partial z_{ij}} }^2\Big].
	\end{align*}
\end{lemma}

\begin{lemma}[Adopted from Lemma F.5 of \cite{bellec2024uncertainty}] 
	\label{lem:Chi2type}
    Let $\bU,\bV: \R^{n\times p} \to \R^{n\times T}$
    be two locally Lipschitz 
    functions of $\bZ$ with \iid $\mathsf{N}(0,1)$ entries.
    Provided the following expectations are finite, we have 
    \begin{align*}
        &\E\Bigl[
        \fnorm[\Big]{
            p \bU^\top \bV - \sum_{j=1}^p 
            \Bigl(\sum_{i=1}^n \partial_{ij} \bU^\top \be_i - \bU^\top \bZ \be_j\Bigr)
            \Bigl(\sum_{i=1}^n \partial_{ij} \be_i^\top \bV  - \be_j^\top \bZ^\top \bV\Bigr)
        }
        \Bigr] 
        \\
        \le~&
        (1 + 2\sqrt{p}) 
        \bigl( 
            \E [\fnorm{\bU}^4]^{1/2}
            + \E [\fnorm{\bV}^4]^{1/2}
            + \E [\norm{\bU}^4_{\partial}]^{1/2}
            + \E [\norm{\bV}^4_{\partial}]^{1/2}
            \bigr),
    \end{align*}
    where $ \partial_{ij} \bU = \partial \bU /\partial z_{ij}$ 
    and $\|\bU\|_{\partial} = (\sum_{i=1}^n\sum_{j=1}^p\|\partial_{ij} \bU\|_{\rm F}^2)^{1/2}$.
\end{lemma}

We will use the above two lemmas, conditionally on $\bep$,  to bound the conditional moments of $\bTheta_1, \bTheta_2, \bTheta_3, \bTheta_4$ and $\bTheta_5$ given $\bep$. 

\paragraph{Bound of $\bTheta_1$.}
By the definition of $\bTheta_1$, we have
\begin{align*}
	\bF^\top \bX \bH - \sum_{i,j} \pdv{\bF^\top \be_i\be_j^\top \bH}{x_{ij}}
	&= \bF^\top \bX \bH + \tbK \bH^\top \bH - \bF^\top \bF \bW^\top + \bUpsilon_3 &&\mbox{by \eqref{eq:3}}\\
	&= \bTheta_1 + \bUpsilon_3.
\end{align*}
Applying \Cref{lem:steinX} conditionally on $\bep$ to 
$(\bZ, \bU, \bV) = (\bX, \bF, \bH)$ gives 
\begin{align*}
	\E[\fnorm{\bTheta_1}^2 \mid\bep]
	\lesssim~& 
	\E\Big[\norm[\Big]{\bF^\top \bX \bH - \sum_{i,j} \pdv{\bF^\top \be_i\be_j^\top \bH}{x_{ij}}}^2\mid\bep\Big] + \E[\fnorm{\bUpsilon_3}^2\mid\bep]\\
	\lesssim~& \E [\fnorm{\bF}^2 \fnorm{\bH}^2\mid\bep]+ \E \sum_{ij}\Big[
	\fnorm{\bH}^2\fnorm{ \pdv{\bF}{x_{ij}} }^2
	+ \fnorm{\bF}^2\fnorm{ \pdv{\bH}{x_{ij}} }^2 \mid\bep\Big] + \E[\fnorm{\bUpsilon_3}^2 \mid\bep]\\
	\le~& n C(T, \gamma, \eta_{\max}, c_0) (\delta^2 + \norm{\bb^*}^2),
\end{align*}
where the last line uses the moment bounds of $\fnorm{\bF}, \fnorm{\bH}$ in \Cref{lem:HF-moment}, the moment bounds of $\fnorm{\pdv{\bH}{x_{ij}}}$, $\fnorm{\pdv{\bF}{x_{ij}}}$ in \Cref{lem:HF-derivative-moment}, the moment bound of $\opnorm{\bUpsilon_3}$ in \Cref{lem:Upsilon-bound}. 

\paragraph{Bound of $\bTheta_2$.}
By the definition of $\bTheta_2$, we have by \eqref{eq:4}
\begin{align*}
	\bF^\top \bX \bX^\top\tbF - \sum_{i,j} \pdv{\bF^\top  \be_i\be_j^\top \bX^\top\tbF}{x_{ij}}
	&= \bF^\top \bX \bX^\top\tbF + \tbK\bH^\top \bX^\top \tbF - 
	p \bF^\top\tbF + \bF^\top \bF \hbA^\top + \bUpsilon_5 \\
	&= n\bTheta_2 + \bUpsilon_5.
\end{align*}
Applying \Cref{lem:steinX} conditionally on $\bep$ to 
$(\bZ, \bU, \bV) = (\bX, \bF, \bX^\top\tbF)$ gives 
\begin{align*}
	n^2 \E[\fnorm{\bTheta_2}^2\mid\bep]
	\lesssim~& 
	\E\Big[\norm[\Big]{\bF^\top \bX \bX^\top\tbF - \sum_{i,j} \pdv{\bF^\top  \be_i\be_j^\top \bX^\top\tbF}{x_{ij}}}^2\mid\bep\Big] + \E[\fnorm{\bUpsilon_5}^2\mid\bep]\\
	\lesssim~& \E [\fnorm{\bF}^2 \fnorm{\bX^\top\tbF}^2\mid\bep]+ \E \sum_{ij}\Big[
	\fnorm{\bX^\top\tbF}^2\fnorm{ \pdv{\bF}{x_{ij}} }^2
	+ \fnorm{\bF}^2\fnorm{ \pdv{\bX^\top\tbF}{x_{ij}} }^2 \mid\bep\Big] + \E[\fnorm{\bUpsilon_5}^2\mid\bep]\\
	\lesssim~& \E [\fnorm{\tbF}^4 \opnorm{\bX}^2\mid\bep]+ \E \Big[(1+\opnorm{\bX}^2)
	\fnorm{\tbF}^2\sum_{ij}\fnorm{ \pdv{\tbF}{x_{ij}} }^2
	\mid\bep\Big] + T \E[\opnorm{\bUpsilon_5}^2\mid\bep]\\
	\le~& n^3 C(T, \gamma, \eta_{\max}, c_0) (\delta^2 + \norm{\bb^*}^2).
\end{align*}
Here, the penultimate line uses $\fnorm{\bF}\le\fnorm{\tbF}$,
$\fnorm{\pdv{\bF}{x_{ij}}}\le \fnorm{\pdv{\tbF}{x_{ij}}}$, 
and $\fnorm{\pdv{\bX^\top\tbF}{x_{ij}}} = \fnorm{\be_j^\top \be_i^\top \pdv{\tbF}{x_{ij}} + \bX^\top \pdv{\tbF}{x_{ij}}}
\le \fnorm{\pdv{\tbF}{x_{ij}}} + \opnorm{\bX} \fnorm{\pdv{\tbF}{x_{ij}}}$. 
The last line uses the moment bounds of $\opnorm{\bX},\fnorm{\tbF}$ in \Cref{lem:HF-moment}, the moment bound of $\fnorm{\pdv{\tbF}{x_{ij}}}$ in \Cref{lem:HF-derivative-moment}, the moment bound of $\opnorm{\bUpsilon_5}$ in \Cref{lem:Upsilon-bound}.

\paragraph{Bound of $\bTheta_3$.}
By the definition of $\bTheta_3$, we have
\begin{align*}
	&\bH^\top \bX^\top \bX \bH - \sum_{i,j} \pdv{\bH^\top \bX^\top \be_i\be_j^\top \bH}{x_{ij}}\\
	=~& \bH^\top \bX^\top \bX \bH - (n\bI_T -\tbA) \bH^\top \bH - 
	\bH^\top \bX^\top \bF \bW^\top + \bUpsilon_4 &&\mbox{by \eqref{eq:4}}\\
	=~& \bTheta_3 + \bUpsilon_4.
\end{align*}
Applying \Cref{lem:steinX} conditionally on $\bep$ to 
$(\bZ, \bU, \bV) = (\bX, \bX\bH, \bH)$ gives 
\begin{align*}
	\E[\fnorm{\bTheta_3}^2\mid \bep]
	\lesssim~& 
	\E\Big[\norm[\Big]{\bH^\top \bX^\top \bX \bH - \sum_{i,j} \pdv{\bH^\top \bX^\top \be_i\be_j^\top \bH}{x_{ij}}}^2 \mid \bep\Big] + \E[\fnorm{\bUpsilon_4}^2\mid \bep]\\
	\lesssim~& \E [\fnorm{\bX\bH}^2 \fnorm{\bH}^2\mid \bep]+ \E \sum_{ij}\Big[
	\fnorm{\bX\bH}^2\fnorm{ \pdv{\bH}{x_{ij}} }^2
	+ \fnorm{\bH}^2\fnorm{ \pdv{\bX\bH}{x_{ij}} }^2\mid \bep\Big] + \E[\fnorm{\bUpsilon_4}^2\mid \bep]\\
	\lesssim~& \E [\opnorm{\bX}^2 \fnorm{\bH}^4\mid \bep]+ \E \Big[(1 + \opnorm{\bX}^2)\fnorm{\bH}^2\sum_{ij}
	\fnorm{ \pdv{\bH}{x_{ij}} }^2 \mid \bep\Bigr] + 
	T \E[\opnorm{\bUpsilon_4}^2\mid \bep]\\
	\le & n C(T, \gamma, \eta_{\max}, c_0) (\delta^2 + \norm{\bb^*}^2).
\end{align*}
Here, the penultimate line uses 
$
\fnorm{\pdv{\bX\bH}{x_{ij}}}
= \fnorm{\be_i\be_j^\top \pdv{\bH}{x_{ij}} + \bX \pdv{\bH}{x_{ij}}}
\le (1 + \opnorm{\bX}) \fnorm{\pdv{\bH}{x_{ij}}},
$
and the last line uses the moment bounds of $\opnorm{\bX}, \fnorm{\bH}$ in \Cref{lem:HF-moment}, the moment bounds of $\fnorm{\pdv{\bH}{x_{ij}}}$ in \Cref{lem:HF-derivative-moment}, the moment bound of $\opnorm{\bUpsilon_4}$ in \Cref{lem:Upsilon-bound}.

\paragraph{Bound of $\bTheta_4$.}
By definition, we have 
\begin{align*}
	\sum_{i=1}^n \pdv{\bF^\top \be_i}{x_{ij}} 
	= - (\tbK \bH^\top + \bUpsilon_1^\top) \be_j 
	\text{ and }
	\sum_{i=1}^n \pdv{\tbF^\top \be_i}{x_{ij}} 
	= - (\hbK \bH^\top + \tbUpsilon_1^\top) \be_j.
\end{align*}
Using $\sum_{j=1}^p \be_j \be_j^\top = \bI_p$, we find
\begin{align*}
	&\sum_{j=1}^p 
	\bigl(\sum_{i=1}^n \pdv{\bF^\top \be_i}{x_{ij}} - \bF^\top \bX \be_j\bigr)\bigl(\sum_{i=1}^n \pdv{\tbF^\top \be_i}{x_{ij}} - \tbF^\top \bX \be_j\bigr)^\top\\
	=~& (\tbK \bH^\top + \bF^\top \bX + \bUpsilon_1^\top)
	(\hbK \bH^\top + \tbF^\top \bX + \tbUpsilon_1^\top).
\end{align*}
This further implies
\begin{align*}
	&p \bF^\top \tbF - \sum_{j=1}^p 
	\bigl(\sum_{i=1}^n \pdv{\bF^\top \be_i}{x_{ij}} - \bF^\top \bX \be_j\bigr)\bigl(\sum_{i=1}^n \pdv{\tbF^\top \be_i}{x_{ij}} - \tbF^\top \bX \be_j\bigr)^\top\\
	=~& n \bTheta_4 +  
	\underbrace{\bUpsilon_1^\top
	(\hbK \bH^\top + \tbF^\top \bX) 
	+(\tbK \bH^\top + \bF^\top \bX) \tbUpsilon_1^\top + \bUpsilon_1^\top \tbUpsilon_1^\top}_{\tbUpsilon_4}.
\end{align*}
Applying \Cref{lem:Chi2type} conditionally on $\bep$ to $(\bZ, \bU, \bV) = (\bX, \bF, \tbF)$, 
\begin{align*}
	n \E[\fnorm{\bTheta_4}\mid\bep]
	\lesssim~& 
	\E \Bigl[\fnorm[\Big]{p \bF^\top \tbF - \sum_{j=1}^p 
	\bigl(\sum_{i=1}^n \pdv{\bF^\top \be_i}{x_{ij}} - \bF^\top \bX \be_j\bigr)\bigl(\sum_{i=1}^n \pdv{\tbF^\top \be_i}{x_{ij}} - \tbF^\top \bX \be_j\bigr)^\top}\mid\bep\Bigr] 
	+ \E [\fnorm{\tbUpsilon_4}\mid\bep]\\
	\lesssim~& 
	(1 + 2\sqrt{p}) 
        \bigl( 
            \E [\fnorm{\bF}^4\mid\bep]^{1/2}
            + \E [\fnorm{\tbF}^4\mid\bep]^{1/2}
            + \E [\norm{\bF}^4_{\partial}\mid\bep]^{1/2}
            + \E [\norm{\tbF}^4_{\partial}\mid\bep]^{1/2}
            \bigr) 
	+ T\E[\opnorm{\tbUpsilon_4}^2\mid\bep]\\
	\le~ & n^{3/2} C(T, \gamma, \eta_{\max}, c_0) (\delta^2 + \norm{\bb^*}^2).
\end{align*}
Here, the last inequality uses the moment bounds of $\fnorm{\bF}, \fnorm{\tbF}$ in \Cref{lem:HF-moment}, and the moment bounds of 
$\norm{\bF}^2_{\partial}$, $\norm{\tbF}^2_{\partial}$ in \Cref{lem:HF-derivative-moment}, and the moment bound of $\opnorm{\tbUpsilon_4}$ in \Cref{lem:Upsilon-bound}. 
\paragraph{Bound of $\bTheta_5$.}
By the identity \eqref{eq:2}, we have 
\begin{align*}
	&\sum_{i=1}^n 
	\bigl(\sum_{j=1}^p \pdv{\bH^\top \be_j}{x_{ij}} - \bH^\top \bX^\top \be_i\bigr)
	\bigl(\sum_{j=1}^p \pdv{\bH^\top \be_j}{x_{ij}} - \bH^\top \bX^\top \be_i\bigr)^\top\\
	=~& (\bW\bF^\top -\bH^\top\bX^\top - \bUpsilon_2^\top)(\bW\bF^\top -\bH^\top\bX^\top - \bUpsilon_2^\top)^\top.
\end{align*}
By the definition of $\bTheta_5$, we have 
\begin{align*}
	&~n \bH^\top \bH - \sum_{i=1}^n 
	\bigl(\sum_{j=1}^p \pdv{\bH^\top \be_j}{x_{ij}} - \bH^\top \bX^\top \be_i\bigr)
	\bigl(\sum_{j=1}^p \pdv{\bH^\top \be_j}{x_{ij}} - \bH^\top \bX^\top \be_i\bigr)^\top\\
	=~& \bTheta_5 + 
	\underbrace{\bUpsilon_2^\top (\bW\bF^\top -\bH^\top\bX^\top)^\top
	+ (\bW\bF^\top -\bH^\top\bX^\top) \bUpsilon_2
	- \bUpsilon_2^\top \bUpsilon_2}_{\tbUpsilon_5}.
\end{align*}
Here 
$\bUpsilon_2 = \sum_{j} 
((\be_j^\top \bH)\otimes \bI_n)
\calD \calS
(\bI_T\otimes \bX)
\bGamma^\top 
(\bI_T \otimes \be_j)$. 

Applying \Cref{lem:Chi2type} conditionally on $\bep$ to 
$(\bZ, \bU, \bV) = (\bX^\top, \bH, \bH)$ (\ie, consider the mapping from $\R^{p\times n}$ to $\R^{p\times T}$: $\bX^\top \mapsto \bH$) gives 
\begin{align*}
	&\E[\fnorm{\bTheta_5}\mid\bep]\\
	\le~& 
	\E \Bigl[\fnorm[\Big]{n \bH^\top \bH - \sum_{i=1}^n 
	\bigl(\sum_{j=1}^p \pdv{\bH^\top \be_j}{x_{ij}} - \bH^\top \bX^\top \be_i\bigr)
	\bigl(\sum_{j=1}^p \pdv{\bH^\top \be_j}{x_{ij}} - \bH^\top \bX^\top \be_i\bigr)^\top}\mid\bep\Bigr] 
	+ \E [\fnorm{\tbUpsilon_5}\mid\bep]\\
	\lesssim~& 
	(1 + 2\sqrt{p}) 
		\bigl( 
			\E [\fnorm{\bH}^4\mid\bep]^{1/2}
			+ \E [\norm{\bH}^4_{\partial}\mid\bep]^{1/2}
			\bigr) 
	+ T\E[\opnorm{\tbUpsilon_5}^2\mid\bep]\\
	\le~ & n^{1/2} C(T, \gamma, \eta_{\max}, c_0) (\delta^2 + \norm{\bb^*}^2).
\end{align*}
Here, the last line use the moment bounds of $\fnorm{\bH}$ in \Cref{lem:HF-moment}, the moment bound of $\norm{\bH}^2_{\partial}$ in \Cref{lem:HF-derivative-moment}, and the moment bound of $\opnorm{\tbUpsilon_5}$ in \Cref{lem:Upsilon-bound}.
\paragraph{Bound of $\bTheta_6$.}

Using \eqref{eq:2}, we have 
\begin{align*}
	\sum_{i=1}^n\sum_{j=1}^p
	\pdv{\bE^\top\be_i\be_j^\top\bH}{x_{ij}}
	= \bE^\top (\bF\bW^\top - \bUpsilon_2).
\end{align*}
It follows that 
\begin{align*}
	\bE^\top \bX \bH - \sum_{i=1}^n\sum_{j=1}^p \pdv{\bE^\top\be_i\be_j^\top\bH}{x_{ij}}
	=~& \bE^\top \bX \bH - \bE^\top (\bF\bW^\top - \bUpsilon_2)\\
	=~& \fnorm{\bE} \bTheta_6 + \bE^\top \bUpsilon_2.
\end{align*}
Thus, using $\tbE = \bE/\fnorm{\bE}$, we have 
$$
\bTheta_6 = \tbE^\top \bX \bH - \sum_{i=1}^n\sum_{j=1}^p \pdv{\tbE^\top\be_i\be_j^\top\bH}{x_{ij}} - \tbE^\top \bUpsilon_2.\\
$$
Applying \Cref{lem:steinX}
conditionally on $\bep$ to $(\bZ, \bU, \bV) = (\bX, \tbE,\bH)$gives 
\begin{align*}
	\E[\fnorm{\bTheta_6}^2\mid\bep]
	&\lesssim
	\E\Big[\norm[\Big]{\tbE^\top \bX \bH - \sum_{i=1}^n\sum_{j=1}^p \pdv{\tbE^\top\be_i\be_j^\top\bH}{x_{ij}}}^2\mid\bep\Big] + \E[\fnorm{\tbE^\top \bUpsilon_2}^2\mid\bep]\\
	&\lesssim \E [\fnorm{\tbE}^2 \fnorm{\bH}^2\mid\bep] 
	+ \E\Bigl[\fnorm{\tbE}^2\fnorm{ \pdv{\bH}{x_{ij}} }^2\mid\bep\Bigr] + \E[\fnorm{\tbE^\top \bUpsilon_2}^2\mid\bep]\\
	& \le \E [\fnorm{\bH}^2\mid\bep] 
	+ \E\Bigl[\fnorm{ \pdv{\bH}{x_{ij}} }^2\mid\bep\Bigr] + \E[\opnorm{\bUpsilon_2}^2\mid\bep]\\
	&\le C(T, \gamma, \eta_{\max}, c_0) (\delta^2 + \norm{\bb^*}^2).
\end{align*}
Here, the last line uses the moment bound of $\fnorm{\bH}$ in \Cref{lem:HF-moment}, the moment bounds of $\fnorm{\pdv{\bH}{x_{ij}}}$ in \Cref{lem:HF-derivative-moment}, and the moment bound of $\opnorm{\bUpsilon_2}$ in \Cref{lem:Upsilon-bound}.
This finishes the proof of \Cref{lem:Theta-bound}.


\newpage
\section*{NeurIPS Paper Checklist}

\begin{enumerate}

\item {\bf Claims}
    \item[] Question: Do the main claims made in the abstract and introduction accurately reflect the paper's contributions and scope?
    \item[] Answer: \answerYes{} 
    \item[] Justification: The abstract and introduction clearly state the claims made in the paper.

\item {\bf Limitations}
    \item[] Question: Does the paper discuss the limitations of the work performed by the authors?
    \item[] Answer: \answerYes{} 
    \item[] Justification: Our results only holds under several assumptions discussed in the paper.

\item {\bf Theory Assumptions and Proofs}
    \item[] Question: For each theoretical result, does the paper provide the full set of assumptions and a complete (and correct) proof?
    \item[] Answer: \answerYes{} 
    \item[] Justification: All the proofs are provided in the supplemental material.

    \item {\bf Experimental Result Reproducibility}
    \item[] Question: Does the paper fully disclose all the information needed to reproduce the main experimental results of the paper to the extent that it affects the main claims and/or conclusions of the paper (regardless of whether the code and data are provided or not)?
    \item[] Answer: \answerYes{} 
    \item[] Justification: All the code needed to reproduce the results is in the supplemental material.

\item {\bf Open access to data and code}
    \item[] Question: Does the paper provide open access to the data and code, with sufficient instructions to faithfully reproduce the main experimental results, as described in supplemental material?
    \item[] Answer: \answerYes{} 
    \item[] Justification: All the code are provided in the supplemental material.

\item {\bf Experimental Setting/Details}
    \item[] Question: Does the paper specify all the training and test details (e.g., data splits, hyperparameters, how they were chosen, type of optimizer, etc.) necessary to understand the results?
    \item[] Answer: \answerYes{} 
    \item[] Justification: The choices of step size for GD and SGD are provided in Section 4.

\item {\bf Experiment Statistical Significance}
    \item[] Question: Does the paper report error bars suitably and correctly defined or other appropriate information about the statistical significance of the experiments?
    \item[] Answer: \answerYes{} 
    \item[] Justification: All the figures include the 2-standard error bars.

\item {\bf Experiments Compute Resources}
    \item[] Question: For each experiment, does the paper provide sufficient information on the computer resources (type of compute workers, memory, time of execution) needed to reproduce the experiments?
    \item[] Answer: \answerYes{} 
    \item[] Justification: It is provided in the README file of the supplemental material.
    
\item {\bf Code Of Ethics}
    \item[] Question: Does the research conducted in the paper conform, in every respect, with the NeurIPS Code of Ethics \url{https://neurips.cc/public/EthicsGuidelines}?
    \item[] Answer: \answerYes{} 
    \item[] Justification: 

\item {\bf Broader Impacts}
    \item[] Question: Does the paper discuss both potential positive societal impacts and negative societal impacts of the work performed?
    \item[] Answer: \answerNA{} 
    \item[] Justification: Given the theoretical nature of our work, it does not have any negative societal impact.
    
\item {\bf Safeguards}
    \item[] Question: Does the paper describe safeguards that have been put in place for responsible release of data or models that have a high risk for misuse (e.g., pretrained language models, image generators, or scraped datasets)?
    \item[] Answer: \answerNA{} 
    \item[] Justification: The paper is of the theoretical nature, it does not invent any new models or datasets.

\item {\bf Licenses for existing assets}
    \item[] Question: Are the creators or original owners of assets (e.g., code, data, models), used in the paper, properly credited and are the license and terms of use explicitly mentioned and properly respected?
    \item[] Answer: \answerNA{} 
    \item[] Justification: We write all the code by ourselves, and did not use other code or data. 

\item {\bf New Assets}
    \item[] Question: Are new assets introduced in the paper well documented and is the documentation provided alongside the assets?
    \item[] Answer: \answerNA{} 
    \item[] Justification: 

\item {\bf Crowdsourcing and Research with Human Subjects}
    \item[] Question: For crowdsourcing experiments and research with human subjects, does the paper include the full text of instructions given to participants and screenshots, if applicable, as well as details about compensation (if any)? 
    \item[] Answer: \answerNA{} 
    \item[] Justification: 

\item {\bf Institutional Review Board (IRB) Approvals or Equivalent for Research with Human Subjects}
    \item[] Question: Does the paper describe potential risks incurred by study participants, whether such risks were disclosed to the subjects, and whether Institutional Review Board (IRB) approvals (or an equivalent approval/review based on the requirements of your country or institution) were obtained?
    \item[] Answer: \answerNA{} 
    \item[] Justification: 

\end{enumerate}

\end{document}

%% file: macros.tex
\usepackage{bm}
\newcommand{\boldone}{\mathbf{1}}
\newcommand{\boldzero}{\mathbf{0}}
\usepackage{mathtools}
\DeclarePairedDelimiter\norm{\lVert}{\rVert}
\DeclarePairedDelimiter\fnorm{\lVert}{\rVert_{\rm{F}}}
\DeclarePairedDelimiter\opnorm{\lVert}{\rVert_{\rm{op}}}

\def\cbK{\check{\mathbold{K}}}
\def\cbTheta{\check{\mathbold{\Theta}}}

\def\bUpsilon{\mathbold{\Upsilon}}
\def\tbUpsilon{\widetilde{\mathbold{\Upsilon}}}

\def\tD{{\tilde D}}
\def\tF{{\tilde F}}

\def\tcalD{{\widetilde{\cal D}}}

\def\tDelta{{\tilde \Delta}}
\def\defas{\stackrel{\text{\tiny def}}{=}}

\def\sign{{\rm sign}}

\def\argmin{\mathop{\rm arg\, min}}

\DeclareMathOperator{\soft}{soft}

\DeclareMathOperator{\trace}{Tr}

\let\vec\relax
\DeclareMathOperator{\vec}{\mathbf{vec}}

\usepackage{xspace} 
\def\iid{i.i.d.\@\xspace}

\def\ie{i.e.\@\xspace}

\def\R{{\mathbb{R}}}
\def\E{{\mathbb{E}}}

\def\P{{\mathbb{P}}}

\def\mathbold{\boldsymbol} 
\def\cal{\mathcal} 

\def\ba{\mathbold{a}}

\def\bA{\mathbold{A}}
\def\hbA{\smash{\widehat{\mathbf A}}}\def\tbA{{\widetilde{\mathbf A}}}

\def\bb{\mathbold{b}}
\def\hbb{{\widehat{\bb}}}

\def\bD{\mathbold{D}}\def\calD{{\cal D}}
\def\tbD{{\widetilde{\bD}}}

\def\bE{\mathbold{E}}
\def\tbE{{\widetilde{\bE}}}

\def\bF{\mathbold{F}}
\def\tbF{{\widetilde{\bF}}}

\def\bG{\mathbold{G}}

\def\bH{\mathbold{H}}

\def\bI{\mathbold{I}}

\def\bK{\mathbold{K}}\def\hbK{{\widehat{\bK}}}\def\tbK{{\widetilde{\bK}}}

\def\bL{\mathbold{L}}\def\hbL{{\widehat{\bL}}}

\def\calM{{\cal M}}

\def\calN{{\cal N}}

\def\bP{\mathbold{P}}

\def\br{\mathbold{r}}

\def\bR{\mathbold{R}}

\def\bS{\mathbold{S}}
\def\calS{{\cal S}}

\def\be{\mathbold{e}}

\def\bu{\mathbold{u}}

\def\bU{\mathbold{U}}

\def\bv{\mathbold{v}}

\def\bV{\mathbold{V}}

\def\bW{\mathbold{W}}\def\tbW{{\widetilde{\bW}}}

\def\bx{\mathbold{x}}

\def\bX{\mathbold{X}}

\def\by{\mathbold{y}}

\def\bZ{\mathbold{Z}}

\def\bGamma{\mathbold{\Gamma}}

\def\ep{\varepsilon}\def\veps{\varepsilon}
\def\bep{ {\mathbold{\ep} }}

\def\btheta{\mathbold{\theta}}

\def\bTheta{\mathbold{\Theta}}

\def\bLambda{\mathbold{\Lambda}}\def\hbLambda{{\widehat{\bLambda}}}

\def\brho{\mathbold{\rho}}

\def\bSigma{\mathbold{\Sigma}}

\def\bphi{\mathbold{\phi}}

\def\bpsi{\mathbold{\psi}}